\documentclass[12pt]{amsart}

\usepackage{mathrsfs, amssymb,amsthm, amsfonts}
\usepackage[all]{xy}
\usepackage{upgreek} 
\usepackage[shortlabels]{enumitem}
\usepackage{setspace}
\usepackage{tabularx,multirow}

\usepackage{hyperref, tikz}

\makeatletter
\@addtoreset{equation}{subsection}
\@addtoreset{figure}{subsection}
\makeatother

\renewcommand\labelenumi{{\rm (\roman{enumi})}}
\renewcommand\theenumi{{\rm (\roman{enumi})}}

\newcommand{\Ga}{{\mathbb G}_{\mathrm{a}}}
\newcommand{\Gm}{{\mathbb G}_{\mathrm{m}}}
\newcommand{\FF}{\mathbb{F}}
\newcommand{\TT}{\mathbb{T}}
\newcommand{\CC}{\mathbb{C}}

\newcommand{\ZZ}{\mathbb{Z}}
\newcommand{\PP}{\mathbb{P}}

\newcommand{\OO}{\mathbb{O}}
\newcommand{\OOO}{{\mathscr{O}}}
\newcommand{\LLL}{{\mathscr{L}}}
\newcommand{\PPP}{{\mathscr{P}}}
\newcommand{\SSS}{{\mathscr{S}}}
\newcommand{\NNN}{{\mathscr{N}}}
\newcommand{\Levi}{{\operatorname{L}}}
\newcommand{\tr}{{\operatorname{tr}}}

\newcommand{\mumu}{{\boldsymbol{\mu}}}

\newcommand{\W}{\mathbb{W}}
\newcommand{\ol}{\ell}
\newcommand{\G}{\operatorname{G}_2}
\newcommand{\pr}{\operatorname{pr}}
\newcommand{\Lie}{\operatorname{Lie}}
\newcommand{\Transl}{\operatorname{Transl}}
\newcommand{\Fl}{\operatorname{Fl}}
\newcommand{\aaa}{\mathrm{a}}
\newcommand{\s}{\mathrm{s}}
\newcommand{\n}{\mathrm{n}}
\newcommand{\x}{\mathbf{x}}
\newcommand{\y}{\mathbf{y}}
\newcommand{\Aut}{\operatorname{Aut}}
\newcommand{\Ad}{\operatorname{Ad}}
\newcommand{\Stab}{\operatorname{Stab}}
\newcommand{\GL}{\operatorname{GL}}
\newcommand{\gl}{\mathfrak{gl}}
\newcommand{\SL}{\operatorname{SL}}
\newcommand{\sll}{\mathfrak{sl}}
\newcommand{\PGL}{\operatorname{PGL}}
\newcommand{\bb}{\operatorname{b}}
\newcommand{\id}{\operatorname{id}}
\newcommand{\diag}{\operatorname{diag}}
\newcommand{\codim}{\operatorname{codim}}
\newcommand{\Ru}{\operatorname{R}_{\operatorname{u}}}
\newcommand{\xref}[1]{{\rm \ref{#1}}}

\newcommand{\z}{\operatorname{z}}
\newcommand{\rk}{\operatorname{rk}}
\newcommand{\Pic}{\operatorname{Pic}}
\newcommand{\Sing}{\operatorname{Sing}}

\newcommand{\red}{\operatorname{red}}

\newcommand{\Gr}{\operatorname{Gr}}
\newcommand{\Cl}{\operatorname{Cl}}
\newcommand{\Cone}{\operatorname{Cone}}

\renewcommand{\hat}[1]{\widehat{#1}}
\renewcommand{\tilde}[1]{\widetilde{#1}}

\renewcommand{\emptyset}{\varnothing}

\swapnumbers
\theoremstyle{plain}
\newtheorem{theorem}[subsection]{Theorem}
\newtheorem{lemma}[subsection]{Lemma}
\newtheorem{proposition}[subsection]{Proposition}

\newtheorem{scorollary}[equation]{Corollary}
\newtheorem*{claim*}{Claim}
\newtheorem{sclaim}[equation]{Claim}
\newtheorem{slemma}[equation]{Lemma}
\newtheorem{sproposition}[equation]{Proposition}
\newtheorem{problem}[subsection]{Problem}

\theoremstyle{definition}
\newtheorem{definition}[subsection]{Definition}
\newtheorem{mdefinition}[subsection]{}
\newtheorem{sdefinition}[equation]{}
\newtheorem{notation}[subsection]{Notation}
\newtheorem{snotation}[equation]{Notation}
\newtheorem{remarks}[subsection]{Remarks}
\newtheorem{remark}[subsection]{Remark}
\newtheorem{sremarks}[equation]{Remarks}
\newtheorem{sremark}[equation]{Remark}

\newtheorem{construction}[subsection]{Construction}
\newtheorem{sconstruction}[equation]{Construction}

\newtheorem{question}[subsection]{Question}

\newcounter{NN}
\renewcommand{\theNN}{\rm\arabic{NN}${}^o$}
\def\nr{\refstepcounter{NN}{\theNN}}

\newcounter{NNN}
\renewcommand{\theNNN}{\rm(\alph{NNN})}
\def\nnr{\refstepcounter{NNN}{\theNNN}}

\newcounter{NNNN}
\renewcommand{\theNNNN}{\rm(\roman{NNNN})}
\def\nnnr{\refstepcounter{NNNN}{\theNNNN}}

\title{Fano-Mukai fourfolds of genus $10$ \\ as compactifications of $\CC^4$}

\author{Yuri Prokhorov}
\thanks{
The first author was partially supported by
the RFBR grants 15-01-02164, 15-01-02158,
and by the Russian Academic Excellence Project '5-100'.
}

\address{Yuri Prokhorov:\newline\indent
Steklov Mathematical Institute,
8 Gubkina street, Moscow 119991, Russia
\newline\indent
Faculty of Mathematics, Moscow State
University, Russia
\newline\indent
National Research University Higher School of Economics, Russia
}
\email{prokhoro@mi.ras.ru}

\author{Mikhail Zaidenberg}

\address{Mikhail Zaidenberg:\newline\indent
Univ.\ Grenoble Alpes, CNRS, Institut Fourier, F-38000 Grenoble, France}
\email{Mikhail.Zaidenberg@univ-grenoble-alpes.fr}

\begin{document}

\begin{abstract} 
It is known that the moduli space of smooth Fano-Mukai fourfolds $V_{18}$ of 
genus $10$ has dimension one. We show that any such fourfold is a completion of 
$\mathbb{C}^4$ in two different ways. Up to isomorphism, there is a unique 
fourfold $V_{18}^{\mathrm s}$ acted upon by $\operatorname{SL}_2(\mathbb{C})$. 
The group $\operatorname{Aut}(V_{18}^{\mathrm s})$ is a semidirect product 
$\operatorname{GL}_2(\mathbb{C})\rtimes(\mathbb{Z}/2\mathbb{Z})$. Furthermore, 
$V_{18}^{\mathrm s}$ is a $\operatorname{GL}_2(\mathbb{C})$-equivariant 
completion of $\mathbb{C}^4$, and as well of $\operatorname{GL}_2(\mathbb{C})$. 
The restriction of the $\operatorname{GL}_2(\mathbb{C})$-action on 
$V_{18}^{\mathrm s}$ to $\mathbb{C}^4\hookrightarrow V_{18}^{\mathrm s}$ yields 
a faithful representation with an open orbit. There is also a unique, up to 
isomorphism, fourfold $V_{18}^{\mathrm a}$ such that the group 
$\operatorname{Aut}(V_{18}^{\mathrm a})$ is a semidirect product $({\mathbb 
G}_{\mathrm{a}}\times{\mathbb G}_{\mathrm{m}})\rtimes (\mathbb{Z}/2\mathbb{Z})$. 
For a Fano-Mukai fourfold $V_{18}$ neither isomorphic to $V_{18}^{\mathrm s}$, 
nor to $V_{18}^{\mathrm a}$, the group $\operatorname{Aut}(V_{18})$ is a 
semidirect product of $({\mathbb G}_{\mathrm{m}})^2$ and a finite cyclic group 
whose order is a factor of $6$. Besides, we establish that the 
affine cone over any polarized Fano-Mukai variety $V_{18}$ is flexible 
in codimension one, and flexible if $V_{18}=V_{18}^{\mathrm s}$.
\end{abstract}

\maketitle

\tableofcontents

\section{Introduction}
\label{sec:intro}
Let $V$ be
a compact complex manifold, and let $A$ be a closed analytic subset in $V$. The 
pair $(V,A)$ is
called a \emph{compactification of $\CC^n$} if $V\setminus A$ is
biholomorphically equivalent to $\CC^n$. A compactification $(V,A)$ of
$\CC^n$
is said to be \emph{projective} if $V$ is a smooth projective variety.
A celebrated Hirzebruch problem (\cite[Problem~27]{Hirzebruch1954}) asks to 
describe all
possible compactifications of $\CC^n$ with $\bb_2(V)=1$.
For any projective compactification $(V,A)$ of $\CC^n$ the Kodaira
dimension of $V$
is negative (see~\cite[Theorem~3]{Kodaira1971/72}). It is unknown, however,
whether
the complement $V\setminus A$ must be automatically biregularly isomorphic to
$\CC^n$. There is a
related open problem (\cite{Za})
on existence of
an affine algebraic variety $X$ non-isomorphic to $\CC^n$ biregularly, but
analytically isomorphic to $\CC^n$.

In this paper
we deal
with projective compactifications
$(V,A)$ of $\CC^n$ with
$\bb_2(V)=1$ and with $V\setminus A$ biregularly isomorphic to $\CC^n$.
Then $V$ is a Fano manifold with
Picard number one. In the first nontrivial case $n=3$
the classification was completed in a series of papers
\cite{Peternell1988},~\cite{Furushima1990},
\cite{Prokhorov1991a},~\cite{Furushima1993a}, see also
the references therein.
Such a threefold $V$ is one of the following:

\begin{itemize}

\item
$\PP^3$;

\item
a smooth quadric $Q\subset\PP^4$;

\item
a del Pezzo quintic threefold $V_5\subset\PP^6$;

\item
a Fano threefolds $V_{22}\subset\PP^{13}$ varying in a
proper subset of codimension $2$ in the moduli space.
\end{itemize}

All the possibilities for the divisor $A$ are also
described.

The four-dimensional
case is much more complicated. Indeed, already the very
classification of smooth Fano fourfolds is still lacking.

Recall that
the \emph{Fano index} $\iota(V)$ of a Fano manifold $V$ is the largest integer
$\iota$
such that the canonical class $K_V$ is divisible by $\iota$ in the Picard
group.
It is well known that $1\le\iota(V)\le\dim V+1$.
Furthermore, $\iota(V)=\dim V+1$ if and only if $V\cong\PP^n$ and
$\iota(V)=\dim V$ if and only if $V$ is a quadric $Q^n\subset\PP^{n+1}$.
This implies immediately that $(\PP^{n},\PP^{n-1})$ and $(Q^{n}, Q')$, where
$Q'$
is a singular hyperplane section of the quadric $Q^n$, are the only projective
compactifications of $\CC^n$
with $\bb_2(V)=1$ and $\iota(V)\ge n$.

A Fano manifold $V$ with $\rk\Pic (V)=1$ is called a \emph{del Pezzo manifold}
if $\iota(V)=\dim V-1$
and a \emph{Fano-Mukai manifold} if $\iota(V)=\dim V-2$. Compactifications
$(V,A)$ of $\CC^4$ with $\bb_2(V)=1$ and $\iota(V)=3$
were classified in~\cite{Prokhorov1994}. Such a del Pezzo fourfold
$V=W_5\subset\PP^7$ is unique up to isomorphism, and there are exactly
four possible choices for a divisor $A\subset V$ with $V\setminus A\cong
\CC^4$.
All such divisors $A$ are singular.
The singular locus of $A$ is either a plane, or a special line of one of
two possible types,
or finally a unique ordinary double point.

In~\cite[\S~6.3]{Prokhorov-Zaidenberg-2015} we asked the following
question:\smallskip

\noindent \emph{Which Fano-Mukai fourfolds can serve as compactifications of
$\CC^4$?}\smallskip

\noindent We give the answer for
the Fano-Mukai fourfolds
$V_{18}\subset\PP^{12}$
of genus $10$ and Fano index $2$.
According to~\cite{Mukai-1989} there is a homogeneous space
$\G/P\subset\PP^{13}$ of the simple algebraic group of type $\G$ such that any
variety $V_{18}$
is isomorphic to a hyperplane section of $\G/P$.
Up to isomorphism, such varieties $V_{18}$ form a one-parameter family, see 
Remark~\xref{cor:moduli}. It
occurs that all of them are compactifications of $\CC^4$, see 
Theorem~\xref{thm:main}.
There are two distinguished members of the family. One of them, denoted by 
$V_{18}^{\s}$, is quasihomogeneous with respect to
a $\GL_2(\CC)$-action
and
yields a $\GL_2(\CC)$-equivariant compactification of $\CC^4$ and of $\GL_2(\CC)$.
Another one,
denoted by $V_{18}^{\aaa}$, is acted upon by the product $\Ga\times\Gm$. This 
group is the identity component of $\Aut(V_{18}^{\aaa})$ and has index two in the 
full automorphism group. Any other smooth member $V_{18}$ of the family is acted upon 
by the torus of rank two, see Theorem~\xref{thm:main-aut-n} below.
More formally, our main results are the following three theorems.

\begin{theorem}
\label{thm:main}
Any smooth Fano-Mukai fourfold $V=V_{18}\subset\PP^{12}$ of genus $10$
contains at least two distinct $\Aut^0(V)$-invariant cones
over rational twisted cubic curves. If $S\subset V$ is such a cone, then there
is a unique $\Aut^0(V)$-invariant hyperplane section
$A=A_S$ of $V$ with $\Sing(A)=S$ such that $(V,A)$ is an $\Aut^0(V)$-equivariant
compactification of $\CC^4$.
\end{theorem}

\begin{theorem}
\label{thm:G/P} Given a
Fano-Mukai fourfold $V=V_{18}\subset\PP^{12}$ of genus $10$ consider a Mukai realization of $V$ 
as a hyperplane section of $\G/P\hookrightarrow\PP^{13}$, see 
Section~\xref{appendix}. Then
the $\Aut^0(V)$-action on $V$
extends to an $\Aut^0(V)$-action on $\G/P$ induced by the natural $\G$-action.
\end{theorem}

\begin{theorem}
\label{thm:main-aut-n}
\setenumerate[0]{leftmargin=8pt,itemindent=7pt}
\begin{enumerate}
\item
\label{thm:main-aut-n-GL2} 
There exists a smooth Fano-Mukai fourfold $V_{18}^{\s}$
of genus $10$ with
\begin{equation}
\label{eq:aut-GL2} 
\Aut(V_{18}^{\s})\cong\GL_2(\CC)\rtimes(\ZZ/2\ZZ),
\end{equation}
where the generator of $\ZZ/2\ZZ$ acts on $\GL_2(\CC)$ via $M\mapsto 
(M^t)^{-1}$.
Furthermore, $\Aut^0(V_{18}^{\s})\cong\GL_2(\CC)$ has a principal dense open orbit in 
$V_{18}^{\s}$ and exactly two fixed points.
Any Fano-Mukai fourfold $V_{18}$
of genus $10$ whose automorphism group has non-abelian identity component
is isomorphic to $V_{18}^{\s}$.

\item
\label{thm:main-aut-n-Ga}
There exists a smooth Fano-Mukai fourfold $V_{18}^{\aaa}$
of
genus $10$ with
\begin{equation}
\label{eq:aut-Ga-Gm} 
\Aut(V_{18}^{\aaa})\cong (\Ga\times\Gm)\rtimes (\ZZ/2\ZZ),
\end{equation} 
where the generator of $\ZZ/2\ZZ$ acts by the inversion $g\mapsto g^{-1}$ on 
$\Ga\times\Gm$.
Such a fourfold is unique up to isomorphism.

\item
\label{thm:main-aut-n-Gm}
For any smooth Fano-Mukai fourfold $V_{18}$ of
genus $10$ non-isomorphic to one of the $V_{18}^{\s}$ and $V_{18}^{\aaa}$ one 
has
\begin{equation}
\label{eq:aut-Gm-2} (\Gm)^2\subset\Aut(V_{18})\subset (\Gm)^2\rtimes (\ZZ/6\ZZ).
\end{equation}
\end{enumerate}
\end{theorem}

See also Theorems~\xref{cor:final} and~\xref{thm:main-aut} for some additional 
information. In particular, it occurs that the fourfold $V_{18}^{\s}$ contains 
two one-parameter families of cones over twisted cubics and a unique pair of 
$\Aut^0(V)$-invariant such cubic cones. The number of cubic cones in $V_{18}^{\aaa}$ 
equals $4$, and equals 
$6$ in $V_{18}\not\cong V_{18}^{\s},V_{18}^{\aaa}$. Any such cone $S$ defines 
an 
$\Aut^0(V)$-invariant compactification $(V_{18},A_S)$ of $\CC^4$ as in 
Theorem~\xref{thm:main}. 

The proofs of Theorems~\xref{thm:main}--\xref{thm:main-aut-n} are done in 
Section~\xref{sec:proofs}. They use a construction of $V_{18}$ starting with 
$\PP^4$ via a sequence of two Sarkisov links, see 
Sections~\xref{sec-2}--\xref{sec-3} for details. The first Sarkisov link gives 
the quintic del Pezzo fourfold $W_5\subset\PP^7$. In 
Sections~\xref{sec-2bis}--\xref{sec-3-bis} we study the automorphism group of 
$W_5$ (cf.~\cite{Piontkowski-Van-de-Ven-1999}), its action on $W_5$, and the 
stabilizers of certain rational normal quintic scrolls in $W_5$. Such a scroll 
$F$ serves as the center of blowup for the second Sarkisov link. For a properly 
chosen $F$ the stabilizer of $F$ in $\Aut^0(W_5)$ is isomorphic to the 
automorphism group $\Aut^0(V_{18})$ of the resulting Fano-Mukai fourfold 
$V_{18}$. On the other hand, the construction of S.~Mukai (\cite{Mukai-1988}) 
embeds any Fano-Mukai fourfold $V_{18}$ onto a hyperplane section of 
$\G/P\subset\PP^{13}$. Using results of M.~Kapustka and K.~Ranestad 
(\cite{KapustkaRanestad2013}) we compute in Section~\xref{appendix} the 
stabilizers of these hyperplane sections as subgroups of $\Aut^0(V_{18})$. It 
remains to show that such a stabilizer coincides actually with the whole group 
$\Aut^0(V_{18})$, cf.\ Theorem~\xref{thm:G/P}. To this end, following again 
(\cite{KapustkaRanestad2013}) we study in detail the Hilbert schemes of lines 
and of rational normal cubic scrolls on $V_{18}$ (see 
Sections~\xref{sec:lines-in-V}--\xref{sec:thm-1.2}), and the subschemes of 
rational cubic cones (see Section~\xref{sec:del-Pezzo}). The latter cones occur 
to be in one-to-one correspondence with the lines on certain singular del Pezzo 
sextic surfaces. This geometry allows us to describe the automorphism groups of 
the fourfolds $V_{18}$, see Sections~\xref{sec:1.3-1.4}-\xref{sec:1.5-1.6}. This 
leads finally in Section~\xref{sec:proofs} to our main results. In 
Section~\xref{sec:aut-aff-cones} (see Theorem~\xref{thm:flexible-cones}) we 
deduce the flexibility in codimension one of the affine cones over the 
Fano-Mukai fourfolds $V_{18}$ and the flexibility if $V_{18}=V_{18}^{\s}$, cf.\ 
\cite{AFKKZ}. The concluding Section~\xref{sec:open-problems} contains some 
remarks and open problems.

\section{Linking the del Pezzo fourfold $W_5$ to $\PP^4$}
\label{sec-2}

\begin{mdefinition}
In this and the next sections $W=W_5\subset\PP^7$ stands
for a del Pezzo quintic fourfold realized as a smooth section of the
Grassmannian $\Gr(2,5)$ under its Pl\"ucker embedding in $\PP^9$ by a general
linear subspace of codimension $2$ in $\PP^9$.
In fact, a del Pezzo quintic fourfold is unique up to isomorphism
(\cite{Fujita-620281}). We use the following
description of the
planes in $W$ (see the classical paper~\cite{Todd1930} for more details and
\cite[\S~6]{Piontkowski-Van-de-Ven-1999} for a modern treatment).
\end{mdefinition}

\begin{sdefinition}
We regard the Grassmannian $\Gr(2,5)$ as the variety
of lines in $\PP^4=\PP(\CC^5)$.
Recall that any plane in $\Gr(2,5)\subset\PP^9$ is a Schubert variety of one of
the following two types :

\begin{itemize}

\item
$\sigma_{2,2}$, that is, the Schubert variety
of lines in a fixed plane $\PP^2\subset\PP^4$;

\item
$\sigma_{3,1}$, that is, the Schubert variety
of lines passing
through a fixed point and contained
in
a fixed 3-space $\PP^3\subset\PP^4$.
\end{itemize}
\end{sdefinition}

\begin{sdefinition}
\label{def:skew-symmetric}
Let $\PPP$ be a pencil of hyperplane sections which cut out a del Pezzo
quintic fourfold $W$ in $\Gr(2,5)\subset\PP^9$.
It can be treated as a pencil
of skew-symmetric bilinear forms
$\lambda_1 q_1+\lambda_2 q_2\in (\wedge^2\CC^5)^\vee$.
Since $W$ is smooth, each form $\lambda_1 q_1+\lambda_2 q_2\in\PPP$ is of
rank 4. Consider
the map $\upsilon:\PPP\to\PP(\CC^5)$ that sends a
form $\lambda_1 q_1+\lambda_2 q_2$
to the projectivization of its kernel
$\ker (\lambda_1 q_1+\lambda_2 q_2)\subset\CC^5$. This map
is given by the Pfaffians of the corresponding matrix. Hence
the image $\upsilon(\PPP)$ is a conic in
$\PP^4=\PP(\CC^5)$ (\cite[Prop.~6.3]{Piontkowski-Van-de-Ven-1999}).
The linear span $\Theta=\langle\upsilon(\PPP)\rangle$ is a plane in
$\PP(\CC^5)$, which is a maximal common isotropic
subspace for the forms $\lambda_1 q_1+\lambda_2 q_2\in\PPP$. Such a
subspace is unique. This defines a unique
$\sigma_{2,2}$-plane $\Xi\subset W$ parameterizing the lines in
$\Theta\cong\PP^2$.
On the other hand, there is a one-parameter family of
$\sigma_{3,1}$-planes
$\Pi_{\gamma}$ parameterizing the lines passing through a point $P_{\gamma}\in
\upsilon(\PPP)$
and contained in the three-dimensional subspace in $\PP^4$ orthogonal to
$P_{\gamma}$ with respect to
any form $\lambda_1 q_1+\lambda_2 q_2\in\PPP$. Let $\Upsilon\subset\Xi$
be the dual conic of $\upsilon(\PPP)\subset\Theta$.
Each plane $\Pi_\gamma$ meets $\Xi$ along a tangent line $l_{\gamma}$
to $\Upsilon$ at a point $\gamma\in\Upsilon$.
Any two distinct planes $\Pi_\gamma$ and $\Pi_{\gamma'}$ meet at a
unique point $l_\gamma\cap l_{\gamma'}$ on $\Xi\setminus\Upsilon$.

The planes $\{\Pi_\gamma\}_{\gamma\in\Upsilon}$ and $\Xi$ are the only planes
contained in $W$. The union $R=\bigcup_{\gamma\in\Upsilon}\Pi_\gamma$ is a
hyperplane
section of $W$. The threefold $R$ contains also $\Xi$ and
is singular along $\Xi$ (see, 
e.g.,~\cite{Todd1930},~\cite[Prop.~3.4]{Prokhorov-Zaidenberg-4-Fano} and the 
references therein).
The triple $W\supset R\supset\Xi$ plays an important
role in what follows.
\end{sdefinition}
Consider the following Sarkisov link (for the proofs, see~\cite{Fujita-620281},
\cite{Prokhorov1994},
\cite[Prop.~4.9]{Prokhorov-Zaidenberg-4-Fano} and the references therein).

\begin{proposition}
\label{prop:link-1}
In the notation as before, the following hold.

\begin{enumerate}
\item
\label{prop:link-1-1}
There is a commutative diagram
\begin{equation}
\label{equation-diagram-1}
\vcenter{
\xymatrix@1{
&\hat E\ar[dl]
\,\ar@{}[r]|-*[@]{\subset}
&&\hat W\ar[dr]^{\varphi}\ar[dl]_{\rho}\, &\ar@{}[r]|-*[@]{\supset}&\hat
R\ar[dr]
\\
\Xi\,\ar@{}[r]|-*[@]{\subset} &R\,\ar@{}[r]|-*[@]{\subset}
&W\ar@{-->}[rr]^{\phi}&&
\PP^4\ar@{}[r]|-*[@]{\supset}&E=\langle\Gamma\rangle\ar@{}[r]|-*[@]{\supset}
&\Gamma
}}
\end{equation}
where $\rho$ is the blowup of the plane $\Xi$ in $W$ and
$\varphi$ is the blowup of a rational twisted cubic curve
$\Gamma\subset\PP^4$.

\item
\label{prop:link-1-2}
The $\rho$-exceptional divisor $\hat E=\rho^{-1}(\Xi)$
is the proper transform in
$\hat W$ of the linear span $\langle
\Gamma\rangle\cong\PP^3$ of $\Gamma$ in $\PP^4$.
The $\varphi$-exceptional divisor $\hat R$ is the proper transform
of $R$ in $\hat W$.

\item
\label{prop:link-1-3}
The morphism $\varphi$ \textup{(}$\rho$, respectively\textup{)} is defined by
the linear system $|H^*-\hat E|$ \textup{(}$|2L^*-\hat R|$,
respectively\textup{)} on $\hat W$,
where $H$ \textup{(}$L$, respectively) is the class of hyperplane section on
$W$
\textup{(}on $\PP^4$, respectively\textup{)}, and $H^*=\rho^*(H)$,
$L^*=\varphi^*(L)$. The birational map
$\phi: W\subset\PP^7\dashrightarrow\PP^4$ is the linear
projection with center $\Xi$.
\end{enumerate}
\end{proposition}

\begin{scorollary}[{\cite{Prokhorov1994}}]
\label{cor:2.2.2}
$(W,R)$ is a compactification of $\CC^4$.
\end{scorollary}

\begin{proof}
Indeed, due to~\eqref{equation-diagram-1} we have isomorphisms

\begin{equation}
\label{eq:2.2.2}
W\setminus R\cong\hat W\setminus (\hat R\cup\hat E)\cong\PP^4\setminus
\langle\Gamma\rangle\cong\CC^4.
\end{equation}
\end{proof}

\begin{scorollary}
\label{cor:Aut-W5}\footnote{Cf.\ Lemma \ref{lem:PvdV}.} There is an isomorphism
$\Aut(W)\cong\Aut(\PP^4,\Gamma)$. 
\end{scorollary}

\begin{proof}
Indeed, the plane $\Xi\subset W$ is
$\Aut(W)$-invariant.
Due to diagram
\eqref{equation-diagram-1} one has $\Aut(W)=\Aut(W,\Xi)\cong
\Aut(\PP^4,\Gamma)$. 
\end{proof}

\begin{lemma}
\label{lem:hatR}
In the notation as before, the following hold.

\begin{enumerate}
\renewcommand\labelenumi{\rm (\alph{enumi})}
\renewcommand\theenumi{\rm (\alph{enumi})}

\item
\label{lem:hatR-a}
$\varphi|_{\hat R}:\hat R\to\Gamma$ is a
$\PP^2$-bundle.

\item
\label{lem:hatR-b}
For any
$\gamma\in\Gamma$, the fiber $\hat{\Pi}_\gamma:=\textup(\varphi|_{\hat
R})^{-1}(\gamma)\subset\hat R$ is sent by $\rho$ isomorphically onto a plane
$\Pi_\gamma\subset R$.

\item
\label{lem:hatR-c}
$\rho|_{\hat R}:\hat R\to R$ is the normalization morphism.

\item
\label{lem:hatR-d}
Let $\hat\Xi:=\hat E\cap\hat R=(\rho|_{\hat R})^{-1}(\Xi)$.
Then $\hat\Xi\cong\PP^1\times\PP^1$ and $\varphi|_{\hat\Xi}:\hat\Xi\to\Xi$ is a 
double cover branched along the conic $\Upsilon\subset\Xi$
\textup(see~\eqref{def:skew-symmetric}\textup).
\end{enumerate}
\end{lemma}

\begin{proof}
\xref{lem:hatR-a}
follows since
$\hat R=\PP(\NNN_{\Gamma/\PP^4}^{\vee})$ is the exceptional divisor of
the blowup of $\Gamma$ in $\PP^4$.

\xref{lem:hatR-b} Let $L$ be a hyperplane in $\PP^4$. On $\hat W$ one has
(\cite[Prop.~4.9]{Prokhorov-Zaidenberg-4-Fano}):

\begin{equation*}
L^*\sim H^*-\hat E\quad\text{and}\quad\hat R\sim H^*-2\hat E.
\end{equation*}
Hence $\varphi$ is defined by the linear system $|H^*-\hat E|$ on $\hat W$, and
\begin{equation}
\label{eq:exprim}
H^*\sim 2L^*-\hat R\quad\text{and}\quad\hat E\sim L^*-\hat R.
\end{equation}
So, $\rho$
is the morphism given by the base point free linear system $|2L^*-\hat R|$ on
$\hat W$, and the map
$\phi^{-1}$ in~\eqref{equation-diagram-1} is defined by the family of quadrics
in $\PP^4$ passing through
$\Gamma$.

We have: $\OOO_{\hat R}(\hat R)
=\OOO_{\PP(\NNN^\vee_{\Gamma/\PP^4})}(-1)$. Since
$\OOO_{\hat{\Pi}_\gamma}(L^*)=\OOO_{\hat{\Pi}_\gamma}$, we obtain

\begin{equation*}
\OOO_{\hat{\Pi}_\gamma}(H^*)=\OOO_{\hat{\Pi}_\gamma}(2L^*-\hat
R)\cong\OOO_{\PP^2}(1).
\end{equation*}
It follows that
\begin{equation*}
H^2\cdot\Pi_\gamma=(H^*)^2\cdot\hat{\Pi}_\gamma=1.
\end{equation*}
Therefore,
$\{\Pi_\gamma\}$ is a family of planes in $\PP^7$ that covers $R$. However,
there is just one such family, see~\xref{def:skew-symmetric}. This proves
\xref{lem:hatR-b}.

\xref{lem:hatR-c}
$\hat R$ is smooth being an exceptional divisor of the blowup $\varphi$ with a
smooth center. Hence to establish~\xref{lem:hatR-c} it
suffices to show that
$\rho|_{\hat R}:\hat R\to R$ is a finite birational morphism.
Suppose that $\varphi|_{\hat R}$ contracts a curve $J\subset\hat R$.
Clearly, $J$ is not contained in a fiber $\cong\PP^2$ of $\rho|_{\hat R}$.
Therefore, $J$ meets any fiber of $\rho|_{\hat R}$.
But then the planes $\Pi_\gamma$ should pass all through a common point
$\rho(J)$, which
is not the case (see~\xref{def:skew-symmetric}). Since $\rho$ is a birational
morphism of smooth varieties,
$\rho|_{\hat R}$ is birational.

\xref{lem:hatR-d}
By construction, $\hat\Xi=\hat E\cap\hat R$ is the exceptional divisor of
the blowup $\varphi|_{\hat E}:\hat E\to\langle\Gamma\rangle\cong\PP^3$
with center $\Gamma$. Since
$\NNN_{\Gamma/\PP^3}\cong\OOO_{\PP^1}(5)\oplus\OOO_{\PP^1}(5)$
(see e.g.~\cite{Eisenbud-van-de-Ven-1981}), we have
$\hat E\cap\hat
R\cong\PP\bigl(\NNN^\vee_{\Gamma/\PP^3}\bigr)\cong\PP^1\times\PP^1$.
On $\hat W$ one has ${L^*}^2\cdot\hat R^2=0={L^*}^3\cdot\hat R$.
By~\eqref{eq:exprim} and~\cite[Lem.~2.3, Lem.~4.6]{Prokhorov-Zaidenberg-4-Fano}
the degree of the morphism $\rho|_{\hat\Xi}:\hat\Xi\to\Xi$ equals
\begin{equation*}
{H^*}^2\cdot\hat E\cdot\hat R=(2L^*-\hat R)\cdot (L^*-\hat R)\cdot\hat R=5\hat
R^3\cdot L^*-\hat R^4=2.
\end{equation*}
It follows that
$\rho|_{\hat\Xi}:\hat\Xi\to\Xi$ is equivalent to a covering
$\PP^1\times\PP^1\to\PP^2$ of degree $2$ branched along a smooth conic.
Furthermore, $\varphi|_{\hat\Xi}:\hat\Xi\to\Gamma$ is equivalent to a
canonical projection $\PP^1\times\PP^1\to\PP^1$, with fibers being the
lines $\hat\Pi_\gamma\cap\hat\Xi$. These lines are sent by $\rho$
to the tangent lines $\Pi_\gamma\cap\Xi$ of $\Upsilon$. Hence $\Upsilon$ is
the branching divisor of $\hat\Xi\to\Xi$ and
the
image of the ramification divisor of type $(1,1)$ on
$\hat\Xi\cong\PP^1\times\PP^1$. Now~\xref{lem:hatR-d} follows.
\end{proof}

\begin{lemma}
\label{lem:Gamma}
There exists a smooth rational
curve $\hat\Gamma\subset\hat R$ such that

\begin{itemize}

\item
the restrictions $\varphi|_{\hat\Gamma}:\hat\Gamma\to\Gamma$ and 
$\rho|_{\hat\Gamma}:\hat\Gamma\to\rho(\hat\Gamma)$
are isomorphisms;

\item
$\rho(\hat\Gamma)\subset R$ is a twisted cubic;

\item
$\hat E\cap\hat\Gamma=\emptyset$ and
$\bigl\langle\rho(\hat\Gamma)\bigr\rangle\cap\Xi=\emptyset$.
\end{itemize}
\end{lemma}

\begin{proof}
Fix a point $P\in\PP^4\setminus\langle\Gamma\rangle$. Let $N\subset\PP^4$ be
the cone over $\Gamma$ with vertex $P$, and let $\hat N$ be the proper
transform of $N$
in $\hat W$. Consider the curve $\hat\Gamma=\hat
N\cap\hat R$. We claim that $\hat\Gamma$ satisfies the conditions of the
lemma.
Indeed, since $\Gamma\subset N$ is a Cartier divisor,
$\varphi|_{\hat N}:\hat N\to N$ is an isomorphism, and
$\varphi(\hat\Gamma)=\Gamma$. Hence $\varphi|_{\hat\Gamma}:
\hat\Gamma\to\Gamma$ is an isomorphism as well.

Furthermore, one has $\Gamma=N\cap\langle\Gamma\rangle$,
and $N$ meets $\langle\Gamma\rangle$ transversely
along $\Gamma$.
It follows that
$\hat E$ and $\hat\Gamma$ are disjoint. Since $\hat E$ is the
exceptional divisor of $\rho$ and $\rho(\hat E)=\Xi$, we conclude that
$\rho|_{\hat\Gamma}:\hat\Gamma\to\rho(\hat\Gamma)$ is an isomorphism,
and $\rho(\hat\Gamma)\cap\Xi=\emptyset$. Moreover, since $\rho(\hat\Gamma)$
does not meet the center $\Xi$ of the birational projection $\phi:
W\dashrightarrow\PP^4$, the restriction 
$\phi|_{\rho(\hat\Gamma)}:\rho(\hat\Gamma)\to\Gamma$ is an isomorphism 
preserving the degree. Therefore,
$\rho(\hat\Gamma)$ is a twisted cubic, since $\Gamma\subset\PP^4$ is.

Finally, if $\Xi$ meets $\langle\rho(\hat\Gamma)\rangle\cong\PP^3$,
then $\rho(\hat\Gamma)$ has a $2$-secant line meeting $\Xi$.
However, in the latter case $\phi|_{\rho(\hat\Gamma)}:\rho(\hat\Gamma)\to
\Gamma$
cannot be an isomorphism, a contradiction.
\end{proof}

\section{Linking the Fano-Mukai fourfolds $V_{18}$ to $W_5$}
\label{sec-4}

Any smooth Fano-Mukai fourfold $V_{18}$ of genus 10 can be obtained starting
with $W_5$ via a
Sarkisov link. This link can be described as follows.

\begin{proposition}
\label{prop:reversion}
Let $V=V_{18}\subset\PP^{12}$ be a smooth Fano-Mukai fourfold, and let 
$S\subset V$ be either a smooth two-dimensional cubic scroll, or a cone over a 
rational twisted cubic curve\footnote{Such cones are the only cubic cones in 
$V$, see~\xref{def:cubic-cones}. Hence in the sequel we call them simply cubic 
cones. We will show that any smooth variety $V_{18}$ contains a cubis scroll or a cubic cone.}. 
Then $V\cap\langle S\rangle=S$ as a scheme, and there is a Sarkisov 
link
\begin{equation}
\label{diagram-2}
\vcenter{
\xymatrix@1{
&\tilde B\ar[dl]\,\ar@{}[r]|-*[@]{\subset}& &
\tilde W\ar[dr]^{\eta}\ar[dl]_{\xi}&\,\ar@{}[r]|-*[@]{\supset}&\tilde A\ar[dr]
\\
S\,\ar@{}[r]|-*[@]{=}&\langle S\rangle\cap V\subset A\,
\ar@{}[r]|-*[@]{\subset}&
V\ar@{-->}[rr]^{\theta}&& W\,\ar@{}[r]|-*[@]{\supset}& B=W\cap\langle
F\rangle\,\ar@{}[r]|-*[@]{\supset}&F
}}
\end{equation}
where
\begin{itemize}

\item
$\xi$ is the blowup of $S$ in $V$ with exceptional divisor $\tilde B$;

\item
$\eta$ is the blowup of a smooth rational quintic scroll $F\subset W$
with exceptional divisor $\tilde A$;

\item
$\xi$ sends $\tilde A$ to a hyperplane section $A$ of $V$ with
$\Sing(A)=S$, and $\eta$ sends $\tilde B$ to the hyperplane section
$B=W\cap\langle
F\rangle$;

\item
the map $\theta: V\dashrightarrow W=W_5$
comes from the linear projection $\PP^{12}\dashrightarrow\PP^7$
with center $\langle S\rangle\cong\PP^4$.
\end{itemize}
\end{proposition}

\begin{proof}
The proof proceeds in several steps.

\begin{sclaim}
\label{claim:they-coinside}
We have $V\cap\langle S\rangle=S$ as a scheme.
\end{sclaim}

\begin{proof}
It is well known that any linearly non-degenerate surface of degree $3$
in $\PP^4$
is an intersection of three quadrics.
On the other hand, $V=V_{18}\subset\PP^{12}$ is an intersection of
quadrics too (see~\cite[Lem.~2.10]{Iskovskih1977a}).
Let $\mathcal{Q}\subset H^0(\PP^{12},\OOO_{\PP^{12}}(2))$ be the linear system 
of quadrics passing through $V$, and let
$\mathcal{Q}'$
be the restriction of $\mathcal{Q}$ to $\langle S\rangle$.
Then $\mathcal{Q}'$ cuts out on $\langle S\rangle$ a scheme $V\cap\langle
S\rangle$ of dimension $2$ containing $S$. Hence $2\le\dim\mathcal{Q}'\le 3$.
Suppose that $\dim\mathcal{Q}'=2$, that is, $\mathcal{Q}'$ is
generated by two linearly independent quadrics $Q_1, Q_2\in\mathcal{Q}'$. Then
$Q_1\cap Q_2=S\cup\Pi$, where $\Pi\subset V$ is the residual plane.
However, a Fano-Mukai fourfold of genus $10$ does not contain any plane
(\cite[Lem.~3]{KapustkaRanestad2013}).
Hence $\dim\mathcal{Q}'=3$. So, any
quadric in $\langle S\rangle$ containing $S$ is a member of $\mathcal{Q}'$.
Since these quadrics cut out in $\PP^4$ the surface $S$,
it follows that $S=V\cap\langle S\rangle$ (as a scheme).
\end{proof}

\begin{sclaim}
\label{claim:3.3.2}
There exists a line $l$ on $V$
meeting $S$.
\end{sclaim}

\begin{proof}
Let $L=V\cap\PP^{11}$ be a general hyperplane section of $V$. Then $L$ meets
$S$ along a twisted cubic curve, say, $\Gamma'$.
By the adjunction formula and the Lefschetz hyperplane section theorem,
$L\subset\PP^{11}$ is an anticanonically embedded Fano threefold
with $\Pic L=\ZZ\cdot K_{L}$.
It is known (\cite[Prop.~4.2.2]{Iskovskikh-Prokhorov-1999}) that
$L\subset\PP^{11}$ carries a one-parameter family of lines. Since $\Pic
L\cong\ZZ$, the surface $T\subset L$ swept up by these lines meets $\Gamma'$. 
Now the claim follows.
\end{proof}

The next claim concludes the proof of Proposition~\xref{prop:reversion}.

\begin{sclaim}
Consider the linear projection $\theta:\PP^{12}\supset
V\dashrightarrow W:=\theta(V)\subset\PP^7$ with center
$\langle S\rangle\cong\PP^4$. Consider also the blowup $\xi:\tilde W\to V$ with 
center $S$.
Then $W\subset\PP^7$ is a smooth del Pezzo quintic fourfold, and there is a
morphism $\eta:\tilde W\to W$ such that these objects and morphisms fit in
diagram~\eqref{diagram-2}.
\end{sclaim}

\begin{proof}
Let $\tilde B$ be the exceptional divisor of the blowup $\xi:\tilde W\to V$ with 
center $S$. One can check that $\tilde W$ is smooth.
Let $L$ be a hyperplane section of $V$. Consider
the proper transform $|L^*-\tilde B|$ on $\tilde W$ of the linear system
of hyperplanes in $\PP^{12}$
passing through $S$. It is nef and base point free. The morphism
$\eta=\Phi_{|L^*-\tilde B|}:\tilde W\to W\subset\PP^7$ resolves
indeterminacies of $\theta$.
Using~\cite[Lem.~2.3]{Prokhorov-Zaidenberg-4-Fano} one computes
$(L^*-\tilde B)^4=5$ (cf.~\eqref{eq:zero-intersection} and the
subsequent paragraph).
Thus, $\deg W=5$, and so, $\eta$ is birational. The divisor
$-K_{\tilde W}=(L^*-\tilde B)+L^*$ is ample being the sum of two
non-proportional nef divisors. Let $\eta:\tilde W
\stackrel{\eta'}{\longrightarrow} W'\longrightarrow W$ be the Stein
factorization. Then $\eta'$ is either an isomorphism, or an extremal Mori 
contraction.
Let, as before, $H$ be a hyperplane section of $W\subset\PP^7$, and let
$H^*=\eta^*H$. Thus, $H^*\sim L^*-\tilde B$ on $\tilde W$.

Let $\tilde l\subset\widetilde W$ be the proper transform of a line $l\subset
V$ which meets $S$
properly, see~\xref{claim:3.3.2}. Then $\tilde B\cdot\tilde l\ge 1$ and
$L^*\cdot\tilde l=1$. Hence
\begin{equation*}
H^*\cdot\tilde l=(L^*-\tilde B)\cdot\tilde l\le
0.
\end{equation*}
Since the linear system $|H^*|$ is base point free, we have
\begin{equation}
\label{eq:decompos-0}
H^*\cdot\tilde l=0,\qquad\tilde B\cdot\tilde l=1,\qquad -K_{\tilde W}\cdot\tilde 
l=1.
\end{equation}
It follows that a non-ample, nef divisor $H^*$ supports an extremal ray $\tilde
l$ contracted by $\eta$.
In particular, $\eta'$ is not an isomorphism.

Suppose that $\eta'$ has a two-dimensional fiber, say, $\tilde M$.
By the main theorem
in~\cite{Andreatta1998a}, $\tilde M$ is isomorphic either to $\PP^2$, or to
a quadric $Q\subset\PP^3$. Moreover, $\OOO_{\tilde M}(-K_{\tilde W})\cong
\OOO_{\PP^2}(1)$, or,
respectively, $\OOO_{\tilde M}(-K_{\tilde W})\cong\OOO_{Q}(1)$
(\cite[Prop.~4.11]{Andreatta1998a}). On the other hand,
$\OOO_{\tilde M}(-K_{\tilde W})\cong\OOO_{\tilde M}(L^*)$
(cf.~\eqref{eq:3}).
Therefore, $\xi(\tilde M)$ is
either a plane, or a quadric surface in $V$.
However, $V$ contains neither a plane, nor a quadric surface, see~\cite[Lem.~3
and Cor.~3]{KapustkaRanestad2013}.
Thus, $\eta'$ has no two-dimensional fiber, and so, it is not flipping
(see e.g.~\cite[Main Theorem]{Andreatta1998a}).
Thus,
$\eta'$ contracts a divisor, say $\tilde A\subset\tilde W$, to a surface, say,
$F'=\eta'(\tilde A)\subset W'$. Again by the
main theorem
in~\cite{Andreatta1998a}, $W'$ and $\tilde A$ are smooth, and $\eta':\tilde
W\to W'$ is the blowup of a smooth surface $F'\subset W'$.

We have $-K_{\tilde W}=2H^*+\tilde B$. Set $H'=\eta'_* H^*$ and
$B'=\eta'_*\tilde B$. Then $-K_{W'}=2H'+B'$. Since $\rk\Pic W'=1$, the Fano
index of $W'$ is at least $3$.
Since ${H'}^4={H^*}^4=5$, $H'$ is not divisible. It follows from the
classification that $W'$ is a del Pezzo fourfold of degree $5$ and $H'\sim B'$ 
(\cite{Fujita-620281}).
Moreover, $H'=-\frac{1}{3}K_{W'}$ is very ample.
Hence $W'\to W$ is an isomorphism. Thus we have $-K_W\sim 3H$ and
\begin{equation}
\label{eq:decompos-2}
-K_{\tilde W}\sim 3H^*-\tilde A\sim 2H^*+\tilde B,\quad\tilde B\sim H^*-\tilde 
A.
\end{equation}

It remains to show that the smooth surface $F\subset W$ is a rational quintic
scroll. On $\tilde W$
we have $\tilde B\sim H^* -\tilde A$ due to~\eqref{eq:decompos-2}.
Since $V$ is smooth and $S\subset V$ has at worst isolated singularities,
the intersection numbers $(L^*)^i\cdot{\tilde B}^{4-i}$ for $i>0$ can be
computed in a usual way (see, 
e.g.,~\cite[Lem.~2.3]{Prokhorov-Zaidenberg-4-Fano}). This gives:
\begin{equation}
\label{eq:inters-indices}
(L^*)^4=18,\quad (L^*)^3\cdot\tilde B=0,\quad (L^*)^2\cdot\tilde B^2=-3,
\quad L^*\cdot\tilde B^3=-1.
\end{equation}
Since $H^*\sim L^*-\tilde B$ and, by~\eqref{eq:decompos-2},
\begin{equation}
\label{eq:inters-ALB}
\tilde A\sim L^*-2\tilde B,
\end{equation}
then the equality
$(H^*)^3\cdot\tilde A=0$ reads
\begin{equation}
\label{eq:zero-intersection}
(L^*-\tilde B)^3\cdot (L^*-2\tilde B)=0.
\end{equation}
Expressing ${\tilde B}^4$ via the intersection numbers
$(L^*)^i\cdot{\tilde B}^{4-i}$, $i=1,\ldots,4$, from~\eqref{eq:inters-indices}
and~\eqref{eq:zero-intersection} one gets ${\tilde B}^4=1$.
Furthermore,
\begin{equation*}
\deg F=-(H^*)^2\cdot{\tilde A}^{2}=-(L^*-\tilde B)^2\cdot (L^*-2\tilde B)^2=5.
\end{equation*}
Since $\tilde B\sim H^* -\tilde A$, the divisor $B:=\eta(\tilde B)$ is a
hyperplane section of
$W$ passing through $F$.
For any such hyperplane section, say, $D$, the proper
transform $\tilde D$ of $D$ in $\tilde W$ belongs to the linear system $|B|$.
However, this linear system contains just a single member $\tilde B$. Indeed,
the $\xi$-exceptional divisor $\tilde B$ is covered by lines
with negative intersection with $\tilde B$, and so, is not movable.
Therefore, there is just one hyperplane
section $B$ in $W$ through $F$, that is, $\langle F\rangle\cap W=B$.
Finally, $F\subset\langle F\rangle\cong\PP^6$ is a linearly nondegenerate
surface of minimal degree, hence a rational quintic scroll, see
\cite[Thm.~1]{Eisenbud1987}.
\end{proof}

This ends the proof of Proposition~\xref{prop:reversion}.
\end{proof}

\section{Linking $W_5$ to $V_{18}$}
\label{sec-3}

\begin{notation}
\label{nota:quintic}
Consider again a smooth del Pezzo
quintic fourfold $W=W_5\subset\PP^7$ with a distinguished hyperplane
section
$R$ of $W$
swept
up by the planes contained in $W$, with a distinguished plane $\Xi=\Sing
R\cong\PP^2$ and a distinguished smooth conic $\Upsilon\subset\Xi$.
Let $F$ be a smooth rational quintic scroll in $W$, not necessarily contained
in
$R$.
Let $\eta:\tilde W\to W$ be the blowup with center $F$ and exceptional
divisor $\tilde A$. Consider the hyperplane section $B=W\cap\langle F\rangle$
of $W$, and let $\tilde B$ be the
proper transform of $B$ in $\tilde W$.
Clearly, $B$ is smooth in codimension $1$, and so, $\tilde B\sim H^*-\tilde A$
on $\tilde W$, where $H$ is the class of a
hyperplane section of $W$.
\end{notation}

\begin{sremark}
\label{remark-quintic-scroll}
Recall that a rational normal quintic scroll $F$ spans $\PP^6$. There exist such 
scrolls of two different kinds, namely,
\begin{enumerate}
\item
\label{remark-quintic-scroll-F1}
$F\cong\FF_1$, where the embedding $\FF_1\hookrightarrow\PP^6$ is given by
the
linear system $|s_0+3f|$;

\item
\label{remark-quintic-scroll-F3}
$F\cong\FF_3$, where the embedding $\FF_3\hookrightarrow\PP^6$ is given by
the
linear system $|s_0+4f|$,
\end{enumerate} with the exceptional section $s_0$ and a fiber $f$.
\end{sremark}

The following proposition is an analog of Theorem~2.1.a in
\cite{Prokhorov-Zaidenberg-2015}. The latter theorem is proven in
\cite{Prokhorov-Zaidenberg-2015} under
an additional assumption that the given rational quintic scroll $F\subset W$
does not
meet any plane in $W$ along a conic. We do not keep any longer this assumption.
Nonetheless, a part
of the proof of Theorem~2.1.a in~\cite{Prokhorov-Zaidenberg-2015} goes through
in our setup as well.

\begin{proposition}
\label{prop:link-2}
For any smooth rational quintic scroll $F\subset
\PP^7$ contained in $W$, the data $(W,F,\tilde W,\tilde A,\tilde B)$ as in
\xref{nota:quintic} fits in a Sarkisov link~\eqref{diagram-2}, where
\begin{enumerate}
\renewcommand\labelenumi{\rm (\alph{enumi})}
\renewcommand\theenumi{\rm (\alph{enumi})}
\item
\label{prop:link-2-a}
$V=V_{18}\subset\PP^{12}$ is a Fano-Mukai fourfold of genus $10$ with at
worst
ordinary double points as singularities;

\item
\label{prop:link-2-b}
$\xi$ is a birational Mori contraction that contracts the divisor $\tilde B$
to a surface $S=V\cap\langle S\rangle$, and sends $\tilde A$ to a hyperplane
section $A$ of $V$ with $\Sing(A)=S$;

\item
\label{prop:link-2-c}
$S$ is either
a smooth rational cubic scroll, or a cone over a
rational normal twisted cubic curve;

\item
\label{prop:link-2-d}
\label{lemma-classification-singular-fibers-c}
$\xi$ has at most a finite number of two-dimensional fibers.
If $\tilde Y_1,\dots,\tilde Y_k$ are the two-dimensional fibers of $\xi$,
then
both $V$ and $S$ are smooth outside the finite set $\xi(\bigcup_i\tilde Y_i)$,
and $\xi|_{\tilde W\setminus\bigcup_i\tilde Y_i}:\tilde W\setminus\bigcup_i
\tilde Y_i\to
V\setminus\{\xi(\bigcup_i\tilde Y_i)\}$
is the blowup of $S\setminus\{\xi(\bigcup_i\tilde Y_i)\}$.
In particular, if $\xi$ has no two-dimensional fiber, then both $V$ and $S$ are
smooth,
and $\tilde W\to V$ is the blowup of $S$.
\end{enumerate}
\end{proposition}

\begin{proof}[Proof of Proposition~\xref{prop:link-2}]
The following two facts are borrowed in~\cite{Prokhorov-Zaidenberg-2015}.

\begin{sclaim} [{\cite[Lem.~5.5]{Prokhorov-Zaidenberg-2015}}]
On $\tilde W$ the following equalities hold:
\begin{equation}
\label{equation-intersectin-theory-W}
(H^*)^4=5,\quad (H^*)^3\cdot\tilde A=0,\quad
(H^*)^2\cdot\tilde A^2=-5,\quad H^*\cdot\tilde A^3=-8,\quad\tilde A^4=-6.
\end{equation}
\end{sclaim}

\begin{sclaim} [{\cite[Lem.~5.2]{Prokhorov-Zaidenberg-2015}}]
The linear system $|2H^*-\tilde A|$ is base point free and defines a morphism
$\Phi_{|2H^*-\tilde A|}:\tilde W\to\PP^{12}$.
\end{sclaim}

Letting
$V=\Phi_{|2H^*-\tilde A|}(\tilde W)\subset\PP^{12}$, consider the
Stein factorization
\begin{equation}
\label{eq:Stein}
\Phi_{|2H^*-\tilde A|}:\tilde W\overset{\xi}\longrightarrow U
\overset{\psi}\longrightarrow V\subset\PP^{12}.
\end{equation}
An easy computation using~\eqref{equation-intersectin-theory-W}
gives
\begin{equation*}
(2H^*-\tilde A)^4=18\quad\text{and}\quad (2H^*-\tilde A)^3\cdot\tilde
B=(2H^*-\tilde A)^3\cdot (H^*-\tilde A)=0.
\end{equation*}
Therefore, $\xi$ in~\eqref{eq:Stein} is birational and contracts $\tilde B$.
Furthermore,
\begin{equation}
\label{equation-degree}
\deg V\cdot\deg\psi=18.
\end{equation}
We have $\rk\Pic(\tilde W)=2$. Hence the Mori cone of $\tilde W$ is generated by
two extremal rays. These rays are generated by the nef divisors $2H^*-\tilde A$
and $H^*$. Their sum is an ample anticanonical divisor $-K_{\tilde W}$. So, the
morphism $\xi$ in~\eqref{eq:Stein} is a Mori contraction, and $\Pic (U)\cong
\ZZ\cdot L$, where $\xi^*L=2H^*-\tilde A$. One has
\begin{equation}
\label{equation-degree-S}
L^2\cdot S=-(2H^*-\tilde A)^2\cdot\tilde B^2=-(2H^*-\tilde A)^2\cdot
(H^*-\tilde A)^2=3.
\end{equation} Thus,
$S=\xi(\tilde A)$ is a surface in $\PP^{12}$ of degree $3$.
By the main theorem of~\cite{Andreatta1998a}, $U$ has at worst ordinary double
points as singularities.
Since $-K_{\tilde W}=2(2H^*-\tilde A)-\tilde B$, we have $-K_U=2L$, that is,
$U$ is a Fano-Mukai fourfold (possibly with ordinary double points).
Moreover, outside an at most finite union of two-dimensional fibers
of $\xi$, the morphism $\xi$ is a blowup of a smooth surface in a
smooth fourfold.
Consequently, the surface $S$ has at worst isolated singularities.

\begin{sclaim}
The morphism $\psi$ in~\eqref{eq:Stein} is birational.
\end{sclaim}

\begin{proof}
Suppose that $\deg\psi>1$. Then by~\eqref{equation-degree},
$V\subset\PP^{12}$ is a fourfold of degree $\le 9$.
By the del Pezzo-Bertini Theorem (see, e.g.,~\cite[Thm.~1]{Eisenbud1987}) we
have
$\deg V\ge\codim_{\PP^{12}} V+1=9$.
Thus, $\deg V=9$, $\deg\psi=2$, and $V$ is either a smooth scroll,
or a cone over a smooth scroll, or finally a cone with vertex a line over a 
rational
normal curve of
degree $9$ in $\PP^9$, see \emph{loc. cit.} In the first two cases, $\rk
\Cl(V)=2$. On the other hand,
there is a natural injection $\psi^*:\Cl(V)\hookrightarrow\Cl(U)\cong
\ZZ\cdot L$. This leads to a contradiction.
In the latter case $\Cl(U)=\ZZ\cdot P$, where $P$ is the class of a plane.
As before, this gives a contradiction.
\end{proof}

The following assertion is well known, even in a more general form. However,
due to the lack of
a reference, we provide an argument.

\begin{sclaim}
The morphism
$\psi$ is an isomorphism.
\end{sclaim}

\begin{proof}
It suffices to show that the graded algebra
\begin{equation*}
\operatorname{R}(U,L)=\bigoplus_{n\ge 0} H^0(U,\OOO_U(nL))
\end{equation*}
is generated by its component of degree $1$.
By our construction, the linear system $|L|$
is base point free. By Bertini's theorem,
a general member $X\in |L|$ is smooth.
Then $X$ is a smooth Fano threefold, and the anticanonical
linear system $|-K_X|=|\OOO_{X}(L)|$ is base point free.

By the Kawamata-Viehweg vanishing theorem, we have
$H^1(U,\OOO_U)=0$. Applying Lemma~2.9 in~\cite{Iskovskih1977a}, it is enough to
show that the graded algebra

\begin{equation*}
\operatorname{R}(X,L)=\bigoplus_{n\ge 0} H^0(X,\OOO_X(nL))
\end{equation*}
is generated by its component of degree $1$.
Continuing in the same manner, one arrives at a smooth K3 surface $Z\in |-K_X|$
and a smooth curve $C\in |L|_Z|$, where $K_C=L|_C$. Since the map given by
$|K_C|=|L|_C|$
is birational, $C$ is not hyperelliptic.
By a theorem of Max Noether (\cite[Ch.~2, \S~3]{Griffiths-Harris-1994}), the
algebra

\begin{equation*}
\operatorname{R}(C,K_C)=\bigoplus_{n\ge 0} H^0(C,\OOO_C(nK_C))
\end{equation*}
is generated by its component $H^0(C,\OOO_C(K_C))$.
Now the claim follows.
\end{proof}

Thus, we may suppose in the sequel that $V=U$
is a Fano-Mukai fourfold of degree $18$ (so, of genus $10$) with at worst
ordinary double points as singularities, as claimed in~\xref{prop:link-2-a}, and
$\xi=\Phi_{|2H^*-\tilde A|}$.

Since $L^*\sim 2H^*-\tilde A$ and $\tilde B\sim H^*-\tilde A$, one has $\tilde
A\sim L^*-2\tilde
B$ on $\tilde W$. Note that $\tilde
A\subset\tilde W$ in~\eqref{diagram-2} is smooth being a $\PP^1$-bundle over a
smooth surface $F$.
Letting $A=\xi(\tilde A)\subset V$, one deduces that $A$ is a hyperplane
section
of $V$ with desingularization $\tilde A$, and $S=\Sing(A)$. This concludes the
proof of~\xref{prop:link-2-b}.

To show~\xref{prop:link-2-c}, note that
$\deg S=3$ by~\eqref{equation-degree-S}, and $S$ has at
worst isolated singularities.
We claim that $\dim\langle S\rangle>3$.
Indeed,
using~\eqref{equation-intersectin-theory-W} one can compute
\begin{equation*}
L^*\cdot\tilde B^3=(2H^*-\tilde A)\cdot (H^*-\tilde A)^3=-1.
\end{equation*}
On the other hand, using~\cite[Lem.~2.3(ii)]{Prokhorov-Zaidenberg-4-Fano} we
obtain
\begin{equation*}
L^*\cdot\tilde B^3=-L|_{S}\cdot K_S+K_V\cdot L\cdot S\,
\end{equation*}
(the corresponding formula in \emph{loc. cit.} works since $S$ has at worst
isolated singularities, and $L$ is movable). So,
\begin{equation*}
L|_{S}\cdot K_S=-L^*\cdot\tilde B^3-2 L^2\cdot S=1-6=-5.
\end{equation*}
If $\dim\langle S\rangle <4$, then $S$ is a cubic surface in $\PP^3$. In this
case
$L|_{S}\cdot K_S=-K_S^2=-3$, a contradiction.
Therefore, $\dim\langle S\rangle=4$, and so, $S\subset\PP^4$ is a linearly
nondegenerate cubic
surface, that is, a
surface of minimal degree in $\PP^4$. Now the assertion follows by the del
Pezzo
theorem
(\cite[Thm.~1]{Eisenbud1987}).
This proves~\xref{prop:link-2-c}.
Finally,~\xref{prop:link-2-d} follows from~\cite[Main Theorem]{Andreatta1998a}.
This ends the proof of Proposition~\xref{prop:link-2}.
\end{proof}

Next we examine the alternative case, where $\xi$ in
\eqref{diagram-2} has a two-dimensional
fiber (cf.~\xref{prop:link-2}.\xref{prop:link-2-d}). We need such a simple 
observation.

\begin{remark}[cf. {\cite [Rem.~5.2.14]{Kuznetsov-Prokhorov-Shramov}}]
\label{rem:touching-conics}
A smooth conic $J\subset\Xi$
touches $\Upsilon$ with even multiplicities if and only if one of the following
holds \textup(see Figure
\xref{fig-touching-conics}\textup):
\begin{enumerate}
\renewcommand\labelenumii{\rm (\roman{enumii})}
\renewcommand\theenumii{\rm (\roman{enumii})}
\item
\label{lem:touching-conics-Gl2}
$J=\Upsilon$;

\item
\label{lem:touching-conics-Ga}
$J$ is tangent to $\Upsilon$ in a single point
with multiplicity $4$;

\item
\label{lem:touching-conics-Gm}
$J$ is tangent to $\Upsilon$ in two points.
\end{enumerate}
\end{remark}
\begin{figure}[h]
\begin{tabular}{l@{\hspace{70pt}}l@{\hspace{70pt}}l}
\xref{lem:touching-conics-Gl2}
\includegraphics[width=70pt,height=80pt]{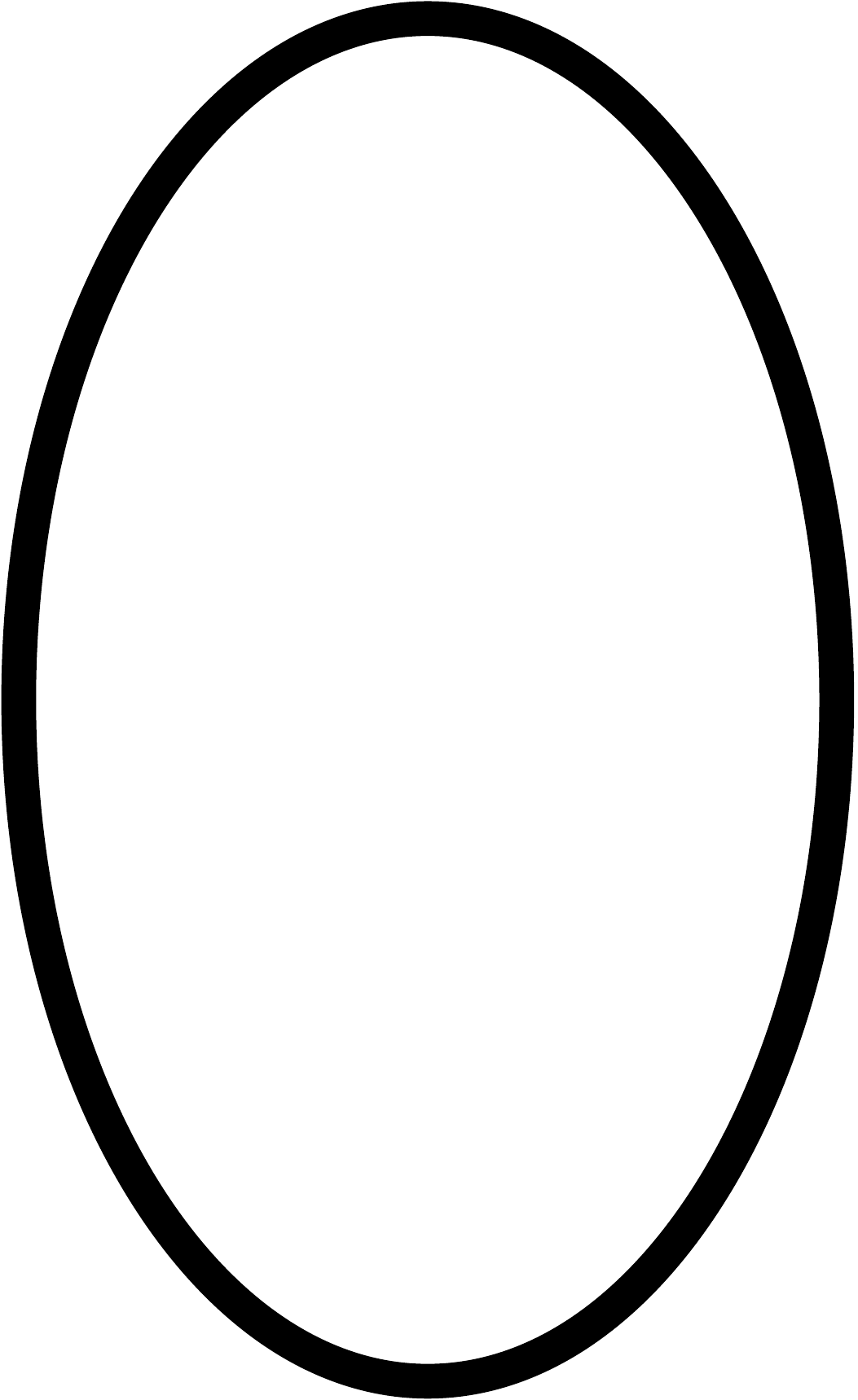} &
\xref{lem:touching-conics-Ga}
\includegraphics[width=70pt,height=80pt]{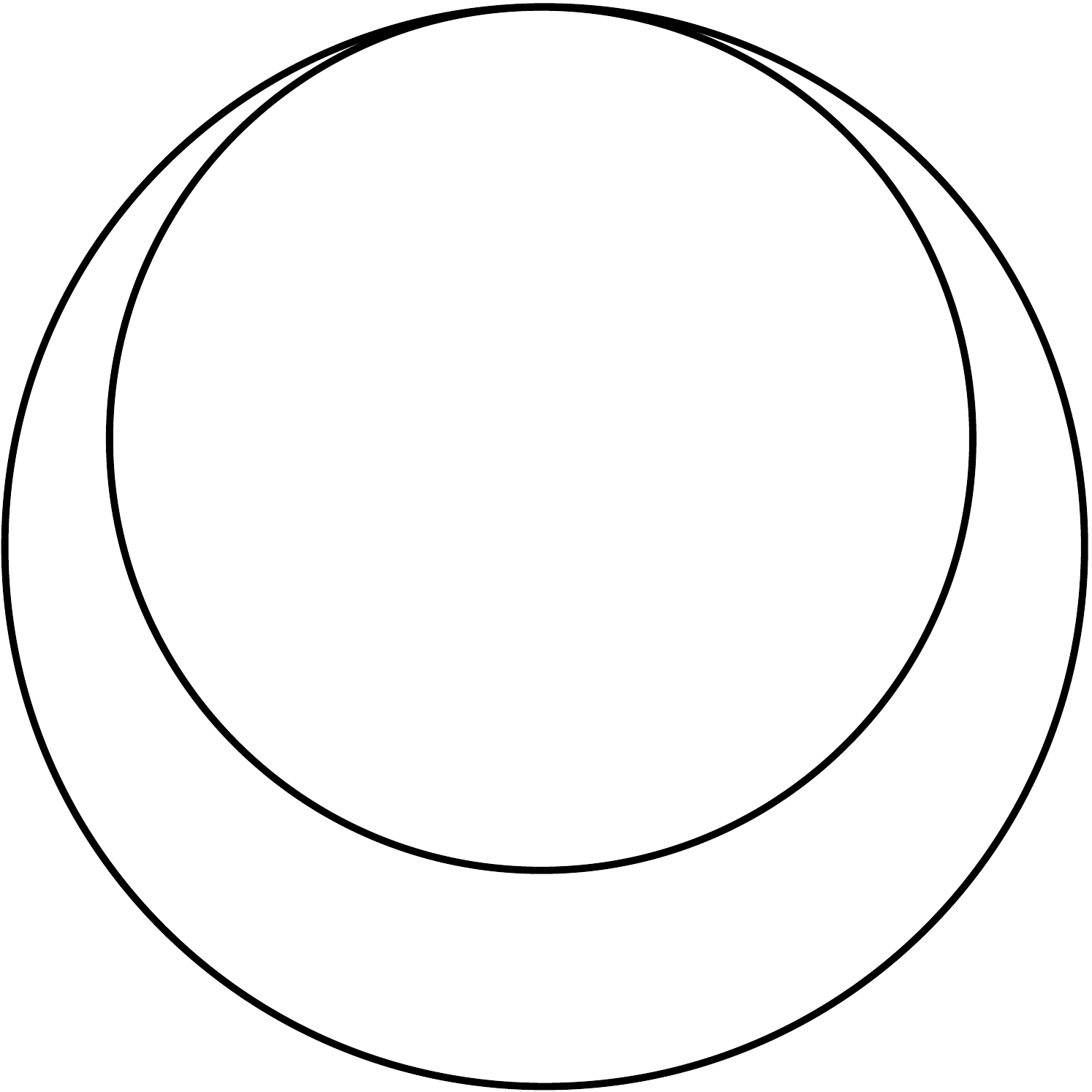} &
\xref{lem:touching-conics-Gm}
\includegraphics[width=70pt,height=80pt]{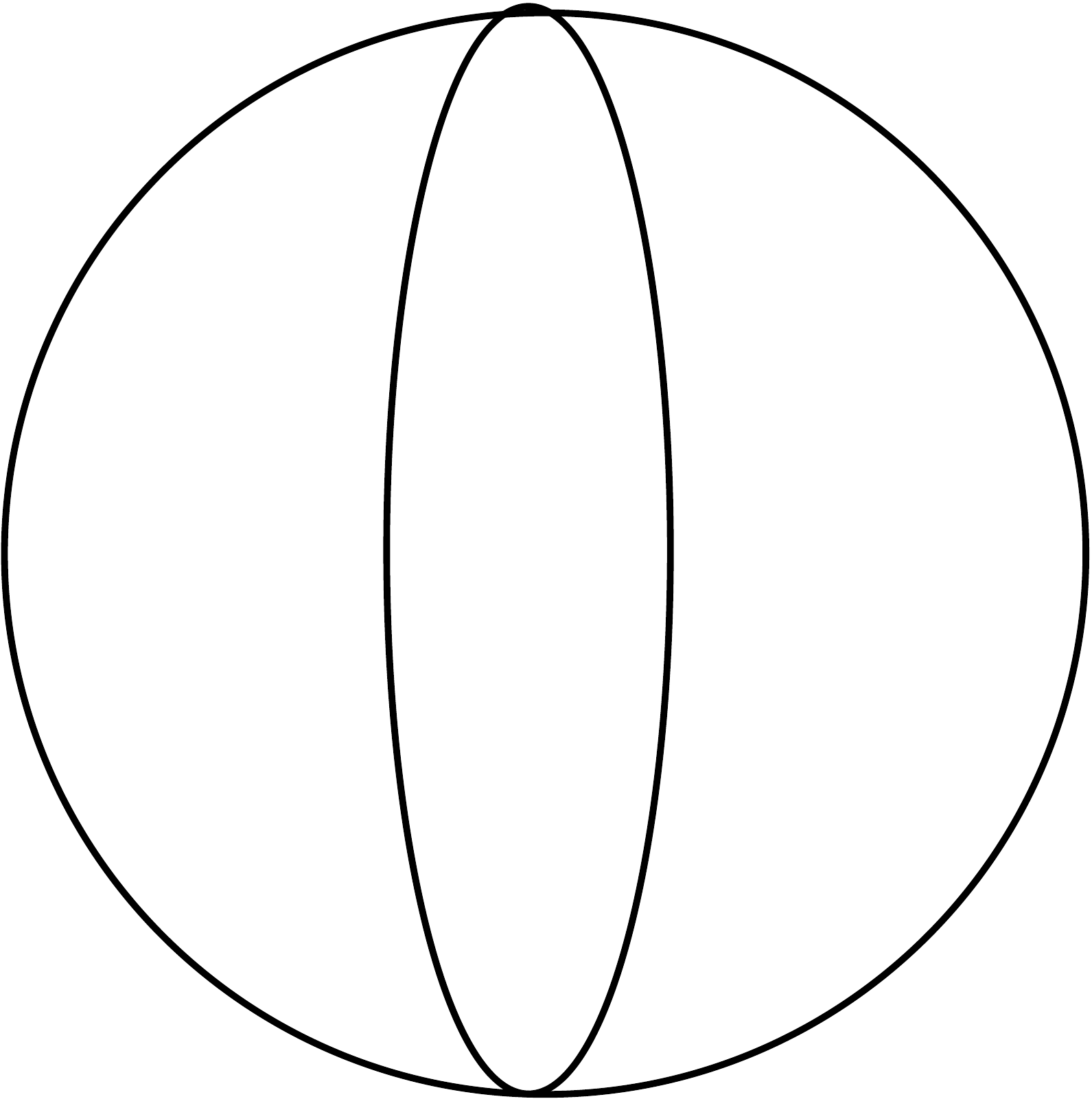}
\end{tabular}
\caption{}
\label{fig-touching-conics}
\end{figure}

\begin{proposition}
\label{lemma-classification-singular-fibers}
With the assumptions and notation of~\xref{nota:quintic} and
\xref{prop:link-2}, suppose 
that $\xi$ has a two-dimensional fiber
$\tilde Y$. Then the following hold.
\begin{enumerate}
\renewcommand\labelenumi{\rm (\alph{enumi})}
\renewcommand\theenumi{\rm (\alph{enumi})}
\item
\label{lemma-classification-singular-fibers-a}
$Y:=\eta(\tilde Y)\subset W$ is a plane, and $Y\cap F$
is a reduced (but possibly reducible) conic.
Conversely, the proper transform
$\tilde Y\subset\tilde W$ of any plane $Y\subset W$ such that $Y\cap F$ is a
conic,
is a two-dimensional fiber of $\xi$.

\item
\label{lemma-classification-singular-fibers-F1}
If $F\cong\FF_1$, then
$\tilde Y$ is a unique
two-dimensional fiber of $\xi$, and $Y\cap F$ is a smooth conic, the
exceptional
section
of $F\cong\FF_1$.

\item
\label{lemma-classification-singular-fibers-d}
One of the following holds.

\begin{enumerate}
\renewcommand\labelenumii{\rm (\roman{enumii})}
\renewcommand\theenumii{\rm (\roman{enumii})}
\item
\label{lemma-classification-singular-fibers-smooth}
$B$ is singular at a general point of $Y:=\eta(\tilde Y)$, $\xi(\tilde Y)$ is
a smooth point
of $V$ and a (unique) singular point of the surface $S$. This point $\xi(\tilde
Y)$ is of type $\frac13 (1,1)$ in $S$.
Furthermore, $B=R$ and $Y=\Xi$.

\item
\label{lemma-classification-singular-fibers-singular}
$B$ is smooth at a general point of $Y$, $\xi(\tilde Y)\in V$ is an
ordinary double point,
and the surface $S$ is smooth at $\xi(\tilde Y)$.
Furthermore, $Y$ is a plane of type $\sigma_{3,1}$ \textup(that is, a
$\Pi$-plane\textup).
\end{enumerate}

\item
\label{lemma-classification-singular-fibers-toric}
If $V$ is smooth, then $F\cong\FF_1$, $\tilde Y$ is the only
two-dimensional fiber of $\xi$, $B=R$, $Y=\Xi$, $S$ is a cone over a twisted
cubic curve, and the exceptional
section of the ruling $F\to\PP^1$ is a smooth conic
$J:=\Xi\cap F$ of one of types
\xref{lem:touching-conics-Gl2}--\xref{lem:touching-conics-Gm} in 
Remark~\xref{rem:touching-conics}, see Figure~\xref{fig-touching-conics}.
\end{enumerate}
\end{proposition}

\begin{proof}
By Proposition~\xref{prop:link-2}\xref{prop:link-2-c},
$S$ is normal and possesses at most one singular point. It follows from the
main
theorem and Proposition 4.11 in~\cite{Andreatta1998a} that
$Y\cong\PP^2$ and
\begin{equation}
\label{eq:1}
\OOO_{\tilde Y}(3H^*-\tilde A)\cong\OOO_{\tilde Y}(-K_{\tilde W})\cong
\OOO_{\PP^2}(1).
\end{equation}
Since ${\tilde Y}$ is contracted to a point under $\xi=\Phi_{|2H^* -\tilde
A|}$, one has
\begin{equation}
\label{eq:2}
\OOO_{\tilde Y}(2H^* -\tilde A)=\OOO_{\tilde Y}.
\end{equation}
Then~\eqref{eq:1} and~\eqref{eq:2} imply
\begin{equation}
\label{eq:3}
\OOO_{\tilde Y}(H^*)\cong\OOO_{\PP^2} (1)\quad\text{and}\quad
\OOO_{\tilde Y}(\tilde A)\cong\OOO_{\PP^2}(2).
\end{equation}
Thus, $Y\subset W$ is a plane, and $Y\cap F$ is a conic.
The converse statement is left to the reader.

Assume that $Y\cap F=2\Lambda$ is a double line.
If $F\cong\FF_3$, then there exists another line $\Lambda'\subset F$ meeting
$\Lambda$.
The plane $Y$ coincides with the tangent plane to $F$ at the point $\Lambda\cap
\Lambda'$.
Hence, $Y\cap F=2\Lambda+\Lambda'\neq 2\Lambda$.

Thus, we may assume that
$F\cong\FF_1$ and $\Lambda$ is a ruling of $F$.
Let $J$ be the negative section of $F=\FF_1\to\PP^1$,
and let $J'$ be a section disjoint with $J$.
Then $J$ is a conic, $J'$ is a twisted
cubic in $F$, and the linear spans
$\langle J\rangle$ and $\langle J'\rangle$ are disjoint.
The ruling $\Lambda$ meets $J$ in a point, say, $P_0$, and $J'$ in
$P_1$.
By our assumption, the plane $Y$ is tangent to $F$ along $\Lambda$. Hence $Y$
contains the
tangent lines $T_{P_0}J$ and $T_{P_1}J'$. It follows that these lines
intersect, and then also the linear spans $\langle J\rangle$ and $\langle
J'\rangle$ do, a contradiction
This proves~\xref{lemma-classification-singular-fibers-a}.

\xref{lemma-classification-singular-fibers-F1}
In this case the only conic on $F$ is the negative section of $F\cong\FF_1$.

To show~\xref{lemma-classification-singular-fibers-d}
we note that $\eta$ provides a local analytic isomorphism of pairs $(\tilde B,
\tilde Y)$ and $(B,Y)$
at the corresponding generic points of $\tilde Y$ and $Y$, respectively.
The statements of~\xref{lemma-classification-singular-fibers-d}
follow now directly from the main theorem in
\cite{Andreatta1998a} and the proof of Proposition 6.3 in \emph{loc. cit.},
except for the equalities $B=R$ and $Y=\Xi$ 
in~\xref{lemma-classification-singular-fibers-smooth}.
Let us show the latter equalities. Since
the variety $B$ in~\xref{lemma-classification-singular-fibers-smooth}
is a hyperplane section of $W$ singular along a plane $Y$ (see Definition
\xref{nota:quintic}), $B$ contains any line in $W$ meeting $\Sing(B)\supset Y$.
Any two planes in $W$ intersect. In particular, any plane in $W$ meets $Y$.
Hence
$B$ contains any plane in $W$. Consequently, $B$
coincides with the union $R=\bigcup_{\gamma\in\Upsilon}\Pi_\gamma$, see
\xref{def:skew-symmetric}.
It follows finally that $Y=\Sing(R)=\Xi$ and $R=B\supset F$.

\xref{lemma-classification-singular-fibers-toric}
Since $V$ is smooth,~\xref{lemma-classification-singular-fibers-smooth} holds,
and so,
$B=R$, $Y=\Xi$, and $S$ is the cone over a twisted cubic curve.
Since $Y=\Xi$, then according to~\xref{lemma-classification-singular-fibers-a}, 
$F\cap\Xi$ is a conic. This conic contains the exceptional section, say, $J$ of 
$F\cong\FF_n$, $n\in\{1,3\}$. Thus, $J\subset\Xi$.
For $P\in J$, let $\Lambda_P\subset F$ be the ruling
through $P$.

Assume first that $F\cong\FF_3$.
Then $F\cap\Xi=J+\Lambda_0$, where $\Lambda_0$ is the ruling
through the point $P_0=J\cap\Lambda_0$.
For a point $P\in J\setminus\{P_0\}$ one has $\Lambda_P\neq\Lambda_0$,
and so, $\Lambda_P\not\subset\Xi$.
Hence $\Lambda_P$ is contained in a (unique) plane $\Pi_\gamma$.
This defines a (regular) map
\begin{equation*}
\psi: J\setminus\{P_0\}\longrightarrow\Upsilon,
\qquad P\longmapsto\gamma.
\end{equation*}
Since each plane $\Pi_\gamma$ contains a unique ruling $\Lambda_P$,
this map extends to a bijection $J\to\Upsilon$. The intersection 
$l_{\gamma}:=\Pi_{\gamma}\cap\Xi$
is the tangent line to $\Upsilon$ at $\gamma$.
Then $\psi^{-1}$ is defined as follows
\begin{equation*}
\psi^{-1}:\Upsilon\longrightarrow J,\qquad
\gamma\longmapsto P=J\cap l_\gamma.
\end{equation*}
This map is bijective if and only if $J$ is a tangent line to $\Upsilon$.
Thus $J=l_{\gamma_0}$ for some $\gamma_0\in\Upsilon$.
If $\Lambda_{\gamma_0}\not\subset\Xi$, then $\Pi_{\gamma_0}\cap
F=J+\Lambda_{\gamma_0}$.
By~\xref{lemma-classification-singular-fibers-singular} $V$ is singular, a
contradiction.
Hence $\Lambda_{\gamma_0}\subset\Xi$ and $F\cap\Xi=J+\Lambda_{\gamma_0}$.
In this case, $\Lambda_{\gamma_0}\neq J=l_{\gamma_0}$, and so,
$\Lambda_{\gamma_0}\cap\Upsilon=\{\gamma_0,\,\gamma_1\}$, where $\gamma_1\neq
\gamma_0$.
On the other hand, $\Lambda_{\gamma_0}$ is contained in some plane
$\Pi_{\gamma}$,
because this is true for a general $\Lambda_{\gamma}$.
Hence $\Lambda_{\gamma_0}=\Xi\cap\Pi_\gamma$ is tangent to $\Upsilon$,
a contradiction. Thus, $F\cong\mathbb{F}_1$, as stated.

Assume now that $J\neq\Upsilon$. Define a correspondence between
$\Upsilon$ and $J$ via
\begin{equation*}
Z=\{(P,\gamma)\in J\times\Upsilon\mid P\in l_{\gamma}\},
\end{equation*}
where $l_{\gamma}$ is the line on $\Xi$ tangent to $\Upsilon$ at $\gamma$.
Clearly, $Z$ is a curve of bidegree $(2,2)$ on $J\times\Upsilon
\cong\PP^1\times
\PP^1$.

The curve $Z$ admits a natural interpretation as a subvariety of the variety
$\Fl(\PP^2)\subset\PP^2\times {\PP^2}^\vee$
of full flags on $\PP^2$. Under this interpretation
one has $p_1(Z)=J$ and
$p_2(Z)=\Upsilon^\vee$,
where $\Upsilon^\vee\subset {\PP^2}^\vee$ is the dual conic and $p_1,p_2$ are
the canonical projections of the product $\PP^2\times {\PP^2}^\vee$ to the
factors.
Moreover, $Z=p_1^{-1}(J)\cap p_2^{-1}(\Upsilon^\vee)$.
Note that $p_2^{-1}(\Upsilon^\vee)\cong\PP^1\times\PP^1$. The restriction
$\pi:=p_1|_{p_2^{-1}(\Upsilon^\vee)}: p_2^{-1}(\Upsilon^\vee)\to\PP^2$
is a double cover branched along $\Upsilon$.

Consider a generically one-to-one map $\delta: J\to\Upsilon$,
\begin{equation*}
P\in J\longmapsto\Lambda_P\longmapsto\Pi_P\longmapsto l_P=\Xi\cap\Pi_P
\longmapsto\gamma=l_P\cap\Upsilon\in\Upsilon,
\end{equation*}
where $\Pi_P\subset R$ is the plane containing the ruling $\Lambda_P\subset F$.
The graph of $\delta$ in $J\times\Upsilon$ is a component of $Z$.
Therefore, $Z$ splits into two components
$Z_1$ and $Z_2$, which are
curves of bidegree $(1,1)$ on $\PP^1\times\PP^1$. Also the preimage
$\pi^{-1}(J)=p_1^{-1}(J)\cap
p_2^{-1}(\Upsilon^\vee)$ splits into two irreducible components
$Z_1$ and $Z_2$. This is possible only if $J$ touches the branch locus
$\Upsilon$ with even multiplicities, hence only in the cases
\xref{lem:touching-conics-Ga} and
\xref{lem:touching-conics-Gm} in~\xref{rem:touching-conics}.
\end{proof}

\begin{mdefinition}
\label{def:cubic-cones}
Following~\cite{KapustkaRanestad2013} we call 'cubic scrolls` both smooth cubic
scrolls and cones over rational twisted cubic curves. The latter cones in 
$V_{18}$ will be called
'cubic cones` for short. This does not lead to a confusion. Indeed, since $V$ is 
an intersection of
quadrics (\cite[Lem.~2.10]{Iskovskih1977a}), $V$ does not contain any cubic
surface $F$ with $\langle F\rangle\cong\PP^3$.
\end{mdefinition}

From Propositions~\xref{prop:reversion},~\xref{prop:link-2}, and
\xref{lemma-classification-singular-fibers} we deduce the following corollaries.

\begin{scorollary}
\label{cor:diagr-2}
The cubic scroll $S\subset V$ in diagram
\eqref{diagram-2} is a cubic cone if and only if
$V$ is smooth, $F\subset R$, $F\cong\FF_1$, and the exceptional section
$J=F\cap\Xi$ of $F\to\PP^1$ is a smooth conic touching $\Upsilon$ with even 
multiplicities.
There are three types
\xref{rem:touching-conics} \xref{lem:touching-conics-Gl2},~\xref{rem:touching-conics}\xref{lem:touching-conics-Ga}, and
\xref{rem:touching-conics}\xref{lem:touching-conics-Gm}
of such pairs $(F,J)$.
\end{scorollary}

The next corollary will be used in the proof of Theorem~\xref{thm:main}; cf.\ 
also Lemma~\xref{lemma--SS}.

\begin{scorollary}
\label{rem:interrompu}
For any cubic cone $S\subset V$
there is a unique hyperplane section
$A$ of $V$ with $\Sing(A)=S$ such that $V\setminus A\cong\CC^4$.
\end{scorollary}

\begin{proof} By Proposition~\xref{prop:reversion}, $V$ and $S$ can be included
in diagram~\eqref{diagram-2}, where $A$ is a hyperplane section of $V$ with
$\Sing(A)=S$. Such a divisor $A$ is unique being the image of the 
$\eta$-exceptional divisor $\tilde A\sim L^*-2\tilde B$, 
see~\eqref{eq:inters-ALB}.
Since $V$ is smooth and $S$ is singular, the cases
\xref{prop:link-2-d} and~\xref{lemma-classification-singular-fibers-singular}
in Propositions~\xref{prop:link-2} and
\xref{lemma-classification-singular-fibers}, respectively, are excluded. So,
there is a unique two-dimensional fiber $\tilde Y$ of $\xi$ with
$\Sing\,S=\xi(\tilde Y)$, and we are in case
\xref{lemma-classification-singular-fibers-smooth} of Proposition
\xref{lemma-classification-singular-fibers}. Due to this proposition, we have
\begin{equation*}
\eta(\tilde Y)=\Xi,\quad B=R,\quad\mbox{and so,}\quad F\subset R .
\end{equation*}
As follows from diagram~\eqref{diagram-2}, there are isomorphisms
\begin{equation*}
V\setminus A\cong
\tilde W\setminus (\tilde A\cup\tilde R)\cong
W\setminus
R\cong\CC^4,
\end{equation*} see Corollary~\xref{cor:2.2.2}.
Thus, the pair $(V,A)$ yields a compactification of $\CC^4$ into a smooth
Fano-Mukai fourfold $V$ of genus $10$.
\end{proof}

Finally, we introduce the following notion.

\begin{definition}
\label{def:linked-pairs}
We say that the pairs $(W,F)$ and $(V,S)$ are \emph{linked} if they fit in 
diagram~\eqref{diagram-2} and verify the conditions of one of 
Propositions~\xref{prop:link-2} and~\xref{lemma-classification-singular-fibers}.

In the sequel we deal only with linked pairs that verify the conditions of 
Corollary~\xref{cor:diagr-2}.
\end{definition}

\section{Automorphisms of $W_5$}
\label{sec-2bis}

Let us introduce the
following notation.

\begin{notation}
\label{nota:3-forms}
Consider the linear space
$M_3=S^3{\CC^2}^\vee$ of binary cubic forms $f_3(x,y)$
and the projectivization $\PP(M_3\oplus\CC)\cong\PP^4$.
A point of $\PP^4$
can be viewed as a class $[(f_3,c)]$, where $f_3\in M_3$, $c\in\CC$, and
$(f_3,c)\neq (0,0)$.
Up to automorphisms of $\PP^4$,
a twisted cubic curve $\Gamma\subset\PP^4$ can be represented as the 
projectivization
\begin{equation*}
\Gamma=\PP\left(\Bigl\{(f_3,0)\in M_3\oplus\CC\setminus\{(0,0)\}\mid f_3
=(\alpha x+\beta y)^3\Bigr\}\right)\subset
\PP(M_3\oplus\{0\})\subset\PP(M_3\oplus\CC).
\end{equation*}
So, the linear span $\langle\Gamma\rangle$ is identified with the linear 
subspace $\{c=0\}$ in $\PP^4$, and the point $P=\PP(\{0\}\oplus 
\CC)\in\PP^4\setminus\langle\Gamma\rangle$ with the class $[(0,1)]$.

The standard representation of $\GL_2(\CC)$ on $\CC^2$ induces an irreducible
representation of $\GL_2(\CC)$ on
$M_3$, $(g,f_3)\longmapsto f_3\circ g^{-1}$. Adding the trivial one-dimensional
representation yields a representation of $\GL_2(\CC)$ on $M_3\oplus\CC$ and,
in
turn,
a $\GL_2(\CC)$-action on $\PP^4$,
which fixes the point $P$ and stabilizes the twisted cubic
$\Gamma$.

Let further $\Delta\subset\PP(M_3\oplus\CC)$
be the closure of the locus of points
$[(f_3,c)]\in\PP(M_3\oplus\CC)$, whose binary cubic form $f_3$
has a multiple factor.
\end{notation}

\begin{remarks}
\label{rem:GL2-action}
\setenumerate[0]{leftmargin=8pt,itemindent=7pt}
\begin{enumerate}
\renewcommand\labelenumi{\rm\arabic{enumi}.}
\renewcommand\theenumi{\rm\arabic{enumi}}
\item
\label{rem:GL2-action-1}
Clearly, $\Delta$ is the cone over
$\Delta_0:=\Delta\cap\PP(M_3)$ with vertex $P$.
The surface $\Delta_0$ in
$\PP^3=\PP(M_3)$ is the zero divisor of
the discriminant $\operatorname{discr}(f_3)$.
This quartic surface has cuspidal singularities along $\Gamma$
and is smooth outside $\Gamma$. In fact,
$\Delta_0$ is
the tangent developable surface of $\Gamma$, that is,
the closure of the union of the tangent lines
to $\Gamma$. The cone $N$ over $\Gamma$ with vertex $P$ is the singular locus
of $\Delta$.

\item
\label{rem:GL2-action-2}
The $\GL_2(\CC)$-action on
$\PP(M_3\oplus\CC)=\PP^4$ has exactly 7 orbits, namely,
\begin{equation*}
\newcommand\probel{\hspace{5pt}}
\{P\},\probel
\Gamma,\probel
\Delta_0\setminus\Gamma,\probel
\PP(M_3)\setminus\Delta_0,\probel
N\setminus (\Gamma\cup\{P\}),\probel
\Delta\setminus (N\cup\Delta_0),\probel
\PP(M_3\oplus\CC)\setminus (\PP(M_3)\cup\Delta).
\end{equation*}
Thus, $P$ is the unique fixed point of the
$\GL_2(\CC)$-action on $\PP^4$, while the last orbit in this list is the open 
orbit.

\item
\label{rem:GL2-action-3}
The center $\Gamma$ of the blowup $\varphi$ in~\eqref{equation-diagram-1} is
invariant under the $\GL_2(\CC)$-action on $\PP^4$ defined 
in~\xref{nota:3-forms}.
Hence this action
lifts to $\hat W$ stabilizing the exceptional divisors $\hat R$ and $\hat E$ of
$\varphi$ and $\rho$, respectively, see Proposition~\xref{prop:link-1}.
The $\GL_2(\CC)$-action on $\hat W$ induces an effective $\GL_2(\CC)$-action on
$W$
via the contraction $\rho:\hat W\to W$ of the $\GL_2(\CC)$-invariant divisor
$\hat E$. Then the birational linear projection $\phi:
W\dashrightarrow\PP^4$ with center $\Xi$ along with diagram
\eqref{equation-diagram-1} are $\GL_2(\CC)$-equivariant.
The $\GL_2(\CC)$-action on
$W$ extends to a linear $\GL_2(\CC)$-action on $\PP^7\supset W$. The latter
action has $W$ as an orbit closure, stabilizes $R$, $\Xi$, and $\Upsilon$, and
induces a linear $\GL_2(\CC)$-action on
$\CC^4=\PP^4\setminus\PP^3\cong W\setminus R $, see~\eqref{eq:2.2.2}, and a
standard $\PGL_2(\CC)$-action on $\Xi$ with a unique closed orbit $\Upsilon$.
\end{enumerate}
\end{remarks}

Using
\xref{nota:3-forms} --
\xref{rem:GL2-action} we describe in
\xref{lemma-GL2-action-1} --~\xref{proposition-GL2-action-c}
the automorphism
group $\Aut(W)$ and its action on $W$.

\begin{lemma}
\label{lemma-GL2-action-1}
Let $\Gamma\subset\PP^4$ be a rational twisted cubic curve, and
$P\in\PP^4\setminus\langle\Gamma\rangle$ be a point. Then there are
isomorphisms
\begin{equation}
\label{eq:GL2}\Aut(\PP^4,\Gamma, P)\cong\GL_2(\CC)/\mumu_3\cong
\GL_2(\CC),
\end{equation}
where the cyclic group $\mumu_3$ of order $3$ is realized as the subgroup of
scalar matrices
$\{\zeta\cdot\id_{\CC^2}\}$ with $\zeta^3=1$.The stabilizer in $\Aut(\PP^4,\Gamma, P)$ of a general point $Q\in \PP^4$ is trivial.
\end{lemma}

\begin{proof} We leave to the reader to check 
the fact that the
endomorphism $\GL_2(\CC)\to
\GL_2(\CC)$, $A\longmapsto (\det A)\cdot A$, yields the second isomorphism in
\eqref{eq:GL2}.
Up to an automorphism of $\PP^4$ one may suppose that the triple
$(\PP^4,\Gamma, P)$ is chosen as in~\xref{nota:3-forms}. The
$\GL_2(\CC)$-action
on $\PP^4$ introduced in
\xref{nota:3-forms} fixes $P$ and stabilizes $\Gamma$.
The scalar matrices $\zeta\cdot\id\in\GL_2(\CC)$ with $\zeta^3=1$, and only 
these,
act identically on $\PP^4$. This gives an
embedding $\GL_2(\CC)/\mumu_3\hookrightarrow
\Aut(\PP^4,\Gamma, P)$. In fact, this embedding is an isomorphism.
The latter follows by comparing the exact sequences
$$1\longrightarrow\Gm\longrightarrow\Aut(\PP^4,\Gamma, P)
\longrightarrow
\Aut(\Gamma)=\PGL_2(\CC)\longrightarrow 1\,$$
and
$$1\longrightarrow\Gm=\z(\GL_2(\CC)/\mumu_3) \longrightarrow\Aut(\PP^4,\Gamma, P)
\longrightarrow
\Aut(\Gamma)=\PGL_2(\CC)\longrightarrow 1\,,$$
where $\z(G)$ stands for the center of a group $G$. The last assertion is immediate.
\end{proof}

We use below the following fact.

\begin{lemma}\label{lem:PvdV} {\rm (Piontkowski-Van-de-Ven \cite[Thm.~6.6]{Piontkowski-Van-de-Ven-1999})}
There is an exact sequence 
\begin{equation}
\label{eq:2nd-sequence-1}
1\longrightarrow (\Ga)^4\rtimes\Gm
\longrightarrow
\Aut(W)\stackrel{\varrho}{\longrightarrow}
\Aut(\Upsilon)=\PGL_2(\CC)\longrightarrow 1\,.
\end{equation}
Therefore, $\Aut(W)$ is a connected algebraic group.
\end{lemma}

In
the next proposition we describe the algebraic Levi decomposition of $\Aut(W)$.

\begin{proposition}
\label{proposition-GL2-action}
\begin{enumerate}
\renewcommand\labelenumi{{\rm (\alph{enumi})}}
\renewcommand\theenumi{{\rm (\alph{enumi})}}
\item
\label{proposition-GL2-action-a}
Let $\Ru=\Ru(\Aut(W))$ be the unipotent
radical of $\Aut(W)$ and $\Levi$ its reductive Levi subgroup.
Then $\Ru\cong (\Ga)^4$ and $\Levi\cong\GL_2(\CC)$.
Therefore,
\begin{equation}
\label{eq:Levi-decomp}
\Aut(W)=\Ru\rtimes\Levi\cong (\Ga)^4\rtimes\GL_2(\CC).
\end{equation}
Furthermore, a $\GL_2(\CC)$-subgroup of $\Aut(W)$ is unique up
to conjugation.

\item
\label{proposition-GL2-action-b}
An isomorphism $\Aut(W)\stackrel{\cong}{\longrightarrow}\Aut(\PP^4,\Gamma)$ as 
in Corollary
\xref{cor:Aut-W5} sends $\Ru$ onto the vector group
$\Transl(\CC^4)\cong (\Ga)^4$ of the vector space
$\CC^4=\PP^4\setminus\langle\Gamma\rangle$, and the reductive Levi subgroup 
$\Levi$
onto
the stabilizer
$\Aut(\PP^4,\Gamma, P)\cong\GL_2(\CC)$ of a point 
$P\in\PP^4\setminus\langle\Gamma\rangle$.

\item
\label{proposition-GL2-action-cc}
The set of reductive Levi subgroups of $\Aut(W)$ coincides with the set of 
stabilizers of points in $W\setminus R$. Any Levi subgroup $\Levi$ of $\Aut(W)$ 
has a unique fixed point in $W\setminus R$, which is the unique fixed point in 
$W\setminus R$ of its center $\z(\Levi)$, and acts on $W\setminus R$ with a principal open orbit.
\end{enumerate}
\end{proposition}

\begin{proof}
\xref{proposition-GL2-action-a}
It follows from \eqref{eq:2nd-sequence-1} that $\Ru$ coincides with the $(\Ga)^4$-subgroup of
$\Aut(W)$. The effective $\GL_2(\CC)$-action on $W$ as in 
Remark~\xref{rem:GL2-action}.\xref{rem:GL2-action-3} defines a 
$\GL_2(\CC)$-subgroup,
say, $\Levi_0\subset\Aut(W)$
with $\Levi_0\cap\Ru=\{e\}$.
Thus, $\Levi_0\cong\GL_2(\CC)$ surjects onto $\PGL_2(\CC)$ with kernel
being the center $\z(\Levi_0)\cong\Gm$. Letting $\Levi=\Levi_0$
we obtain~\eqref{eq:Levi-decomp}.
The last assertion of~\xref{proposition-GL2-action-a} follows from the
Levi-Maltsev-Mostow Theorem, see, e.g., \cite{Mostow1956} or \cite{Hochschild1971}.

\xref{proposition-GL2-action-b}
By~\xref{proposition-GL2-action-a}, the isomorphism 
$\Aut(W)\cong\Aut(\PP^4,\Gamma)$ of Corollary
\xref{cor:Aut-W5} induces an isomorphism of Levi decompositions
\begin{equation*}
\Ru\rtimes\Levi\cong 
(\Ga)^4\rtimes\GL_2(\CC)\cong\Transl(\CC^4)\rtimes\Aut(\PP^4,\Gamma,P).
\end{equation*} Now~\xref{proposition-GL2-action-b} is straightforward.

\xref{proposition-GL2-action-cc} By virtue of~\xref{proposition-GL2-action-b}, 
$W\setminus R$ is the open orbit of $\Ru$ and of $\Aut(W)$. Furthermore, 
the unipotent radical $\Ru\cong (\Ga)^4$ acts freely (by translations) on 
$W\setminus R\cong\PP^4\setminus\langle\Gamma\rangle\cong\CC^4$. The Levi 
subgroup $\Levi\cong\GL_2(\CC)$ goes onto the stabilizer $\Aut(\PP^4,\Gamma, 
P)$ 
of a point $P\in\CC^4=\PP^4\setminus\langle\Gamma\rangle$, hence coincides 
with the stabilizer of the corresponding point in $W\setminus R$. All such 
stabilizers are conjugated via the action of $\Ru\cong\Transl(\CC^4)$ on 
$W\setminus R\cong\CC^4$. 
The last assertion is straightforward by virtue of Lemma \ref{lemma-GL2-action-1}.
\end{proof}

\begin{sremarks}
\label{rem:Mostow}
\setenumerate[0]{leftmargin=8pt,itemindent=7pt}
\begin{enumerate}
\renewcommand\labelenumi{\rm\arabic{enumi}.}
\renewcommand\theenumi{\rm\arabic{enumi}}
\item
\label{rem:Mostow-1}
Under the identification of the triple $(\PP^4,\Gamma, P)$ with
$(\PP(M_3\oplus\CC),\Gamma, P)$, see~\xref{nota:3-forms}, the
$\GL_2(\CC)$-subgroup $\Aut(\PP^4,\Gamma,P)\subset\Aut(\PP^4,\Gamma)$ is
defined as in~\xref{nota:3-forms} and~\xref{lemma-GL2-action-1}. The
$\GL_2(\CC)$-action by
conjugation on $\Ru=\Transl(\CC^4)\cong
(\Ga)^4$
is then given via the standard irreducible representation of $\GL_2(\CC)$ on
the
vector space $M_3\cong\CC^4$ of binary cubic forms. This determines the
structure of $\Aut(W)$ as a semidirect product. Identifying $\Ru=\Transl(\CC^4)$ 
with the additive group $(M_3,+)$, the action of 
$\Aut(W)\cong\Aut(\PP^4,\Gamma)$ on $\PP^4=\PP(M_3\oplus\CC)$ lifts to the 
action on $M_3\oplus\CC$ via
\begin{equation}
\label{eq:action-explicite}
(M_3,+)\rtimes\GL_2(\CC)\ni (h,g): (f,z)\longmapsto (f\circ g^{-1}+zh, z)\in 
M_3\oplus\CC,
\end{equation}
where $g$ is defined modulo multiplication by a cubic root of unity.

\item
\label{rem:Mostow-2}
A torus $T$ in a connected algebraic group $G$ is called \emph{regular} if the 
centralizer $\mathcal{C}_G(T)$ is solvable, and \emph{singular} otherwise. A 
torus $T\subset G$ is regular if and only if it is contained just in a finite 
set of Borel subgroups of $G$. The maximal tori are regular; see, e.g., 
\cite[Ch.~IX, \S~24]{Humphreys1975}.

By the preceding remark we have
$\mathcal{C}_{\Aut(W)}(\z(\Levi))=\Levi$ and $\mathcal{C}_{\Aut(W)}(\Ru)=\Ru$.
Hence the center $\z(\Levi)$ of a Levi subgroup $\Levi\subset\Aut(W)$ uniquely 
determines $\Levi$ and is a singular torus. Using~\eqref{eq:Levi-decomp} it is 
easily seen that, conversely, any singular torus in $\Aut(W)$ is the center 
$\z(\Levi)$ of a Levi subgroup $\Levi\subset\Aut(W)$. Any two singular tori in 
$\Aut(W)$ are conjugated.

\item
\label{rem:Mostow-3} 
The orbits of the $\Aut(\PP^4,\Gamma)$-action on $\PP^4$ are 
(see Remark~\xref{rem:GL2-action}.\xref{rem:GL2-action-2})
\begin{equation}
\label{eq:4-orbits}
\newcommand\probel{\hspace{5pt}}
\Gamma,\qquad
\Delta_0\setminus\Gamma,\qquad
\langle\Gamma\rangle\setminus\Delta_0,\qquad
\PP^4\setminus\langle\Gamma\rangle.
\end{equation}
Due to~\cite[Thm.~6.9]{Piontkowski-Van-de-Ven-1999}, the $\Aut(W)$-action on
$W$ has as well
exactly four orbits of dimensions 1,2,3, and 4, respectively. These
orbits are:
\begin{equation}
\label{eq:Aut-W-orbits}
\Upsilon,\qquad\Xi\setminus\Upsilon,\qquad R\setminus\Xi,\qquad W\setminus R.
\end{equation}
\end{enumerate}
\end{sremarks}

Next we examine
the $\Aut(W)$-action on $W$.

\begin{proposition}
\label{proposition-GL2-action-c}
In the notation as before the following hold.
\begin{enumerate}
\renewcommand\labelenumi{{\rm (\alph{enumi})}}
\renewcommand\theenumi{{\rm (\alph{enumi})}}
\item
\label{proposition-GL2-action-c-a}
The triple $(R,\Xi,\Upsilon)$ is
$\Aut(W)$-invariant.

\item
\label{proposition-GL2-action-c-b} $\Aut(W)$ acts effectively
on $R$.

\item
\label{proposition-GL2-action-c-c} The unipotent radical $\Ru$ of $\Aut(W)$ acts 
freely
on $W\setminus R$ and fixes $\Xi$ pointwise.

\item
\label{proposition-GL2-action-c-d}
Let $T=\z(\Levi)$ be a singular torus in $\Aut(W)$, where $\Levi$ is a reductive 
Levi subgroup.
Then the fixed point locus $W^T$ is $\Levi$-invariant and is a disjoint union 
\begin{equation*}
W^T=\{Q\}\cup\Psi\cup\Xi,
\end{equation*} 
where $Q\in W\setminus R$ is an isolated point and $\Psi\subset R$
is a twisted cubic curve such that $\langle\Psi\rangle\cap\Xi=\emptyset$ and 
$\Psi$ meets any plane $\Pi_\gamma$, $\gamma\in\Upsilon$, in a single point
$Q_\gamma\in\Pi_\gamma$.
\end{enumerate}
\end{proposition}

\begin{proof}
Statement~\xref{proposition-GL2-action-c-a} is immediate from
\xref{def:skew-symmetric}. To show~\xref{proposition-GL2-action-c-b}
let $K$ be the kernel of the restriction homomorphism
$\Aut(W)\to\Aut(R)$, $\alpha\mapsto\alpha|_R$.
The action of $\Aut(\PP^4,\Gamma)$ induces a representation of this
group by automorphisms of the normal bundle
$\NNN_{\Gamma/\PP^4}$ and $\hat R\cong\PP(\NNN_{\Gamma/\PP^4}^\vee)$. The
unipotent radical $\Transl(\CC^4)$ of $\Aut(\PP^4,\Gamma)$ acts
trivially on the base $\Gamma$. Its action on the total space $\hat R$ induces
the $\Ru$-action on $R$, see Lemma
\xref{lem:hatR}\xref{lem:hatR-a}-\xref{lem:hatR-b}.
The latter action is effective
if and only if the former is.

\begin{sclaim}
\label{claim-Transl}
The induced representation of
$\Transl(\CC^4)\cong (\Ga)^4$ on the normal
bundle $\NNN_{\Gamma/\PP^4}$ is faithful.
\end{sclaim}

\begin{proof}[Proof of Claim~\xref{claim-Transl}]
We use Notation~\xref{nota:3-forms}. Let $\pi:
(M_3\oplus\CC)\setminus\{0\}\to\PP(M_3\oplus\CC)$ be the canonical projection,
and let $\Cone(\Gamma)=\pi^{-1}(\Gamma)$ be the affine cone over $\Gamma$
in $(M_3\oplus\CC)\setminus\{0\}\cong\CC^5\setminus\{0\}$.
The normal bundle of $\Cone(\Gamma)$ in $M_3\oplus\CC$ is the pullback
$\pi^*(\NNN_{\Gamma/\PP^4})$.
The induced representation of $\Transl(\CC^4)\cong (\Ga)^4$ on the normal bundle
$\NNN_{\Gamma/\PP^4}$ is faithful if and only if $(M_3,+)\cong\Transl(\CC^4)$
acts effectively on $\NNN_{\Cone(\Gamma)/(M_3\oplus\CC)}$.

An element $g\in (M_3,+)$ acts on
$M_3\oplus\CC$ linearly via
\begin{equation}
\label{eq:action-transl}
g.(f,z)=(f+z g,z)\quad\forall (f,z)\in M_3\oplus\CC.
\end{equation}
It fixes $\Cone(\Gamma)$ pointwise and acts trivially on the
tangent space $T_{(f,0)}\Cone(\Gamma)$. The induced action on the tangent space
$T_{(f,0)}(M_3\oplus\CC)\cong M_3\oplus\CC$ is given by the same formula
\begin{equation}
\label{eq:action-tang-transl}
dg.(u,v)=(u+v g,v)\quad\forall (u,v)\in M_3\oplus\CC.
\end{equation}
Fix a point $(f,0)=((\alpha x+\beta y)^3,0)\in\Cone(\Gamma)$. The
$g$-action on the normal space
\begin{equation*}
\NNN_{\Cone(\Gamma)/(M_3\oplus\CC),
(f,0)}=T_{(f,0)}(M_3\oplus\CC)/T_{(f,0)}\Cone(\Gamma)
\end{equation*}
is trivial if and
only if $g\in T_{(f,0)}\Cone(\Gamma)$. However, one has
\begin{equation*}
\bigcap_{(f,0)\in
\Cone(\Gamma)}
T_{(f,0)}\Cone(\Gamma)\subset T_{(x^3,0)}\Cone(\Gamma)\cap
T_{(y^3,0)}\Cone(\Gamma)=\{0\}.
\end{equation*}
It follows that the only element of $(M_3,+)$ acting trivially on
$\NNN_{\Cone(\Gamma)/(M_3\oplus\CC)}$ is $g=0$.
\end{proof}

\begin{scorollary}
$K\cap\Ru=\{1\}$.
\end{scorollary}

Therefore,~\eqref{eq:2nd-sequence-1} gives an embedding of $K$
in $\GL_2(\CC)=\Aut(W)/\Ru$.
Since $K$ is normal, $K$ is contained in $\z(\GL_2(\CC))\cong\Gm$.
The latter singular torus
acts on $\Upsilon\cong\PP^1$ with a fixed point, say, $Q$. The tangent
representation of the reductive group $\Gm$ on $T_QW$ is faithful. On the other
hand, $Q\in\Xi=\Sing\,R$, hence $T_QW$ is the Zariski tangent space of $R$
at $Q$. Thus, $K$ acts identically on $T_QW$, and so,
$K=\{1\}$.
This proves~\xref{proposition-GL2-action-c-b}.

The first assertion of~\xref{proposition-GL2-action-c-c} follows from the
correspondence of Proposition
\xref{proposition-GL2-action}\xref{proposition-GL2-action-b}. The second
follows from the construction of exact sequence~\eqref{eq:2nd-sequence-1}
in~\cite[Thm.~6.6]{Piontkowski-Van-de-Ven-1999}.

To prove~\xref{proposition-GL2-action-c-d}
we chose coordinates in $\PP^4$ so that the induced action of $\z(\Levi)$ on 
$\PP^4$
is given by the diagonal matrix $\diag(1,1,1,1,\lambda)$,
$\lambda\in\CC^*$; cf.\ the proof of~\xref{proposition-GL2-action-b}. This
action is
identical on $\langle\Gamma\rangle\cong\PP^3$, hence also on the normal bundle
$\NNN_{\Gamma/\PP^3}$. However, it is effective on the normal bundle
$\NNN_{\Gamma/\PP^4}$. This allows to decompose $\NNN_{\Gamma/\PP^4}$ into a
direct sum of proper subbundles of ranks $2$ and 1, respectively. The rank two
subbundle projects in $\hat R\cong\PP(\NNN_{\Gamma/\PP^4}^\vee)$ to a
$\PP^1$-subbundle, and the rank one projects to a disjoint section, say,
$\hat\Psi$. In each plane $\hat\Pi_\gamma$, $\gamma\in\Gamma$, this gives a line
$\hat l_\gamma$ of fixed points of $\z(\Levi)$ and an isolated fixed point
$\hat\Psi\cap\hat\Pi_\gamma$.
The image of $\hat l_\gamma$ in $\Pi_\gamma$ is the line
$l_\gamma=\Pi_\gamma\cap\Xi$. The image $\Psi\subset R$ of the curve
$\hat\Psi\subset\hat R$ is a second irreducible component of the fixed point
locus of $\z(\Levi)|_R$. This curve $\Psi$ is disjoint with $\Xi$ and meets each 
plane
$\Pi_\gamma$ in a point. The linear projection $\phi:
\PP^7\dashrightarrow\PP^4$ with center $\Xi$ as in~\eqref{equation-diagram-1}
sends $\Psi$ isomorphically onto the rational twisted cubic curve $\Gamma$. The 
linear
span $\langle\Psi\rangle$ is sent by $\phi$ onto $\langle
\Gamma\rangle\cong\PP^3$. It follows that $\Psi$ is a rational twisted cubic 
curve, and
$\langle\Psi\rangle\cap\Xi=\emptyset$.
\end{proof}

In the sequel we need the following simple lemma.

\begin{lemma}
\label{lem:GaxGm}
Let $\mathfrak{H}$ be a subgroup of $\Aut(W)$ isomorphic either to $\GL_2(\CC)$, 
or to the $2$-torus $(\Gm)^2$.
Then there is a unique Levi subgroup $\Levi$ of $\Aut(W)$ such that 
$\z(\Levi)\subset\mathfrak{H}\subset\Levi$.
\end{lemma}

\begin{proof} In the case $\mathfrak{H}\cong\GL_2(\CC)$ the assertion follows 
from Proposition~\xref{proposition-GL2-action}\xref{proposition-GL2-action-a}.
If $\mathfrak{H}\cong (\Gm)^2$, then $\mathfrak{H}$ is a maximal torus in 
$\Aut(W)$. Being
reductive, $\mathfrak{H}$ is contained in a Levi subgroup $\Levi$ of $\Aut(W)$
and contains its center $\z(\Levi)$. Up to conjugation, one may suppose that
$\Levi$ is the standard $\GL_2(\CC)$-subgroup of 
$\Aut(W)\hookrightarrow\operatorname{Aff}(\CC^4)$ as in~\xref{nota:3-forms}, and 
$\mathfrak{H}$ is the diagonal torus.
That is, $W\setminus R\cong\PP^4\setminus\PP^3=\CC^4$ is realized as the vector
space $M_3$ of binary cubic forms, where $\Ru$ acts by translations and
$\GL_2(\CC)$ acts via $h.f=f\circ h^{-1}$ for any $h\in\GL_2(\CC)$ and $f\in
M_3$. In particular, $\mathfrak{H}$ acts via
\begin{equation*}
(\lambda,\mu).(a_0,a_1,a_2,a_3)=(\lambda^{-3}a_0,\lambda^{-2}\mu^{-1}
a_1,\lambda^{-1}\mu^{-2} a_2,\mu^{-3} a_3).
\end{equation*}
Hence $\mathfrak{H}$ has a unique fixed point, say, $P$, which is the unique
fixed point of $\Levi$ in $ W\setminus R$. This point determines $\Levi$
uniquely, see Proposition
\xref{proposition-GL2-action}\xref{proposition-GL2-action-cc}.
\end{proof}

\section{Quintic scrolls in $W_5$ and their stabilizers}
\label{sec-3-bis}

In the next proposition we construct a special quintic scroll
$F\cong\FF_1$ in $W$ contained
in $R$ and acted upon by $\GL_2(\CC)$.

\begin{proposition}
\label{prop:quintic-scroll}
With the notation as in~\xref{nota:3-forms} the following hold.

\begin{enumerate}
\renewcommand\labelenumi{\rm (\alph{enumi})}
\renewcommand\theenumi{\rm (\alph{enumi})}
\item
\label{prop:quintic-scroll-b}
Let $\hat\Delta$ be the proper transform of $\Delta$ in
$\hat W$.
Then the image $\rho(\hat\Delta)\subset W$ is
cut out on $W$ by a quadric.

\item
\label{prop:quintic-scroll-c}
The reduced intersection
$F=(\rho(\hat\Delta)\cap R)_{\red}\subset\PP^6$
is a smooth quintic scroll isomorphic to $\FF_1$ and invariant under
an effective $\GL_2(\CC)$-action on $W$.

\item
\label{remark-uniqueness}
$F\cap\Xi=\Upsilon$ and $\Xi$ is the only plane in $W$ meeting $F\cong\FF_1$ 
along a
conic.
\end{enumerate}
\end{proposition}

\begin{proof}
\xref{prop:quintic-scroll-b} 
Since $\Delta$ is singular along $\Gamma$, we have
$\hat\Delta\sim 4L^*-2\hat R\sim 2H^*$ on $\hat W$, see~\eqref{eq:exprim}.
Hence $\rho(\hat\Delta)$ is cut out on $W$ by a quadric.

\xref{prop:quintic-scroll-c}
First we construct a special quintic scroll $F\subset R$, and then we show that 
it coincides with $(\rho(\hat\Delta)\cap R)_{\red}$. Let $N=\Cone(\Gamma, 
P)\subset\PP^4$ be as in Remark~\xref{rem:GL2-action}.\xref{rem:GL2-action-1}, 
and let $\hat\Gamma=\hat N\cap\hat R$ be as in Lemma~\xref{lem:Gamma}. By 
Lemma~\xref{lem:hatR} the normalization morphism $\rho|_{\hat R}:\hat R\to R$ 
sends each fiber $\hat\Pi_\gamma$ of the $\PP^2$-bundle $\varphi|_{\hat R}:\hat 
R\to\Gamma$ isomorphically onto a plane $\Pi_\gamma\subset 
R=\bigcup_{\gamma'\in\Gamma}\Pi_{\gamma'}$. The curve $\hat\Gamma\subset\hat R$ is 
a section of $\varphi|_{\hat R}$ meeting any plane $\hat{\Pi}_\gamma$ in a 
single point, say, $\hat p_\gamma$. Set $\Psi:=\rho(\hat\Gamma)$. By 
Lemma~\xref{lem:Gamma} 
one has $\langle\Psi\rangle\cap\Xi=\emptyset$ and $\Psi$ is a twisted cubic 
curve. Hence 
$p_\gamma:=\rho(\hat{p}_\gamma)\in\Pi_\gamma\setminus\Xi$. Letting 
$\{q_\gamma\}=\Pi_\gamma\cap\Upsilon$, consider the one-parameter family of 
lines $\Lambda_\gamma\subset\Pi_\gamma\subset R$ joining $p_\gamma$ and 
$q_\gamma$. Thus, any $\Lambda_\gamma$ meets the conic $\Upsilon\subset\Xi$ and 
the twisted cubic $\Psi\subset R\setminus\Xi$. Since 
$\langle\Psi\rangle\cap\langle\Upsilon\rangle=\emptyset$, the join
\begin{equation}
\label{equation-F-join-corresponding-points}
F=\bigcup_{\gamma\in\Upsilon}\Lambda_\gamma
\end{equation}
of corresponding points
of $\Upsilon$ and $\Psi$ is
a rational normal quintic scroll in $R$
(\cite[Ch. 4, \S~3]{Griffiths-Harris-1994}).

Let us show that $F\subset (\rho(\hat\Delta)\cap R)_{\red}$ for $F$ as in
\eqref{equation-F-join-corresponding-points}.
Due to our choice of $\hat\Gamma$ we have $\hat\Delta\cap\hat R\supset\hat 
N\cap\hat R=\hat
\Gamma$.
So, $(\rho(\hat\Delta)\cap R)_{\red}\supset\Psi$.
On the other hand, $(\rho(\hat\Delta)\cap\Xi)_{\red}$ is a 
$\GL_2(\CC)$-invariant
curve.
Since $\Upsilon$ is the only $\GL_2(\CC)$-invariant curve in $\Xi$, one has
$\Upsilon=(\rho(\hat\Delta)\cap\Xi)_{\red}\subset (\rho(\hat
\Delta)\cap R)_{\red}$.

For any fiber $\hat\Pi_{\gamma}$ of $\varphi|_{\hat R}:\hat R\to\Gamma$,
where $\gamma\in\Gamma$, the intersection $\hat\Pi_{\gamma}\cap\hat\Delta$
is the projectivization of the tangent cone to $\Delta\cap\PP^3$ at $\gamma$,
where $\PP^3\subset\PP^4$ is a general hyperplane passing through $\gamma$.
Since the singularity of $\Delta\cap\PP^3$ at $\gamma$ looks locally like
a product of a simple cusp singularity by a smooth curve germ,
this projectivization is a line $\hat\Lambda_\gamma$ in $\hat
\Pi_{\gamma}\cong\PP^2$.
The image $\Lambda_\gamma=\rho(\hat\Lambda_\gamma)$ is a line in $W$ passing 
through the
points
$p_\gamma\in\Psi\cap\Pi_\gamma$
and $q_\gamma\in\Pi_\gamma\cap\Upsilon$.
Thus, $F\subset (\rho(\hat\Delta)\cap R)_{\red}$.
On the other hand, $\deg (\rho(\hat\Delta)\cap R)=10$ and $\deg F=5$. 
Furthermore, $\Delta$ has cuspidal singularities along
$\Gamma=\varphi(\hat R)$, hence $\hat R$
is tangent to $\hat\Delta$
along $(\hat\Delta\cap\hat R)_{\red}$.
Therefore, one has
$F=(\rho(\hat\Delta)\cap R)_{\red}$.

The hyperplane section $R$ is invariant under the $\Aut(W)$-action, and
$\rho(\hat\Delta)\subset W$ is invariant under the $\GL_2(\CC)$-action on $W$
as in Remark~\xref{rem:GL2-action}.\xref{rem:GL2-action-3}. Hence
the quintic scroll $F=(\rho(\hat\Delta)\cap R)_{\red}\subset W$
is $\GL_2(\CC)$-invariant
as well.

\xref{remark-uniqueness}
Indeed, by construction,$F\cap\Xi=\Upsilon$ and 
any plane $\Pi_\gamma$ meets $F$ along a ruling $\Lambda_\gamma$ and is
not tangent to $F$ along $\Lambda_\gamma$.
\end{proof}

\begin{lemma}
\label{lem:aut-pencils}
Let $J\subset\Xi$, $J\neq\Upsilon$, be a smooth conic meeting $\Upsilon$ with 
even multiplicities. Then 
the following hold.
\begin{enumerate}
\renewcommand\labelenumi{{\rm (\alph{enumi})}}
\renewcommand\theenumi{{\rm (\alph{enumi})}}
\item
\label{lem:aut-pencils-a}
The group $\Aut(\Xi,\Upsilon,J)$ of
automorphisms of $\Xi$ which leave $\Upsilon$ and $J$ invariant is given by 
the following table:
\par\medskip\noindent
\setlength{\extrarowheight}{1pt}
\newcommand{\heading}[1]{\multicolumn{1}{c|}{#1}}
\newcommand{\headingl}[1]{\multicolumn{1}{c}{#1}}
{\rm
\begin{tabularx}{0.85\textwidth}{p{0.15\textwidth}|p{0.5\textwidth}}
\heading{$J$}&\headingl{$\Aut(\Xi,\Upsilon,J)$}
\\\hline
\xref{rem:touching-conics}\xref{lem:touching-conics-Gl2}&
$\Aut(\Upsilon)\cong\PGL_2(\CC)$
\\
\xref{rem:touching-conics}\xref{lem:touching-conics-Ga}&
$\Ga\rtimes (\ZZ/2\ZZ)$
\\
\xref{rem:touching-conics}\xref{lem:touching-conics-Gm}&
$\Gm\rtimes (\ZZ/2\ZZ)$\\\hline
\end{tabularx}\par\vspace{10pt}\noindent
}
\item
\label{lem:aut-pencils-b}
Up to the natural action of $\Aut(\Xi)$, there is a unique configuration 
$(\Upsilon,J)$ of type~\xref{rem:touching-conics}\xref{lem:touching-conics-Ga}, 
while the configurations of type 
\xref{rem:touching-conics}\xref{lem:touching-conics-Gm} form a one-parameter 
family.
\end{enumerate}
\end{lemma}

\begin{proof}
\xref{lem:aut-pencils-a}
We leave aside the trivial case $J=\Upsilon$. If $J\neq\Upsilon$, then the 
action of $\Aut^0(\Xi,\Upsilon,J)$ leaves invariant $\Upsilon$, $J$, and the
degenerate members of the pencil of conics $\PPP$ generated by $\Upsilon$ and
$J$. Hence $\Aut^0(\Xi,\Upsilon,J)$ acts identically on the base $\PP^1$ of 
$\PPP$. So, $\Aut^0(\Xi,\Upsilon,J)$ leaves
invariant each member of $\PPP$.

For $J$ of type~\xref{rem:touching-conics}\xref{lem:touching-conics-Gm}, in 
appropriate homogeneous coordinates $(x:y:z)$ in $\Xi\cong\PP^2$ the general 
member $J_{(a:b)}$ of $\PPP$ can be given by equation $ax^2+byz=0$. The 
projective transformations fixing the common points of $J$ and $\Upsilon$ are 
of 
the form $(x:y:z)\longmapsto (x:\lambda y:\lambda^{-1}z)$, where 
$\lambda\in\CC^*$. Thus, $\Aut^0 (\Xi,\Upsilon,J)\cong\Gm$. The group 
$\Aut(\Xi,\Upsilon,J)$ is generated by $\Aut^0 (\Xi,\Upsilon,J)$ and the 
involution 
\begin{equation*}
\kappa: (x:y:z)\longmapsto (x:z:y) 
\end{equation*}
that interchanges the two 
points of the intersection $\Upsilon\cap J$.

For $J$ of type~\xref{rem:touching-conics}\xref{lem:touching-conics-Ga}, the
general member $J_{(a:b)}$ of $\PPP$ can be given by the equation
$a(x^2+yz)+bz^2=0$. Then
\begin{equation*}
\Aut^0 (\Xi,\Upsilon,J)=\{(x:y:z)\longmapsto (x+ez:y+2ex+e^2z:z)\mid 
e\in\CC\}\cong\Ga,
\end{equation*}
while $\Aut(\Xi,\Upsilon,J)$ is generated by $\Aut^0 (\Xi,\Upsilon,J)$ and the 
involution 
\begin{equation*}
\kappa: (x:y:z)\longmapsto (-x:y:z).
\end{equation*}
The proof of~\xref{lem:aut-pencils-b} results from elementary
computations. We leave this to the reader.
\end{proof}

\begin{proposition}
\label{prop:center-inv-scroll}
Fix a nondegenerate conic $J\subset\Xi$ which touches $\Upsilon$ with even 
multiplicities, see Figure~\xref{fig-touching-conics}. Fix also a reductive 
Levi 
subgroup $\Levi\subset\Aut(W)$. Let an $\Levi$-invariant
twisted cubic curve $\Psi$ be the set of fixed points of $\z(\Levi)$ in 
$R\setminus\Xi$, see 
Proposition~\xref{proposition-GL2-action-c}\xref{proposition-GL2-action-c-d}.
\begin{enumerate}
\renewcommand\labelenumi{\rm (\alph{enumi})}
\renewcommand\theenumi{\rm (\alph{enumi})}
\item
\label{prop:center-inv-scroll-a}
There exists
a $\z(\Levi)$-invariant
rational normal quintic scroll $F\subset R$, $F\cong\FF_1$, with exceptional 
section
$J=F\cap\Xi$.

\item
\label{prop:center-inv-scroll-b}
If $J=\Upsilon$, then $F$ as in~\xref{prop:center-inv-scroll-a} is unique and 
can be transformed into the scroll 
$(\rho(\hat\Delta)\cap R)_{\red}$ as in Proposition~\xref{prop:quintic-scroll} 
by an automorphism from $\Ru$. If $J\neq
\Upsilon$, then there are exactly two different scrolls $F_1$ and $F_2$ as 
in~\xref{prop:center-inv-scroll-a}. Any such scroll contains $\Psi$.

\item
\label{prop:center-inv-scroll-c}
Furthermore, any $F$ as in~\xref{prop:center-inv-scroll-a} is invariant under 
the action on $W$ of the
identity component $G=G(\Levi,J)$ of the subgroup 
$\Levi\cap\varrho^{-1}(\Aut(\Xi,\Upsilon,J))$, where
$\varrho:\Aut(W)\to\Aut(\Upsilon)=\Aut(\Xi,\Upsilon)$ is the restriction 
homomorphism, see
\eqref{eq:2nd-sequence-1}.

\item
\label{prop:center-inv-scroll-d}
For $J\neq\Upsilon$, consider the involution 
$\kappa\in\Aut(\Xi,\Upsilon,J)\setminus\Aut^0(\Xi,\Upsilon,J)$ \textup{(}see 
the 
proof of Lemma~\xref{lem:aut-pencils}\xref{lem:aut-pencils-a}\textup{)}. Let 
$\tilde\kappa\in\Levi\cap\varrho^{-1}(\Aut(\Xi,\Upsilon,J))\subset\Aut(W)$ be 
such that $\varrho(\tilde\kappa)=\kappa$. Then $\tilde\kappa$ interchanges 
$F_1$ 
and $F_2$. In particular, $(W,F_1)\cong (W,F_2)$.
\end{enumerate}
\end{proposition}

\begin{proof}
\xref{prop:center-inv-scroll-a} Assume first that $J=\Upsilon$.
By Proposition~\xref{proposition-GL2-action-c}~\xref{proposition-GL2-action-c-d}
any plane $\Pi_\gamma$, $\gamma\in\Upsilon$, meets $\Psi$ in a unique point,
say, $p_\gamma$, and meets $\Upsilon$ tangentially in a unique point $q_\gamma$. 
Indeed, $l_\gamma:=\Pi_\gamma\cap\Xi$ is the tangent line to $\Upsilon$ at 
$q_\gamma$. The lines $(p_\gamma,q_\gamma)$ sweep up
a normal rational quintic scroll $F=F(\Psi)$. This scroll
$F$ is $\Levi$-invariant and meets the plane $\Xi$ along $\Upsilon$. In 
particular, $F$ is $\z(\Levi)$-invariant. If $\Levi=\Levi_0$ comes from the 
standard $\GL_2(\CC)$-action
on $\PP(M_3\oplus\CC)$ as in~\xref{nota:3-forms}, then $F$ coincides
with $(\rho(\hat\Delta)\cap R)_{\red}$,
see the proof of Proposition
\xref{prop:quintic-scroll}. Otherwise, $\Levi=g\Levi_0g^{-1}$ for some $g\in\Ru$ 
such that $g(F)=(\rho(\hat\Delta)\cap R)_{\red}$.

Let further $J\neq\Upsilon$. Consider the $2:1$ morphism
\begin{equation*}
\delta: J\longrightarrow\Upsilon,\quad\Pi_\gamma\cap 
J\longmapsto\Pi_\gamma\cap\Upsilon=\{q_\gamma\}.
\end{equation*}
By our assumption, the ramification indices of $\delta$ are even. Hence 
$\delta$ 
admits two distinct sections, say, $\sigma_1,\sigma_2:\Upsilon\to J$, see the 
proof of 
Proposition~\xref{lemma-classification-singular-fibers}\xref{lemma-classification-singular-fibers-toric}. Letting 
$t_{\gamma,i}=\sigma_i(q_\gamma)$, consider the smooth quintic scroll 
$F_i=F_i(J,\Levi)\cong\FF_1$ formed by the lines $(p_\gamma,t_{\gamma,i})$, 
$\gamma\in\Upsilon$. The conic $J\supset\Xi$ is a common exceptional section of 
$F_i\to\PP^1$, $i=1,2$, and $\Psi$ is a common section with $\Psi^2=1$. 
Besides, 
the scrolls $F_1$ and $F_2$ have common rulings passing through the points of 
$J\cap\Upsilon$. Since $J$ and $\Psi$ are pointwise fixed under the 
$\z(\Levi)$-action (see Proposition 
\xref{proposition-GL2-action-c}\xref{proposition-GL2-action-c-d}), each ruling 
of $F_i$ is $\z(\Levi)$-invariant and represents an orbit closure of 
$\z(\Levi)$. Hence $F_i$ is $\z(\Levi)$-invariant for $i=1,2$.

\xref{prop:center-inv-scroll-b}
Let now $ F\cong\FF_1$ be a smooth rational $\z(\Levi)$-invariant quintic
scroll in $R$ with exceptional section $J\subset\Xi$. 
Since $\z(\Levi)$ acts identically on $\Xi$, see again 
Proposition~\xref{proposition-GL2-action-c}\xref{proposition-GL2-action-c-d}, 
the $\z(\Levi)$-action on $F$ leaves invariant each ruling $\Lambda_\gamma$, 
where $\Lambda_\gamma$ represents a $\z(\Levi)$-orbit closure. Hence, there are 
exactly two fixed points of $\z(\Levi)$ on
$\Lambda_\gamma$. One of these points
runs over $J$ when $\gamma$ runs over $\Upsilon$, and the other one runs over 
$\Psi$. In particular, $\Psi\subset F$.

The line $\Lambda_\gamma$ is contained in a unique plane $\Pi_\gamma$ through
the point $p_\gamma\in\Psi\cap\Lambda_\gamma$. It passes through the point 
$p_\gamma$ and one of
the intersection points $t_{\gamma,1}$, $t_{\gamma,2}\in\Pi_\gamma\cap J$. It
follows that $F$ coincides with one of the scrolls $F_1$ and $F_2$
constructed in~\xref{prop:center-inv-scroll-a}, where $F_1=F_2=F$ in the case 
$J=\Upsilon$.

To complete the proof of~\xref{prop:center-inv-scroll-b} in the case 
$J=\Upsilon$, consider the standard quintic scroll $F^{\s}:=(\rho(\hat
\Delta)\cap R)_{\red}$ with $F^{\s}\cap\Xi=\Upsilon$. It is invariant under the 
action on $W$ of a special Levi subgroup $\Levi^{\s}\cong\GL_2(\CC)$ as in 
Proposition~\xref{prop:quintic-scroll}\xref{prop:quintic-scroll-c}. The Levi 
subgroups $\Levi$ and $\Levi^{\s}$ are conjugated via an element $g\in\Ru$. 
Clearly, $g$ conjugates their singular tori $\z(\Levi)$ and $\z(\Levi^{\s})$, 
and sends the twisted cubic curve $\Psi\subset F$ to the corresponding 
$\z(\Levi^{\s})$-fixed curve $\Psi^{\s}\subset F^{\s}$. By virtue of the 
preceding uniqueness result, it follows that $g(F)=F^{\s}$.

Now~\xref{prop:center-inv-scroll-c} is immediate. Notice that $G$ is a connected
subgroup of $\Levi$ which leaves $J$ invariant. We claim that $G$ sends the
rulings of $F$ into rulings. Indeed, these rulings are $\z(\Levi)$-orbit 
closures
meeting $J$. Since $\Levi=\mathcal{C}_L(\z(\Levi))$, $\Levi$ acts on the set of
$\z(\Levi)$-orbits. Hence $G$ sends the rulings of $F$ into $\z(\Levi)$-orbit
closures meeting $J$. By connectedness of $G$, it sends the rulings of $F$ into
rulings. Thus, the scroll $F$ is invariant under $G$.

\xref{prop:center-inv-scroll-d}
Notice that $\tilde\kappa\in\Levi$, $\Psi$ is $\Levi$-invariant, and $J$ is 
invariant under $\kappa=\tilde\kappa|_{\Xi}$. Furthermore,
$\tilde\kappa\in\Levi$ commutes with $\z(\Levi)$, hence sends the 
$\z(\Levi)$-orbits into $\z(\Levi)$-orbits. It follows from the construction of 
$F_1$ and $F_2$ in
\xref{prop:center-inv-scroll-b} that $F_1\cup F_2$ is invariant under 
$\tilde\kappa$. It can be readily seen that $\tilde\kappa(F_i)=F_j$ for $i\neq 
j$.
\end{proof}

From Lemma~\xref{lem:aut-pencils}\xref{lem:aut-pencils-a} and 
Proposition~\xref{prop:center-inv-scroll}\xref{prop:center-inv-scroll-c} we 
deduce such a corollary.

\begin{scorollary}
\label{cor:aut-center}
Consider a $\z(\Levi)$-invariant quintic scroll $F\subset R$ as in 
Proposition~\xref{prop:center-inv-scroll}\xref{prop:center-inv-scroll-a} along 
with the subgroup $G=G(\Levi,J)\subset\Levi$ as 
in~\xref{prop:center-inv-scroll}\xref{prop:center-inv-scroll-c}. Then the 
following hold.
\begin{enumerate}
\renewcommand\labelenumi{\rm (\alph{enumi})}
\renewcommand\theenumi{\rm (\alph{enumi})}
\item
\label{cor:aut-center-a}
$\Aut^0(W,F)=(\Ru\cap\Aut^0(W,F))\rtimes G$, and

\item
\label{cor:aut-center-b} $G$ is isomorphic either to $\GL_2(\CC)$, or to 
$\Ga\times\Gm$, or to $(\Gm)^2$ provided the conic $J=F\cap\Xi$ is of 
type~\xref{rem:touching-conics}\xref{lem:touching-conics-Gl2}, 
\xref{rem:touching-conics}\xref{lem:touching-conics-Ga}, 
and~\xref{rem:touching-conics}\xref{lem:touching-conics-Gm}, respectively. If 
$G\cong\Ga\times\Gm$, then the $\Ga$-subgroup of $G$ acts nontrivially on $J$, 
and the $\Gm$-subgroup preserves each ruling of $F\to\PP^1$;

\item
\label{cor:aut-center-c}
Let $U=\Ru\cap\Aut^0(W,F)\neq\{1\}$. Then $U$ acts trivially on $J$, fiberwise 
on $F\to\PP^1$, and either $U\cong\Ga$, or $U\cong (\Ga)^2$. For any $t\in U$ 
the rational twisted cubic curve $t(\Psi)\subset F$ is pointwise fixed under the 
singular torus $\z(\Levi_t)$, where $\Levi_t=t\cdot\Levi\cdot t^{-1}$.
\end{enumerate}
\end{scorollary}

\begin{proof}
\xref{cor:aut-center-a}
We have $\Aut(W)=\Ru\rtimes L$, see~\eqref{eq:Levi-decomp}. Let $G$ be as in 
Proposition~\xref{prop:center-inv-scroll}\xref{prop:center-inv-scroll-c}.
Since $G\subset L\cap\Aut^0(W,F)$, then
\[
(\Ru\cap\Aut^0(W, F))\rtimes G=(\Ru\cap\Aut^0(W, F))\cdot G\subset\Aut^0(W,F).
\]

To show the opposite inclusion we proceed as follows. Let $g=r\cdot 
l\in\Aut^0(W,F)$, where $r\in\Ru$ and $l\in\Levi$.
\begin{itemize}

\item
Since $r|_{\Xi}=\id_{\Xi}$, see 
Proposition~\xref{proposition-GL2-action-c}\xref{proposition-GL2-action-c-c}, 
and $g(J)=J$, then $l(J)=J$.

\item
Since $l(\Psi)=\Psi$ and $g(\Psi)\subset F$, then $r(\Psi)\subset F$.

\item
Since $r|_{\Xi}=\id_{\Xi}$, then $r(\Pi_\gamma)=\Pi_\gamma$ 
$\forall\gamma\in\Upsilon$.

\item
Since $r(\Psi)\subset F$ and $r(\Psi\cap\Pi_\gamma)\subset 
F\cap\Pi_\gamma=\Lambda_\gamma$, then $r(\Lambda_\gamma)=\Lambda_\gamma$ 
$\forall\gamma\in\Upsilon$.

\item
Thus, $r(F)=F$, and since $g(F)=F$, then $l(F)=F$.

\item
It follows that
$\Aut^0(W,F)=(\Ru\cap\Aut^0(W,F))\rtimes (\Levi\cap\Aut^0(W,F))$.

\item
Since $\Ru\cap\Aut^0(W,F)$ is connected, then $\Levi\cap\Aut^0(W,F)$ is.

\item
Since $l\in (\Levi\cap\Aut^0(W,F))^0$ and $\varrho(l)\in\Aut^0(\Xi,\Upsilon,J)$, 
then $l\in G$.

\item
Therefore, $G=\Levi\cap\Aut^0(W,F)$, and so,~\xref{cor:aut-center-a} follows.
\end{itemize}

\xref{cor:aut-center-b} is straightforward from Lemma
\xref{lem:aut-pencils} and the definition of $G$.

\xref{cor:aut-center-c} Let $g\in (\Ru\cap\Aut^0(W,F))\setminus\{1\}$, and let 
$U$ be the unique one-parameter subgroup of $\Ru$ containing $g$. Then $U$ is 
the Zariski closure of the group generated by $g$. Hence $U$ stabilizes $F$. We 
have shown before that the action of $U=\Ru\cap\Aut^0(W,F)$ on $F$ preserves 
each ruling and fixes $J$ pointwise. The restriction $U|_f$ to a general ruling 
$f\cong\PP^1$ acts via the unipotent part of the Borel subgroup of 
$\Aut(f)\cong\PGL_2(\CC)$ fixing the point $f\cap J$. Contracting $J$ to a 
point 
$p\in\PP^2$ yields an embedding of $U$ onto an abelian unipotent group of 
linear 
transformations fixing $p$ and preserving each line through $p$.
Now the remaining assertions follow.
\end{proof}

In the following lemma we indicate the case, where the condition 
$\Ru\cap\Aut^0(W,F)=\{1\}$ is automatically fulfilled.

\begin{lemma}
\label{lem:j=upsilon}
Let $F\subset R$ be a 
rational normal quintic scroll which meets $\Xi$ along a smooth conic $J$. Then 
the following conditions are equivalent:
\begin{enumerate}
\item
\label{lem:j=upsilon-i} 
$\Aut(W,F)$ contains a subgroup $\mathcal{H}\cong\SL_2(\CC)$;

\item
\label{lem:j=upsilon-ii} 
$\Aut(W,F)\cong\GL_2(\CC)$ and $\Ru\cap\Aut(W,F)=\{1\}$;

\item
\label{lem:j=upsilon-iii} 
$J=\Upsilon$ and $F$ is invariant under the action of a singular torus of $\Aut(W)$.
\end{enumerate}
\end{lemma}

\begin{proof}
The implication~\xref{lem:j=upsilon-iii}$\Rightarrow$\xref{lem:j=upsilon-i} is 
straightforward from Corollary~\xref{cor:aut-center}\xref{cor:aut-center-b}.
The implication~\xref{lem:j=upsilon-ii}$\Rightarrow$\xref{lem:j=upsilon-iii} is
immediate; indeed, $\Upsilon$ is a unique $\PGL_2(\CC)$-invariant conic in 
$\Xi$. 

It remains to prove
\xref{lem:j=upsilon-i}$\Rightarrow$\xref{lem:j=upsilon-ii}. The
$\SL_2(\CC)$-subgroup $\mathcal{H}\subset\Aut(W,F)$ is contained in a unique
reductive Levi subgroup $\Levi\cong\GL_2(\CC)$ of $\Aut(W)$. 
Let $T:=\z(\Levi)$.
The center $\z(\mathcal{H})$ has an isolated fixed point on 
$\Pi_\gamma\setminus\Xi$, $\forall\gamma\in\Upsilon$.
Therefore, $\z(\mathcal{H})$ has a curve of fixed points $\Psi\subset 
R\setminus\Xi$.
Since $\z(\mathcal{H})\subset T$, the points of $\Psi$ are fixed by $T$, 
see 
Proposition~\xref{proposition-GL2-action-c}\xref{proposition-GL2-action-c-d}.
The lines joining the corresponding points of $J$ and $\Psi$ are $T$-orbits. 
Therefore, $\Psi$, $J=F\cap\Xi$, and $F$ are $\Levi$-invariant.

The subgroup $\Levi\subset\Aut(W, F)$ normalizes both $\Aut(W,F)$ and 
$\Ru\cong(\Ga)^4$. Hence it normalizes the intersection $K:=\Aut(W,F)\cap\Ru$. 
Since $K$ is normalized by $\Levi$ and $\Ru$, it is a normal subgroup in 
$\Aut(W)$. However, $K$ cannot be a proper subgroup of $\Ru$. Indeed, the 
representation of $\Levi\cong\GL_2(\CC)$ on $\Ru\cong\CC^4$ by conjugation is 
irreducible, see Remark~\xref{rem:Mostow}.\xref{rem:Mostow-1}. Therefore, there 
is an alternative: either $\Aut(W,F)=\Levi$, or $\Aut(W,F)=\Aut(W)$. The latter 
is impossible since $R\setminus\Xi$ is an orbit of $\Aut(W)$, see 
\eqref{eq:Aut-W-orbits}. This proves~\xref{lem:j=upsilon-ii}.
\end{proof}

\begin{scorollary}
\label{cor:unique-GL2}
Let $F_1, F_2\subset R$ be two quintic scrolls meeting 
$\Xi$ along smooth conics. If $\Aut(W,F_i)\cong\GL_2(\CC)$ for $i=1,2$, then 
$F_2=g(F_1)$ for some $g\in\Ru$. Consequently, the pairs $(V_1,S_1)$ and $(V_2, 
S_2)$ linked to $(W,F_1)$ and $(W, F_2)$, respectively, are isomorphic.
\end{scorollary}

\begin{proof}
Due to Lemma~\xref{lem:j=upsilon}, for $i=1,2$ one has $F_i\cap\Xi=\Upsilon$, 
$\Levi_i:=\Aut(W,F_i)$ is a Levi subgroup in $\Aut(W)$, and so, $\z(L_i)$ 
stabilizes $F_i$. Letting $\Psi_i$ be the component of the fixed point set of 
$\z(\Levi_i)|_{F_i}$ different from $\Upsilon$, we see that all the assumptions 
of Proposition~\xref{prop:center-inv-scroll}\xref{prop:center-inv-scroll-b} are 
satisfied for $F_1$ and $F_2$. By this proposition, both $F_1$ and $F_2$ can be 
transformed into $(\rho(\hat\Delta)\cap R)_{\red}$ by suitable automorphisms 
from $\Ru$. Now the assertions follow.
\end{proof}

In the next lemma we construct some $\Levi$-invariant cubic cone in $W$.

\begin{lemma}
\label{lem:center-inv-cone}
Let $\Levi$ be a reductive Levi
subgroup of $\Aut(W)$, let $Q$ be the unique fixed point of the singular torus 
$T=\z(\Levi)$
in $W\setminus R$, and let $\Psi\subset R\setminus\Xi$ be the one-dimensional 
component of the
fixed point set $W^{T}$, see 
Proposition~\xref{proposition-GL2-action-c}\xref{proposition-GL2-action-c-d}. 
Consider the cubic cone $\Cone(\Psi, Q)
\subset\PP^7$ over $\Psi$ with vertex $Q$. Then $\Cone(\Psi, Q)$ is contained in 
$W$ and $\Levi$-invariant.
\end{lemma}

\begin{proof}
It is easy to see that there is a one-dimensional family of lines in $W$ passing 
through $Q$
(cf.~\cite[Proposition~5.3]{Iskovskih1977a},~\cite[Lemma~2.2.6]{
Kuznetsov-Prokhorov-Shramov}).
The union of these lines form an $\Levi$-invariant surface $S_Q$, 
which is a cone with vertex $Q$.
Moreover, $S_Q=W\cap T_{Q}W$ because $W=W_5\subset\PP^7$ is an intersection of 
quadrics.
The singular torus $T$ has a curve of fixed points on $S_Q$, and this curve must 
coincide with 
$S_Q\cap R=\Psi$. Therefore, $S_Q=\Cone(\Psi, Q)$.
\end{proof}

\begin{mdefinition}
\label{nota:diagr-2}
Let us fix the following setup. Consider two linked pairs $(W,F)$ and $(V,S)$ 
as 
in Corollary~\xref{cor:diagr-2}. That is, $F\subset R$ is a smooth rational 
normal quintic scroll such that $F\cong\FF_1$, $J=F\cap\Xi$ is a smooth conic 
touching $\Upsilon$ with even multiplicities, $V$ is a smooth Fano-Mukai 
fourfold of genus $10$, and $S\subset V$ is a cubic cone with 
$\Aut^0(V,S)\cong\Aut^0(W,F)$. Suppose that $F$ is invariant under a singular 
torus $\z(\Levi)$, where $\Levi\subset\Aut(W)$ is a reductive Levi subgroup, 
see 
Remark~\xref{rem:Mostow}.\xref{rem:Mostow-2}. Via the above isomorphism, the 
identity component 
$G_{\Levi}$ of $\Levi\cap\Aut(W,F)$ acts effectively on $V$ stabilizing $S$ 
(cf.\ Corollary~\xref{cor:aut-center}). In the following lemma we construct a 
pair of disjoint $G_{\Levi}$-invariant cubic cones in $V$.
\end{mdefinition}

\begin{slemma}
\label{lem:disjoint-cones}
In the setting of~\xref{nota:diagr-2} the following hold.
\begin{enumerate}
\renewcommand\labelenumi{\rm (\alph{enumi})}
\renewcommand\theenumi{\rm (\alph{enumi})}
\item
\label{nota:diagr-2-a}
There exists a unique $G_{\Levi}$-invariant cubic cone
$S'\subset V$ disjoint with $S$.

\item
\label{nota:diagr-2-b} If $\Ru\cap\Aut^0(W,F)=\{1\}$, then $S'$ 
in~\xref{nota:diagr-2-a} is $\Aut^0(V,S)$-invariant.

\item
\label{nota:diagr-2-c} If $\Ru\cap\Aut^0(W,F)\neq\{1\}$, then there exists a 
one-parameter family $(S'_t)$ of cubic cones in $V$ disjoint with $S$.
\end{enumerate}
\end{slemma}

\begin{proof}
\xref{nota:diagr-2-a}
Let $Q$ be the
unique fixed point of $\z(\Levi)$ in $W\setminus R$, and let $\Psi\subset 
R\setminus\Xi$ be the 1-dimensional component
of the fixed point set of $\z(\Levi)$. Consider the $\Levi$-invariant cubic cone
$S'_W:=\Cone(\Psi, Q)\subset W$, see Lemma~\xref{lem:center-inv-cone}.
Since $S'_W$
meets the hyperplane section $R$ transversely along
$\Psi$, the proper
transform $S'_{\tilde W}$ of $S'_W$ in $\tilde W$ is isomorphic to $S'_W$ and 
disjoint with $\tilde R\subset\tilde W$.
Then the image $S':=\xi(S'_{\tilde W})\subset V$ is disjoint with
$S=\xi(\tilde R)$ and $G$-invariant. The linear projection 
$\theta:\PP^{12}\dashrightarrow\PP^7$ with center
$\langle S\rangle$ sends $S'$ isomorphically onto $S'_W$. Therefore, $S'$ is a 
cubic cone disjoint with $S$. This yields the existence 
in~\xref{nota:diagr-2-a}.

Let $S''\subset V$ be a $G_{\Levi}$-invariant cubic cone disjoint with $S$. Then
$S''_W=\theta(S'')\subset W$ is a $G_{\Levi}$-invariant cubic cone in $W$. The 
hyperplane section $A=A_S\subset V$ as in~\xref{diagram-2} meets $S''$ along a 
$G_{\Levi}$-invariant rational twisted cubic curve. Its image $\Psi''\subset F$ 
under $\theta$ is also a $G_{\Levi}$-invariant rational twisted cubic curve and 
a section of the ruling $F\to J$. Since $\z(\Levi)\subset G_{\Levi}$, 
see~\xref{nota:diagr-2}, $\Psi''$ is $\z(\Levi)$-invariant. It follows that
$\Psi''$ is a component of the fixed point set of $\z(\Levi)$ in 
$R\setminus\Xi$, hence $\Psi''=\Psi$. The vertex $Q''\in W\setminus R$ of 
$S''_W$ coincides with the unique isolated fixed point $Q$ of $\z(\Levi)$. 
Thus, 
$S''_W=\Cone(\Psi'', Q'')=\Cone(\Psi, Q)=S'_W$. Therefore, $S''=S'$. This 
proves 
the uniqueness in~\xref{nota:diagr-2-a}.

\xref{nota:diagr-2-b} If $\Ru\cap\Aut^0(W,F)=\{1\}$, then 
$\Aut^0(W,F)=G_{\Levi}$, see 
Corollary~\xref{cor:aut-center}\xref{cor:aut-center-a}. By 
Lemma~\xref{lem:center-inv-cone}, the cubic cone $S'_W=\Cone(\Psi, Q)\subset W$ 
is $\Aut^0(W,F)$-invariant. Then $S'$ is $\Aut^0(V,S)$-invariant, because 
diagram~\eqref{diagram-2} is $\Aut^0(W,F)$-equivariant.

\xref{nota:diagr-2-c} If $\Ru\cap\Aut^0(W,F)\neq\{1\}$, then the Sarkisov 
link~\eqref{diagram-2} provides a one-parameter family of cubic cones in $V$ 
disjoint with $S$, see Corollary~\xref{cor:aut-center}\xref{cor:aut-center-c}.
\end{proof}

\begin{construction}
\label{construction}
Let us investigate more closely the rational normal quintic scrolls $F\subset 
R$ 
with $F\cong\FF_1$ and $\Ru\cap\Aut^0(W,F)\neq\{1\}$. We use the notation from 
Lemma~\xref{lem:hatR} and Proposition~\xref{proposition-GL2-action-c}. The 
proper transform $\hat F$ of $F$ in $\hat W$ (see diagram 
\eqref{equation-diagram-1}) is a smooth scroll which meets 
$\hat\Xi\cong\PP^1\times\PP^1$ along a nondegenerate conic, say, $\hat J$. The 
morphism $\varphi|_{\hat F}:\hat F\to\Gamma$ induces on $\hat F$ a structure of 
a 
$\PP^1$-subbundle of the $\PP^2$-bundle $\varphi|_{\hat R}:\hat 
R\to\Gamma$, see Lemma~\xref{lem:hatR}\xref{lem:hatR-a}. This is the 
projectivization of a rank $2$ vector subbundle $\mathcal{L}\to\Gamma$ of the 
normal bundle $\NNN_{\Gamma/\PP^4}$. In each fiber of $\NNN_{\Gamma/\PP^4}$, 
$\mathcal{L}$ is transversal to the rank $2$ vector subbundle 
$\NNN_{\Gamma/\PP^3}$, where $\PP^3=\PP(M_3\oplus\{0\})$. The pullback 
$\tilde{\mathcal{L}}$ of $\mathcal{L}$ in the tangent bundle $T\PP^4|_{\Gamma}$ 
is a vector subbundle of corank $1$ containing the tangent bundle of $\Gamma$.

Consider the canonical projection $\pi: 
(M_3\oplus\CC)\setminus\{0\}\to\PP(M_3\oplus\CC)\cong\PP^4$ along with the 
affine cone (with the vertex removed) $\Cone(\Gamma)\subset 
(M_3\oplus\CC)\setminus\{0\}$, cf.\ the proof of 
Proposition~\xref{proposition-GL2-action-c}\xref{proposition-GL2-action-c-b}. 
The pullback 
$\tilde{\mathcal{L}}^*:=\pi^*\tilde{\mathcal{L}}$ is a vector subbundle of 
corank $1$ of the trivial vector bundle $T(M_3\oplus\CC)|_{\Cone(\Gamma)}\cong 
(M_3\oplus\CC\setminus\{0\})\times\Cone(\Gamma)$ over $\Cone(\Gamma)$. This 
subbundle is invariant under the natural $\Gm$-action on 
$(M_3\oplus\CC)\setminus\{0\}$ and contains the tangent bundle 
$T(\Cone(\Gamma))$. The fibers of $\tilde{\mathcal{L}}^*$ define a 
one-parameter 
family of hyperplanes $(\tilde{\mathcal{L}}^*_{\gamma})_{\gamma\in\Gamma}$ of 
the vector 5-space $M_3\oplus\CC$ transversal to $\tilde 
E:=M_3\oplus\{0\}=\langle\Cone(\Gamma)\rangle$.
\end{construction}

In the following lemma we provide a criterion as to when $F$ is invariant under
the action of a nontrivial subgroup of the unipotent radical $\Ru=\Ru(\Aut(W))$.

\begin{slemma}
\label{lem:unipotent-actions}
In the notation as in~\xref{construction}, consider the family 
$(H_{\gamma})_{\gamma\in\Gamma}$ of projective planes in the projective 
$3$-space $\PP(\tilde E)=\PP(M_3)$, where 
$H_{\gamma}:=\PP(\tilde{\mathcal{L}}^*_{\gamma}\cap\tilde E)$. Then the 
following holds.
\begin{itemize}

\item
$\Ru\cap\Aut^0(W,F)\neq\{1\}$ if and only if the planes $H_{\gamma}$ pass 
through a common point;

\item
$\dim(\Ru\cap\Aut^0(W,F))\ge 2$ if and only if the planes $H_{\gamma}$ vary in a 
pencil.
\end{itemize}
\end{slemma}

\begin{proof}
The unipotent radical $\Ru$ of $\Aut(W)$ can be identified with the group of 
translations $(M_3,+)\cong\Transl(\CC^4)$ acting on $M_3\oplus\CC$ via
\begin{equation*}
(M_3,+)\ni g: (f,z)\longmapsto (f+z g,z)\in M_3\oplus\CC,
\end{equation*} see
\eqref{eq:action-transl}.
This action is trivial on the
tangent space $T_{(f,0)}\Cone(\Gamma)$, where $(f,0)=((\alpha x+\beta 
y)^3,0)\in\Cone(\Gamma)$.
The tangent action on
$T_{(f,0)}(M_3\oplus\CC)\cong M_3\oplus\CC$ is given by
$dg: (u,v)\longmapsto (u+v g,v)$, see~\eqref{eq:action-tang-transl}.

Let $\gamma=\pi(f,0)\in\Gamma$. We claim that $dg|(T_{(f,0)}(M_3\oplus\CC))$ 
preserves the hyperplane $\tilde{\mathcal{L}}^*_{\gamma}\subset 
T_{(f,0)}(M_3\oplus\CC)\cong M_3\oplus\CC$ if and only if $(g,0)\in 
\tilde{\mathcal{L}}^*_{\gamma}\cap\tilde E$. Indeed, 
$\tilde{\mathcal{L}}^*_{\gamma}$ can be given in 
$T_{(f,0)}(M_3\oplus\CC)=M_3\oplus\CC$ by equation of the form $\langle 
u,w\rangle+\delta v=0$, where $w\in\tilde E=M_3\oplus\{0\}$ is a nonzero vector 
and $\delta\in\CC$. Let $(u,v)\in\tilde{\mathcal{L}}^*_{\gamma}$. Then 
$dg(u,v)=(u+v g,v)\in\tilde{\mathcal{L}}^*_{\gamma}$ if and only if $v\cdot 
\langle g,w\rangle=0$. Since $\tilde{\mathcal{L}}^*_{\gamma}\neq\tilde E$ then 
$v\neq 0$ for a general $(u,v)\in\tilde{\mathcal{L}}^*_{\gamma}$. So, 
$dg(u,v)\in\tilde{\mathcal{L}}^*_{\gamma}$ for any $(u,v)\in 
\tilde{\mathcal{L}}^*_{\gamma}$ if and only if $(g,0)\in\ 
\tilde{\mathcal{L}}^*_{\gamma}\cap\tilde E$. Now the claim follows.

Therefore, $dg$ preserves the subbundle $\tilde{\mathcal{L}}^*\subset 
T(M_3\oplus\CC)|_{\Cone(\Gamma)}$ if and only if $(g,0)\in 
\left(\bigcap_{\gamma\in\Gamma}\tilde{\mathcal{L}}^*_{\gamma}\right)\cap\tilde 
E$. Applying this with $g\neq 0$ yields the first assertion. To show the 
second, 
if suffices to apply the same argument to a pair of non-collinear vectors 
$(g_1,0)$ and $(g_2,0)$ from $\tilde E=M_3\oplus\{0\}\cong\Ru$.
\end{proof}

The following simple lemma should be classically known. For a lack of reference, 
we provide an elementary argument.

\begin{slemma}
\label{lem:quartic-scroll}
Consider the tangent developable quartic surface $\Delta_0\subset\PP^3$ of the 
twisted cubic curve $\Gamma$, cf.~\xref{nota:3-forms}. Then any line $l$ on 
$\Delta_0$ is a tangent line to $\Gamma$.
\end{slemma}

\begin{proof}
Assuming the contrary, consider the projection $\pi:\PP^3\dasharrow\PP^2$ 
with center at a general point of $l$. The image $\pi(\Gamma)$ is a rational 
plane cubic such that the tangent lines to $\pi(\Gamma)$ at smooth points pass 
all through the same point $\pi(l)$. However, this cannot happen in 
characteristic zero by the duality argument.
\end{proof}

\begin{scorollary}
\label{cor:unipotent-actions}
Let $F\subset R$ be a rational normal quintic scrolls such that
$J=F\cap\Xi$ is a smooth conic.
Then $\dim(\Ru\cap\Aut^0(W,F))\le 1$.
\end{scorollary}

\begin{proof}\
Assuming the contrary, by Lemma~\xref{lem:unipotent-actions} the planes 
$H_{\gamma}$, $\gamma\in\Gamma$, contain a common line $l$. Besides, each plane 
$H_{\gamma}$ contains the tangent line $l_{\gamma}$ to $\Gamma$ at the point 
$\gamma\in\Gamma$, see~\xref{construction}. Hence $l$ meets each $l_{\gamma}$. 
Since distinct tangent lines $l_{\gamma}$ are disjoint, $l\subset\Delta_0$ is 
different from any $l_{\gamma}$. Now Lemma~\xref{lem:quartic-scroll} gives a 
contradiction.
\end{proof}

\section{$\G$-construction}
\label{appendix}
In this section we exploit the description of the Fano-Mukai fourfolds of genus 
$10$ based on a beautiful construction of S.~Mukai (\cite{Mukai-1988}), see 
Theorem~\xref{thm:Mukai}. Using a lemma due to Kapustka and Ranestad 
\cite[Lem.~1]{KapustkaRanestad2013} we find the stabilizers of elements in the 
Lie algebra of the exceptional group $\G$, and interpret these in terms of the 
automorphism groups of Fano-Mukai fourfolds $V_{18}$. Let us start by recalling 
the following result.

\begin{theorem}[{\cite{Mukai-1989}}]
\label{thm:Mukai}
Let
$\G$ be the simple algebraic group of exceptional type $\G$.
Consider the adjoint variety $\Omega=\G/P\subset\PP(\mathfrak{g}_2)=\PP^{13}$, 
where $P\subset\G$ is a parabolic subgroup of dimension $9$
corresponding to a long root, and $\mathfrak{g}_2$ is the Lie algebra of $\G$.
Then any Fano-Mukai fourfold of genus $10$ is isomorphic to a hyperplane
section
of the homogeneous fivefold $\Omega$.
\end{theorem}

\begin{notation}
We let $\mathfrak{g}_2$ be the Lie algebra of $\G$,
$\mathfrak{h}$ be a fixed Cartan subalgebra of $\mathfrak{g}_2$, and 
$\Delta\subset\mathfrak{h}^\vee$ be the corresponding root system.
Choose a base of simple roots $\boldsymbol\alpha_1,\,\boldsymbol\alpha_2$ of 
$\Delta$ satisfying (\cite{Bourbaki2002})
\begin{equation*}
(\boldsymbol\alpha_1,\boldsymbol\alpha_1)=2,\quad
(\boldsymbol\alpha_2,\boldsymbol\alpha_2)=6,\quad
(\boldsymbol\alpha_1,\boldsymbol\alpha_2)=-3.
\end{equation*}
The Dynkin diagram $\G$ looks like
\begin{equation}
\label{Dynkin-G2}
\begin{tikzpicture}[scale=1]
\node[label=below:$\boldsymbol\alpha_1$] at (0,-0.20) {};
\node[label=below:$\boldsymbol\alpha_2$] at (4,-0.20) {};

\draw[thick,fill=white]
(0,0) circle [radius=.2]
(4,0) circle [radius=.2];

\draw [thick]
(0,-.2) --++(4,0)
(0.2,0) --++(3.6,0)
(0,+.2) --++(4,0);

\draw[thick]
(2,0) --++(-60:.4)
(2,0) --++(60:.4);
\end{tikzpicture}
\end{equation}
\end{notation}

\begin{mdefinition}
Recall that $\G$ has dimension $14$, rank $2$, and its center is trivial.
It can be realized as the automorphism group of the octonion algebra 
$\OO$ over $\CC$. Let $\OO_0\subset\OO$ be the hyperplane of 
pure octonions. Then the adjoint variety $\Omega$ can be realized as the 
subvariety in the Grassmannian 
$\Gr(2,\OO_0)$ of the isotropic two-dimensional subspaces $\Lambda\subset\OO_0$ 
with respect to the multiplication in $\OO$. The latter means 
that the multiplication restricts as zero to such a subspace.

In the presentation $\Omega=\G/P$ (cf.~Theorem~\xref{thm:Mukai}) one can choose 
for $P$
the parabolic subgroup of $\G$ with Lie algebra
\begin{equation*}
\mathfrak{p}=\mathfrak{h}
\oplus\left(\bigoplus_{\boldsymbol\alpha\in\Delta^+}\mathfrak{g}_{
\boldsymbol\alpha}\right)
\oplus\mathfrak{g}_{-\boldsymbol\alpha_1}=\mathfrak{b}\oplus\mathfrak{g}_{
-\boldsymbol\alpha_1},
\end{equation*}
where $\mathfrak{b}$ stands for a Borel subalgebra of $\mathfrak{g}_2$. 
A reductive Levi subgroup of $P$ is, e.g., the $\GL_2(\CC)$-subgroup with Lie 
algebra 
$\mathfrak{h}\oplus\mathfrak{g}_{\boldsymbol\alpha_1}\oplus\mathfrak{g}_{
-\boldsymbol\alpha_1}$. Respectively, the $\SL_2(\CC)$-subgroup with Lie algebra
$[\mathfrak{g}_{\boldsymbol\alpha_1},
\mathfrak{g}_{-\boldsymbol\alpha_1}]\oplus\mathfrak{g}_{\boldsymbol\alpha_1}
\oplus\mathfrak{g}_{-\boldsymbol\alpha_1}$ is a semisimple
Levi subgroup of $P$
(cf.~\cite[\S~1]{Mukai-1988}, \cite[\S~1]{Hwang-Mok-2002}, and 
also~\cite[\S~2.3]{LandsbergManivel2003}).

\begin{mdefinition}
\label{sit:Omega}
By~\cite{Mukai-1988}, $\Omega$ is a Fano-Mukai fivefold of Fano index $3$ with 
$\rk\Pic(\Omega)=1$. The linear system $\left|-\frac 13 K_{\Omega}\right|$ 
defines an embedding $\Omega\hookrightarrow\PP^{13}$ onto a variety of degree 
$\deg\Omega=\left(-\frac 13K_{\Omega}\right)^5=18$. By the adjunction formula 
and the Lefschetz hyperplane section theorem, any smooth hyperplane section 
$V=\Omega\cap\PP^{12}\subset\PP^{13}$ is a Fano-Mukai fourfold of genus $10$. 
According to Theorem~\xref{thm:Mukai} any such variety occurs to be a 
hyperplane 
section of $\Omega\subset\PP^{13}$. 
\end{mdefinition}

By definition, the adjoint variety $\Omega$ is
the orbit of the highest weight vector in the projectivized adjoint
representation of $\G$. This is the projectivized minimal (nonzero) 
nilpotent orbit of the $\Ad(\G)$-action on $\mathfrak{g}_2$ (\cite[Thm.\ 4.3.3]{Collingwood-McGovern1993}).
The adjoint variety $\Omega$
contains the points $\PP(\mathfrak{g}_{\boldsymbol\alpha})$, where $\boldsymbol\alpha$ 
runs over the set $\Delta_{\ol}\subset\Delta$ of long roots 
(\cite[\S~8.5]{Tevelev2005}).
The dual variety $D_{\ol}:=\Omega^*\subset
\PP(\mathfrak{g}_2)^\vee$ of $\Omega$, that is, the locus of hyperplanes in 
$\PP(\mathfrak{g}_2)$
which define singular hyperplane sections of $\Omega$, is $\Ad(\G)$-invariant. 
In fact, $D_{\ol}$
is an irreducible hypersurface in $\PP(\mathfrak{g}_2^\vee)$ of degree 6. The 
latter follows, e.g., from the classification of defective varieties of small 
dimension (see
\cite[\S~7.3.3]{Tevelev2005}), or, alternatively, by 
comparing~\cite[Thm.~7.47]{Tevelev2005} with 
Proposition~\xref{proposition-Omega}\xref{proposition-Omega-4} below.

We let $\delta_{\ol}$ be a homogeneous polynomial of degree 6 defining 
$D_{\ol}$. There exists a second $\Ad(\G)$-invariant sextic hypersurface 
$D_{\s}$ in $\PP(\mathfrak{g}_2^\vee)$ given by a homogeneous polynomial 
$\delta_{\s}$ of degree 6, where $\delta_{\ol}$ and $\delta_{\s}$ are elements 
of 
the graded subalgebra $\CC[\mathfrak{g}_2]^{\Ad(\G)}$ of $\Ad(\G)$-invariant 
functions on $\mathfrak{g}_2$, see~\cite[\S~8.5]{Tevelev2005}. Indeed, let 
$\CC[\mathfrak{h}]^{\W}$ be the graded subalgebra of ${\W}$-invariant functions 
on $\mathfrak{h}$, where ${\W}$ is the Weyl group of $\G$. It is known (see, e.g., 
\cite[Ch. VI, \S~12, Table IX, Summary 32-33]{Bourbaki2002}) that 
$\CC[\mathfrak{h}]^{\W}=\CC[\operatorname f_2,\operatorname f_6]$, where $\deg 
f_2=2$ and $\deg f_6=6$. In particular, the graded piece 
$(\CC[\mathfrak{h}]^{\W})_6$ of dimension $2$ is spanned by $f_2^3$ and $f_6$. 
By 
the Chevalley restriction theorem (\cite[Thm.~3.1.38]{ChrissGinzburg-1997}) the 
restriction map gives an isomorphism of graded algebras
\begin{equation}
\label{eq-invariants}
\CC[\mathfrak{g}_2]^{\Ad(\G)}\cong\CC[\mathfrak{h}]^{\W}.
\end{equation}
Hence
$(\CC[\mathfrak{g}_2]^{\Ad(\G)})_6\cong\langle f_2^3,f_6\rangle$ is spanned by 
a 
pair $(\delta_{\ol},\,\delta_{\s})$ of irreducible invariants chosen so that 
(\cite[Thm.~8.25]{Tevelev2005})
\begin{equation}
\label{equation-discriminant}
\delta_{\ol}|_{\mathfrak{h}}=\prod_{\boldsymbol\alpha\in 
\Delta_{\ol}}\boldsymbol\alpha\quad\mbox{and}\quad 
\delta_{\s}|_{\mathfrak{h}}=\prod_{\boldsymbol\beta\in 
\Delta_{\s}}\boldsymbol\beta,
\end{equation}
where $\Delta_{\s}$ stands for the set of short roots.
\end{mdefinition}

\begin{mdefinition}
\label{sit:adjoint-representation}
Let $G$ be a simple algebraic group of rank $r$ with Lie algebra $\Lie(G)$. For 
an element $g\in\Lie(G)$ we let $\Stab_{G}(g)$ stand for the stabilizer of 
$g$ in $G$ under the adjoint representation. The Lie algebra of $\Stab_{G}(g)$ 
is the centralizer of $g$ in $\Lie(G)$. The element $g$ is called 
\emph{regular} 
if $\dim\Stab_{G}(g)=r$ and \emph{singular} otherwise. We have the following 
facts.
\end{mdefinition}

\begin{sproposition}
\label{prop:regular-centralizers}
For an element $g\in\Lie(G)$ the following hold.
\begin{enumerate}
\renewcommand\labelenumi{\rm (\alph{enumi})}
\renewcommand\theenumi{\rm (\alph{enumi})}
\item
\label{prop:regular-centralizers-a} 
\textup{(}\cite[1.3]{SpringerSteinberg1970}\textup{)} The stabilizer 
$\Stab_{G}(g)$ contains an abelian subgroup of dimension $r=\rk G$.

\item
\label{prop:regular-centralizers-b} 
\textup{(}\cite[Thms.~A and $\alpha$]{Kurtzke1983}; cf.~\cite[1.4, 
1.16-1.18]{SpringerSteinberg1970}\textup{)} $g$ is regular if and only if
the stabilizer $\Stab_{G}(g)$ is abelian.

\item
\label{prop:regular-centralizers-c} 
\textup{(}\cite[Cor.~3.4]{Steinberg1965}\textup{)} 
If $g$ is singular, then $\dim\Stab_{G}(g)\ge r+2$.

\item
\label{prop:regular-centralizers-d} 
\textup{(}\cite[1.7, 3.9]{SpringerSteinberg1970}\textup{)} 
If $g$ is semisimple, then $\Stab_{G}(g)$ is a reductive group. A semisimple $g$ 
is regular if and only if $\Stab_{G}(g)^0$ is a maximal torus of $G$.
\end{enumerate}
\end{sproposition}

\begin{sdefinition}\label{subregular} 
An element $g\in \Lie(G)$ with $\dim \Stab_G(g)=r+2$ is called 
\emph{subregular}. The orbit of $g$ in $\Lie(G)$ is also called 
\emph{subregular}. It should be noted that an orbit $\mathcal{O}$ of the 
adjoint $G$-action on $\Lie(G)$ is conic if and only if $\mathcal{O}$ is a 
nilpotent orbit (\cite[Lem.\ 4.3.1]{Collingwood-McGovern1993}). The image of a 
nilpotent orbit $\mathcal{O}$ in $\PP(\Lie(G))$ is an orbit of the same 
codimension as the one of $\mathcal{O}$ in $\Lie(G)$. In contrast, for a 
non-nilpotent orbit $\mathcal{O}\subset \Lie(G)$, its image in $\PP(\Lie(G))$ 
is an orbit of the same dimension as the one of $\mathcal{O}$.
\end{sdefinition}

\begin{notation}
\label{nota:g-notation}
We identify $\mathfrak{g}_2$ and $\mathfrak{g}_2^\vee$ via the Killing form. 
Under this identification, the adjoint representation $\Ad(\G)$ and its dual 
have the same orbits.

For a nonzero element $g\in\mathfrak{g}_2$ we let $[g]$ be the image of $g$ in 
$\PP(\mathfrak{g}_2)$. Let $V^g_{18}=\Omega\cap\PP(g_2^\bot)$ be the section of 
$\Omega$ by the projectivized hyperplane $g^\bot$ orthogonal to $g$ with respect 
to the Killing form.
\end{notation}

Using Proposition~\xref{prop:regular-centralizers} and Lemma~1 
in~\cite{KapustkaRanestad2013} we deduce the following results.

\begin{proposition}[{\cite{KapustkaRanestad2013}}]
\label{prop:orbits}
Consider the pencil $\mathcal{D}$ of
$\Ad(\G)$-invariant sextic hypersurfaces in 
$\PP(\mathfrak{g}_2^\vee)=\PP(\mathfrak{g}_2)$ generated by
$D_{\ol}$ and $D_{\s}$.
Then the following hold.
\begin{enumerate}
\renewcommand\labelenumi{\rm (\alph{enumi})}
\renewcommand\theenumi{\rm (\alph{enumi})}
\item
\label{prop:orbits-a}
For any $D\in\mathcal{D}$ different from $D_{\ol}=\delta_{\ol}^*(0)$ and 
$D_{\s}=\delta_{\s}^*(0)$ the complement $D\setminus D_{\ol}$ is the 
$\Ad(\G)$-orbit of $[g]$, where $g\in\mathfrak{g}_2$ is regular semisimple with 
an abelian group
$\Stab_{\G}(g)$ and with
$\Stab_{\G} (g)^0\cong (\Gm)^2$.

\item
\label{prop:orbits-b} The complement $D_{\s}\setminus D_{\ol}$ is a union of 
exactly two $\Ad(\G)$-orbits, the orbit of $[g_{\s}]$ and the orbit of 
$[g_{\s}+g_{\n}]$, where
$g_{\s},g_{\n}\in\mathfrak{g}_2$ are commuting elements such that
\begin{itemize}
\item
$g_{\s}$ is subregular semisimple, $g_{\n}\neq 0$ is nilpotent, and 
$g_{\s}+g_{\n}$ is regular;

\item
the orbit $\Ad(\G).[g_{\s}+g_{\n}]$ is an open, dense subvariety of 
$D_{\s}\setminus D_{\ol}$;

\item
the orbit $\Ad(\G).[g_{\s}]$ is a closed subvariety of $D_{\s}\setminus D_{\ol}$ 
of codimension $2$;

\item
the stabilizer $\Stab_{\G}(g_{\s}+g_{\n})$ is abelian, and 
$\Stab_{\G}(g_{\s}+g_{\n})^0\cong\Ga\times\Gm$;

\item
the stabilizer $\Stab_{\G}(g_{\s})$ is a non-abelian reductive group of 
dimension $4$ and of rank $2$.
\end{itemize}
\end{enumerate}
\end{proposition}

\begin{proof}
\xref{prop:orbits-a} 
By~\cite[Lem.~1(2)]{KapustkaRanestad2013}, for $D\neq D_{\s}, D_{\ol}$, the 
complement $D\setminus D_{\ol}$ coincides with the orbit of some 
$[g]\in\PP(\mathfrak{g}_2)$. On the other hand (\cite[\S~8.5]{Tevelev2005}), 
$g\in\mathfrak{g}_2$ is regular semisimple if and only if $[g]\notin D_{\ol}\cup 
D_{\s}$. Now~\xref{prop:orbits-a} follows by 
Proposition~\xref{prop:regular-centralizers}\xref{prop:regular-centralizers-b} 
and~\xref{prop:regular-centralizers-d}.

\xref{prop:orbits-b} By~\cite[Lem.~1(1)]{KapustkaRanestad2013} and its proof,
$D\setminus D_{\ol}$ is a union of 
a semisimple subregular orbit $O_1=\Ad(\G).[g_{\s}]$ and a mixed regular orbit $O_2=\Ad(\G).[g_{\s}+g_{\n}]$, 
where $g_{\n}\neq 0$ is nilpotent and
commutes with $g_{\s}$. 
Thus, $\codim_{\PP(\mathfrak{g}_2)} O_1=3$ and
$\codim_{\PP(\mathfrak{g}_2)} O_2=1$, see \ref{subregular}. In particular, $O_2$ is open and dense in
$D_{\s}\setminus D_{\ol}$, and its complement $O_1$ in
$D_{\s}\setminus D_{\ol}$ is closed of codimension $2$, see~\cite[Rem.~1]{KapustkaRanestad2013}. 
Since $g_{\s}$ is subregular, $\Stab_{\G}(g_{\s})$ is a
non-abelian reductive group of dimension $4$ and of rank $2$ with semisimple 
part $\SL_2(\CC)$, see (\cite{Collingwood-McGovern1993}) and 
Proposition~\xref{prop:regular-centralizers}~\xref{prop:regular-centralizers-d}. 

Since $g_{\s}+g_{\n}$ is regular,
$\Stab_{\G}(g_{\s}+g_{\n})$ is an abelian group of dimension $2$, see 
Proposition~\xref{prop:regular-centralizers}~\xref{prop:regular-centralizers-b}.
The unipotent radical of
$\Stab_{\G}(g_{\s}+g_{\n})^0$ contains the one-parameter unipotent
subgroup $\exp(t g_{\n})$, $t\in\CC$. On the other hand, the semisimple element
$g_{\s}$ centralizes $g_{\s}+g_{\n}$, hence
$\Stab_{\G}(g_{\s}+g_{\n})^0$ is not unipotent. It follows that
$\Stab_{\G}(g_{\s}+g_{\n})^0\cong\Ga\times\Gm$.
\end{proof}

\begin{lemma}
\label{lem:stabilizers-in-Omega}
The stabilizer $\Stab_{\G}(g)$ acts effectively on the hyperplane section 
$V=V^g_{18}$.
\end{lemma}

\noindent
The proof starts with the following claim.
\begin{sclaim}
\label{claim:reductive}
The kernel of the natural homomorphism $\Stab_{\G}(g)\to\Aut(V)$ is contained in 
the unipotent radical of $\Stab_{\G}(g)$.
\end{sclaim}

\begin{proof}
The simple group $\G$ acts effectively on $\Omega$. If $h\in\Stab_{\G}(g)$ acts 
trivially on $V$, then it acts trivially on the hyperplane $\langle 
V\rangle\subset\PP(\mathfrak{g}_2)$. Hence the adjoint action of $h$ on 
$\mathfrak{g}_2$ has at most two distinct eigenvalues $\lambda_1$ and 
$\lambda_2$ of multiplicity 13 and 1, respectively. Since $h(g)=g$, one has 
$\lambda_2=1$. The Lie algebra $\mathfrak{g}_2$ being semisimple, one has
$\det h=1=\lambda_1^{13}$. Up to equivalence, the adjoint representation is a 
unique irreducible representation of $\G$ of dimension 14. Hence it is 
equivalent to the coadjoint representation. It follows that $h$ and $h^{-1}$ are 
conjugated, and so, $\tr(h)=\tr(h^{-1})$. Thus, 
$13\lambda_1+1=13\lambda_1^{-1}+1$, which implies $\lambda_1=1$. Then the 
semisimple part of $h$ equals 1, and $h$ is unipotent.
Finally, the kernel of the representation $\Stab_{\G}(g)\to\Aut(V)$ is a normal 
closed subgroup of $\Stab_{\G}(g)$ consisting of unipotent elements.
\end{proof}

The next claim ends the proof of Lemma~\xref{lem:stabilizers-in-Omega}.

\begin{sclaim}
\label{claim: unipotent}
The unipotent radical of $\Stab_{\G}(g)$ acts effectively on $V$.
\end{sclaim}

\begin{proof}
It suffices to show that any one-parameter unipotent subgroup 
$U=\{\exp(tN)\}_{t\in\CC}\subset\Stab_{\G}(g)$, where $N\in\mathfrak{g}_2$ is 
nilpotent, acts effectively on $V$. Suppose this is not the case. Then $N$ 
vanishes on the affine hyperplane $ g^\bot\subset\mathfrak{g}_2$. Since it also 
vanishes on the line $\CC g$, one has $N=0$.
\end{proof}

\begin{scorollary}
\label{cor:aut-orbits}
Let $V=V^g_{18}$, where $g\in\mathfrak{g}_2\setminus\{0\}$ and $[g]\notin 
D_{\ol}$. Assume 
that $V$ contains an $\Aut^0(V)$-invariant cubic cone $S$. Then the following 
hold.
\begin{enumerate}
\item
\label{cor:aut-orbits-i}
If $g$
is singular semisimple, then $\Aut^0(V)\cong\GL_2(\CC)$;

\item
\label{cor:aut-orbits-ii}
if $g$
is regular
non-semisimple, then $\Aut^0(V)\supset\Ga\times\Gm$;

\item
\label{cor:aut-orbits-iii}
if $g$
is regular
semisimple, then $\Aut^0(V)\supset (\Gm)^2$.

\end{enumerate}
In particular, if $\rk (\Aut^0(V))=1$
then we are in case~\xref{cor:aut-orbits-ii}.
Any two Fano-Mukai fourfolds satisfying~\xref{cor:aut-orbits-i} 
\textup{(}resp.,~\xref{cor:aut-orbits-ii}\textup{)} are isomorphic via an 
automorphisms of $\Omega$ provided by the $\Ad(\G)$-action on $\Omega$.
\end{scorollary}

\begin{proof}
Statements~\xref{cor:aut-orbits-ii} and~\xref{cor:aut-orbits-iii} are 
straightforward from 
Propositions~\xref{prop:regular-centralizers}\xref{prop:regular-centralizers-a}
-\xref{prop:regular-centralizers-b} and~\xref{prop:orbits} and 
Lemma~\xref{lem:stabilizers-in-Omega}. By 
Propositions~\xref{prop:regular-centralizers}\xref{prop:regular-centralizers-b}
-\xref{prop:regular-centralizers-c} and~\xref{prop:orbits}\xref{prop:orbits-b} 
and Lemma~\xref{lem:stabilizers-in-Omega}, in case~\xref{cor:aut-orbits-i} 
$\Aut^0(V)$ contains a non-abelian reductive subgroup of dimension $4$ and of 
rank~$2$. Hence $V$ admits a nontrivial $\SL_2(\CC)$-action. Let $(W,F)$ be the 
pair 
linked to $(V,S)$ via an $\Aut^0(V)$-equivariant Sarkisov link 
\eqref{diagram-2}. Due to Lemma~\xref{lem:j=upsilon} one has 
$\Aut^0(V)\cong\Aut^0(W,F)\cong\GL_2(\CC)$. 

By~\xref{cor:aut-orbits-i}--\xref{cor:aut-orbits-iii} the equality $\rk 
(\Aut^0(V))=1$ implies that $g$ is regular non-semisimple. According to 
Proposition~\xref{prop:orbits}\xref{prop:orbits-b} any two such elements 
$g_1,\,g_2\in\mathfrak{g}_2$ belong to the same $\Ad(\G)$-orbit. Hence the 
corresponding smooth hyperplane sections $V_{18}^{g_1}$ and $V_{18}^{g_2}$ of 
$\Omega$ are isomorphic under the $\Ad(\G)$-action on $\Omega$. The same 
argument applies in case~\xref{cor:aut-orbits-i} (alternatively, see 
Lemma~\xref{cor:unique-GL2}).
\end{proof}

\begin{snotation}
\label{nota:aut-orbits}
We let $V_{18}^{\s}$ and $V_{18}^{\aaa}$, respectively denote a unique, up to 
isomorphism, Fano-Mukai fourfold $V^g_{18}$ as in 
Corollary~\xref{cor:aut-orbits}\xref{cor:aut-orbits-i} and~\xref{cor:aut-orbits-ii}, respectively. The existence of 
these fourfolds is guaranteed by 
Proposition~\xref{prop:orbits}\xref{prop:orbits-b}. Thus, 
$\Aut^0(V_{18}^{\s})\cong\GL_2(\CC)$ and 
$\Aut^0(V_{18}^{\aaa})\supset\Ga\times\Gm$.
\end{snotation}

\section{Lines in $V_{18}$}
\label{sec:lines-in-V}

To study the lines in a Fano-Mukai fourfold $V_{18}$,
let us first investigate the Fano variety of lines in
the homogeneous space $\Omega=\G/P$.
\begin{proposition}
\label{proposition-Omega}
Let $\Omega=\G/P$ be as above.
Then the following hold.

\begin{enumerate}
\renewcommand\labelenumi{\rm (\alph{enumi})}
\renewcommand\theenumi{\rm (\alph{enumi})}
\item
\label{proposition-Omega-0}
$\Omega$ contains a line.

\item
\label{proposition-Omega-1}
For any line $l\subset\Omega$ one has
\begin{equation}
\label{equation-normal-bundle-Omega}
\NNN_{l/\Omega}\cong
\OOO_{\PP^1}\oplus\OOO_{\PP^1}\oplus\OOO_{\PP^1}\oplus\OOO_{\PP^1}(1).
\end{equation}

\item
\label{proposition-Omega-2}
The Hilbert scheme $\Sigma(\Omega)$ of lines on $\Omega$ is a non-singular
variety of
dimension $5$.

\item
\label{proposition-Omega-3}
Consider the universal family of lines on $\Omega$
\begin{equation}
\label{proposition-Omega-3-diagram}
\vcenter{
\xymatrix@R=10pt{
&\LLL(\Omega)\ar[dl]_{}\ar[dr]^{}&
\\
\Sigma(\Omega)&&\Omega
}}
\end{equation}
where $\LLL(\Omega)\to\Sigma(\Omega)$ is a $\PP^1$-bundle. Then
any fiber of $\LLL(\Omega)\to\Omega$ is one-dimensional.

\item
\label{proposition-Omega-4}
For any point $p\in\Omega$ the union $L\Omega_p$ of the lines in $\Omega$
passing
through $p$ is a cone over a rational twisted cubic curve.
\end{enumerate}
\end{proposition}

\begin{proof}
\xref{proposition-Omega-0}
It suffices to show that a general section $\Omega\cap\PP^{11}$
contains a line. However, the latter follows from Shokurov's theorem
\cite{Shokurov1980}.

\xref{proposition-Omega-1}
Let $l$ be a line on $\Omega\subset\PP^{13}$.
Since $\Omega$ is a homogeneous space, the tangent bundle $T_\Omega$ is
generated by global sections (that is, the vector fields from the corresponding 
Lie
algebra).
Hence the vector bundles $T_\Omega|_l$ and $\NNN_{l/\Omega}=T_\Omega|_l/T_l$
are
also
generated by global sections.
This means that the
integers $a,b,c,d$ in the decomposition
\begin{equation*}
\NNN_{l/\Omega}=\OOO_{\PP^1}(a)\oplus
\OOO_{\PP^1}(b)\oplus\OOO_{\PP^1}(c)\oplus\OOO_{\PP^1}(d)
\end{equation*}
are all non-negative.
On the other hand,
\begin{equation*}
\deg\NNN_{l/\Omega}=-K_{\Omega}\cdot l -2+2\operatorname{g}(l)=1.
\end{equation*}
This implies~\eqref{equation-normal-bundle-Omega}.

\xref{proposition-Omega-2}
follows from~\eqref{equation-normal-bundle-Omega} and the
standard facts of the deformation theory. Indeed, we have
$H^1(V,\NNN_{l/\Omega})=0$
and $\dim H^0(V,\NNN_{l/\Omega})=5$. Statement~\xref{proposition-Omega-3}
follows from the facts that $\Omega$ is homogeneous and
the diagram~\eqref{proposition-Omega-3-diagram} is equivariant.
To show~\xref{proposition-Omega-4} we note that $L\Omega_p=T_p\Omega\cap\Omega$
is the tangent cone and $T_p\Omega\cap\Omega$ is a cone over a rational twisted
cubic curve, see \cite[\S~1, Prop.~1]{Hwang-Mok-2002} or
\cite[Proof of Lemma~3]{KapustkaRanestad2013}.
\end{proof}

Next we study the lines in a Fano-Mukai fourfold $V_{18}$.

\begin{proposition}
\label{lem:a-d}
Let $V=V_{18}\subset\PP^{12}$ be
a smooth Fano-Mukai fourfold of genus $10$.
Then the following hold.

\begin{enumerate}
\renewcommand\labelenumi{\rm (\alph{enumi})}
\renewcommand\theenumi{\rm (\alph{enumi})}

\item
\label{lem:a-d(00)}
$V$ contains a line.

\item
\label{lem:a-d(0)}
For any line $l\subset V$ one has
\begin{equation}
\label{equation-normal-bundle-V}
\NNN_{l/V}\cong
\begin{cases}
\OOO_{\PP^1}\oplus\OOO_{\PP^1}\oplus\OOO_{\PP^1}& (*)\qquad\text{or}
\\
\OOO_{\PP^1}\oplus\OOO_{\PP^1}(-1)\oplus\OOO_{\PP^1}(1)&(**)
\end{cases}
\end{equation}

\item
\label{lem:a-d(a)}
The Hilbert scheme $\Sigma(V)$ of lines on $V$ is a non-singular variety of
dimension $3$.

\item
\label{lem:a-d(b)}
Consider the universal family of lines on $V$
\begin{equation}
\label{equation-universal-family-V}
\vcenter{
\xymatrix@R=10pt{
&\LLL(V)\ar[dl]_{r}\ar[dr]^{s}&
\\
\Sigma(V)&&V
}}
\end{equation}
where $r:\LLL(V)\to\Sigma(V)$ is a $\PP^1$-bundle. Then $s$ is a generically
finite morphism of degree $3$. Consequently, $V$ is covered by lines.

\item
\label{lem:a-d(d)}
For $p\in V$, let $LV_p$ be the union of all lines in $V$ passing through $p$.
If $\dim LV_p=2$, then $LV_p$ is a cubic cone.
Otherwise, $LV_p$ is nonempty and consists of at most three lines.

\item
\label{lem:a-d(c)}
A general line $l$ on $V$ has normal bundle of type $(*)$ in
\eqref{equation-normal-bundle-V}.
Any line $l$ on $V$ with normal bundle of type $(**)$
is contained in the branching divisor $\mathcal{B}\subset V$ of $s$.

\item
\label{lem:a-d(e)}
Any line $l$ on $V$ contained in a cubic cone
has normal bundle of type $(**)$.
\end{enumerate}
\end{proposition}

\begin{proof}
\xref{lem:a-d(00)}
The proof is similar to the one of~\xref{proposition-Omega}
\xref{proposition-Omega-0}.

\xref{lem:a-d(0)}
From the standard exact
sequence
\begin{equation*}
\xymatrix@R=5pt{
0\ar[r]&\NNN_{l/V}\ar[r]&\NNN_{l/\Omega}\ar[r]&\NNN_{V/\Omega}|_{l}\ar@{=}[d]
\ar[r]& 0
\\
&&&\OOO_{\PP^1}(1)
}
\end{equation*}
one deduces that all the summands in a decomposition of $\NNN_{l/V}$ into a sum 
of line bundles are of
degree $\le 1$,
and there is at most one summand of degree $1$.
Thus~\eqref{equation-normal-bundle-Omega} implies
\eqref{equation-normal-bundle-V}.

\xref{lem:a-d(a)} follows from~\eqref{equation-normal-bundle-V}.

\xref{lem:a-d(b)}-\xref{lem:a-d(d)}
Recall that $V$ is a section of $\Omega\subset
\PP^{13}$ by a hyperplane $H$. For $p\in V$, the subvariety $LV_p$ is a 
hyperplane
section of $L\Omega_p$, that is,
\begin{equation*}
LV_p=L\Omega_p\cap V=L\Omega_p\cap H.
\end{equation*}
By Proposition~\xref{proposition-Omega}\xref{proposition-Omega-4}, $L\Omega_p$ 
is
a cubic cone, that is, a cone over a
rational twisted cubic curve.
Hence, $LV_p$ is non-empty, $\dim LV_p\le 2$, and $\dim LV_p=2$ if and only if
$LV_p=L\Omega_p$.
Thus, $s$ is dominant. So, this is a generically finite morphism
of degree 3. Indeed, for a general point $p\in V$, $H$ cuts the cubic cone 
$L\Omega_p$ on $V$ with vertex $p$ along three distinct
lines. The last
assertion in~\xref{lem:a-d(d)} is now immediate.

\xref{lem:a-d(c)}
To any line $l\in\Sigma(V)$ there corresponds a smooth rational curve
$l'\subset\LLL(V)$. Clearly, the restriction
$s|_{l'}: {l'}\to l\subset V$ is an isomorphism.
Consider the differential
\begin{equation}\label{equation-ds}
ds:\OOO_{\PP^1}\oplus\OOO_{\PP^1}
\oplus\OOO_{\PP^1}=\NNN_{{l'}/\LLL(V)}\longrightarrow\NNN_{l/V}.
\end{equation}
In case $(**)$ the map $ds$ has cokernel of rank $1$ along $l'$.
Hence $s$ is ramified along $l'$.

\xref{lem:a-d(e)}
Consider a line $l\subset S$, where $S\subset V$ is
a cubic cone. Then $l$ lifts to a ruling $l'$ of a scroll $S'\cong\FF_3$ in
$\LLL(V)$. The map $s|_{S'}: S'\to S$ contracts the exceptional section
of $S'$ to the vertex $p$ of $S$. 
Therefore, the cokernel of $d s$ in~\eqref{equation-ds} in nontrivial at $p$, 
and so,
$\NNN_{l/V}\not\cong\OOO_{\PP^1}\oplus\OOO_{\PP^1}
\oplus\OOO_{\PP^1}$.
\end{proof}

\section{Cubic scrolls in $V_{18}$}
\label{sec:thm-1.2}
In this section we study the cubic scrolls in a Fano-Mukai fourfold
$V=V_{18}$ of genus $10$.
The following facts are proven in~\cite[Prop.~1 and 2 and the proof of 
Prop.~4]{KapustkaRanestad2013}.

\begin{theorem}
\label{prop:KaRa}
For any Fano-Mukai fourfold $V=V_{18}$ of genus $10$ the
following hold.

\begin{enumerate}
\renewcommand\labelenumi{\rm (\alph{enumi})}
\renewcommand\theenumi{\rm (\alph{enumi})}

\item
\label{prop:KaRa-a} 
Let $\Sigma(V)$ be
the Hilbert scheme of lines in $V$. Then $\Sigma(V)$ is isomorphic to
a smooth divisor of bidegree $(1,1)$ on $\PP^2\times\PP^2$.

\item
\label{prop:KaRa-b} Let $\SSS(V)$ be
the Hilbert scheme of cubic scrolls in $V$. Then $\SSS(V)$
is isomorphic to
a disjoint union of two projective planes.
\end{enumerate}
\end{theorem}

\begin{sremarks}
\label{rem:ext-action}
\begin{enumerate}[noitemsep,nolistsep,leftmargin=15pt,itemindent=7pt]
\item\label{rem:ext-action-1}
By the Lefschetz hyperplane section theorem 
we have $\Pic(\Sigma(V))=\ZZ\cdot
[\mathcal{F}_1]\oplus\ZZ\cdot [\mathcal{F}_2]$, where $\mathcal{F}_i$, $i=1,2$, 
are the pull-backs of
lines on the corresponding factors of $\PP^2\times\PP^2$. 

\item\label{rem:ext-action-2}
Note that the embedding $\Sigma(V)\hookrightarrow\PP^2\times\PP^2$ onto a 
smooth divisor $D$ of bidegree $(1,1)$ guaranteed by 
Theorem~\xref{prop:KaRa}\xref{prop:KaRa-a} is defined canonically. Indeed, the 
projections $\pr_i:\Sigma(V)\to\PP^2$, $i=1,2$, to the factors of 
$\PP^{2}\times\PP^{2}$ are extremal contractions. By~\xref{rem:ext-action-1}, 
the projections $\pr_1$ and $\pr_2$ 
are the only extremal contractions of the threefold $\Sigma(V)$. It follows 
that 
the $\Aut^0(\Sigma(V))$-action on $\Sigma(V)$ induces an action of 
$\Aut^0(\Sigma(V))$ on the factors of $\PP^{2}\times\PP^{2}$ making the 
morphisms $\pr_i$, $i=1,2$, equivariant. 

\item\label{rem:ext-action-3}
One can treat $\Sigma(V)\subset (\PP^2)^\vee\times\PP^2$ as the variety of full 
flags in $\PP^{2}$. Using this representation, it can be easily seen that the 
maps $\Aut^0(\Sigma(V))\to\Aut((\PP^2)^\vee))$ and 
$\Aut^0(\Sigma(V))\to\Aut(\PP^{2})$ are isomorphisms.
\end{enumerate}
\end{sremarks}

We use below the following facts.

\begin{lemma}
\label{lemma--SS}
Given a cubic scroll $S\in\SSS(V)$, there is a unique hyperplane section
$A_S$ of $V$ with $\Sing(A_S)=S$. This hyperplane section $A_S$ coincides with
the union of lines in $V$ which meet $S$. Given a point $P\in A_S\setminus S$, 
there is a unique line in $A_S$ through $P$ which meets $S$.
\end{lemma}

\begin{proof}
By Proposition~\xref{prop:reversion} there is a hyperplane section $A=A_S$ of 
$V$
with $\Sing(A_S)=S$. 
The linear projection $\theta: V\dashrightarrow W$ in diagram
\eqref{diagram-2} sends $A$ to a smooth quintic scroll $F\subset W$ contracting
the lines meeting $S$ and not contained in $S$. These lines correspond to the
rulings of $\eta|_{\tilde A}:\tilde A\to F$ in~\eqref{diagram-2}. It follows
that $A=\xi(\tilde A)$ is covered by such lines. Since
through any point in $\tilde A\setminus\tilde B$ passes a unique ruling of
$\eta|_{\tilde A}:\tilde A\to F$, the last statement follows.
\end{proof}

\begin{lemma}
\label{cor:common-exc-section} A line $l$ in $V$ can be a common exceptional
section for at most two smooth cubic scrolls in $V$, and can be contained in at
most finite number of cubic cones. Consequently, the morphism
$s^{-1}(l)\to l$
\textup{(}see diagram~\eqref{equation-universal-family-V}\textup{)} is
generically finite.
\end{lemma}

\begin{proof} We claim that $l$ cannot be a common exceptional section
of three or more smooth cubic scrolls. Indeed,
assuming the contrary, through a general point $p$ of $l$ pass at least 4
distinct lines in $V$. By Proposition~\xref{lem:a-d}\xref{lem:a-d(d)} this 
implies
that $p$ is a vertex of a cubic cone in $V$. So, the cubic cones with vertices
on $l$ form a one-parameter family, say, $\mathcal{C}(l)$. By Lemma
\xref{lemma--SS} for any pair of cubic cones $S',S''$ through $l$ one has
$S''\subset A_{S'}$ and $S'\subset A_{S''}$. Hence for any $S'\in\mathcal{C}(l)$
there is a Zariski dense open subset of $A_{S'}$ swept up by the cubic cones in
$\mathcal{C}(l)$. Therefore, $A_{S'}$ and $S'=\Sing(A_{S'})$ do not depend on 
the
choice of $S'$,
a contradiction.
\end{proof}

By virtue of Theorem~\xref{prop:KaRa} and the next lemma, a general cubic 
scroll 
in $V$ is smooth.\footnote{In~\cite[Cor.~5.13]{Prokhorov-Zaidenberg-2015} we 
constructed a pair $(V,S)$ such that $S\in\SSS(V)$ is a smooth cubic scroll.}

\begin{lemma}
\label{lem:cubic-cones}
Any cubic cone in $V$ is contained in the branching divisor $\mathcal{B}$ of
$s:\LLL(V)\to V$.
The family of cubic cones in $V$ is at most one-dimensional.
\end{lemma}

\begin{proof}
The first statement is immediate from Proposition
\xref{lem:a-d}\xref{lem:a-d(c)}--\xref{lem:a-d(e)}.

To show the second,
suppose the contrary. Then an entire component of $\SSS(V)$, say, $\SSS_i(V)$, 
$i\in\{1,2\}$, consists of cubic cones. By 
Proposition~\xref{lem:a-d}\xref{lem:a-d(d)}
the vertices of these cones are all distinct and cover a surface, say, 
$\mathcal{T}_i\subset\mathcal{B}$.

The pull-back $\tilde{\mathcal{T}_i}=s^{-1}(\mathcal{T}_i)\subset\LLL(V)$ is 
$\PP^1$-fibered over $\mathcal{T}_i$. Indeed, let $S_t\subset\mathcal{B}$ be 
the 
cubic cone with vertex $t\in\mathcal{T}_i$. Then the fiber 
$s_t=s^{-1}(t)\cong\PP^1$ parameterizes the family of lines in $V$ through $t$, 
that is, the family of rulings of $S_t$.\footnote{In fact, $s_t$ is the 
exceptional section of a surface $\tilde S_t\cong\FF_3$ in $\LLL(V)$. The 
restriction $s|_{\tilde S_t}:\tilde S_t\to S_t$ is the minimal resolution of 
singularity of the cubic cone $S_t$.} Thus, $\tilde{\mathcal{T}}_i$ is a 
component of the ramification divisor $\mathcal{R}$ of $s$.

By Proposition~\xref{lem:a-d}\xref{lem:a-d(d)}, a general line in $V$ is not 
contained in $\mathcal{B}$. Hence the image $r(\tilde{\mathcal{T}}_i)$ is a 
proper subvariety in $\Sigma(V)$, see 
diagram~\eqref{equation-universal-family-V}. 
So, $r|_{\tilde{\mathcal{T}}_i}: 
\tilde{\mathcal{T}}_i\to r(\tilde{\mathcal{T}}_i)$ is a $\PP^1$-fibration over 
a 
surface $r(\tilde{\mathcal{T}}_i)\subset\Sigma(V)$. The points of this surface 
parametrize a family of lines in $V$ contained in $\mathcal{T}_i$. So, there 
is a one-parameter family of lines passing through a general point 
$t\in\mathcal{T}_i$. These lines sweep up the cubic cone $S_t$ with vertex $t$, 
see again Proposition~\xref{lem:a-d}\xref{lem:a-d(d)}. Thus, 
$S_t=\mathcal{T}_i$ 
for a general $t\in\mathcal{T}_i$. Hence $S_t$ does not depend on $t$, a 
contradiction.
\end{proof}

\begin{proposition}
\label{lem:fixed-pt}
Each component $\SSS_i(V)$, $i=1,2$, of the Hilbert scheme $\SSS(V)$ of cubic 
scrolls in $V$ \textup{(}see 
Theorem~\xref{prop:KaRa}\xref{prop:KaRa-b}\textup{)} contains at least one 
$\Aut^0(V)$-invariant cubic cone.
\end{proposition}

In the proof we use the following auxiliary results.

\begin{slemma}
\label{lem:no-morphism}
Any morphism $f:\PP^2\to\Sigma(V)$ is constant.
\end{slemma}

\begin{proof}
We regard $\Sigma(V)$ as a smooth divisor in $\PP^2\times\PP^2$ of bidegree
$(1,1)$.
Suppose to the contrary that $\mathcal{T}:=f(\PP^2)\subset\Sigma(V)$ is not a 
point. Then $\dim
\mathcal{T}=2$ and
$f:\PP^2\to\mathcal{T}$ is a finite morphism of degree, say, $d$.
By Remark~\xref{rem:ext-action}\xref{rem:ext-action-1} one has $\mathcal{T}\sim 
a_1\mathcal{F}_1+a_2\mathcal{F}_2$ for some
integers $a_1, a_2\ge 0$. There are relations
\begin{equation}
\label{eq:relations-in-homologies}
\mathcal{F}_1^3=\mathcal{F}_2^3=0,\quad\mathcal{F}_1^2\cdot
\mathcal{F}_2=\mathcal{F}_1\cdot\mathcal{F}_2^2=1.
\end{equation}
Since $\Sigma(V)$ is smooth, in suitable bihomogeneous coordinates
$(x_0:x_1:x_2\, ;\, y_0:y_1:y_2)$ in $\PP^2\times
\PP^2$ the equation of $\Sigma(V)$ can be written as
\begin{equation}
\label{eq:equation-of-Sigma} x_0y_0+x_1y_1+x_2y_2=0.
\end{equation}
Hence any fiber of the projections to the factors
$\pr_i|_{\Sigma(V)}:\Sigma(V)\to\PP^2$, $i=1,2$, is isomorphic to $\PP^1$.
It follows that the restrictions
$\pi_i=\pr_i|_{\mathcal{T}}:\mathcal{T}\to\PP^2$, $i=1,2$, are finite morphisms.
Indeed,
$\pi_i(\mathcal{T})$ is neither a point, nor a curve, since $\PP^2$ does not 
admit any
dominant morphism to a curve. In particular, 
$\mathcal{T}\cdot\mathcal{F}_i^2>0$, $i=1,2$, hence
$a_1,a_2>0$.

The degree of $\pi_1\circ f:\PP^2\to\PP^2$
equals $(\deg f)(\deg\pi_1)=d (\mathcal{T}\cdot\mathcal{F}_1^2)=da_2$, while
$\deg(\pi_2\circ f)=da_1$.
Identifying $\Sigma(V)$ with
a smooth hyperplane section
$D_6\subset\PP^7$ of the image of $\PP^2\times\PP^2$ under the Segre embedding,
we obtain 
$\OOO_{\mathcal{T}}(1)=\OOO_{\mathcal{T}}(\mathcal{F}_1+\mathcal{F}_2)$. It 
follows that
\begin{equation*}
\deg\mathcal{T}=(\mathcal{F}_1+\mathcal{F}_2)^2\cdot\mathcal{T}=3(a_1+a_2).
\end{equation*}
On the other hand, $f^*\OOO_{\mathcal{T}}(\mathcal{F}_1)=\OOO_{\PP^2}(da_2)$ and
$f^*\OOO_{\mathcal{T}}(\mathcal{F}_2)=\OOO_{\PP^2}(da_1)$.
So, $f^*\OOO_{\mathcal{T}}(1)=\OOO_{\PP^2}(da_1+da_2)$, and
\begin{equation*}
\deg\mathcal{T}=\frac 1 d (da_1+da_2)^2=d(a_1+a_2)^2.
\end{equation*}
Thus,
$d(a_1+a_2)=3$.
Up to a transposition in
indices, the only possibility is $d=1$, $a_1=2$, $a_2=1$.
In particular, $f$ and $\pi_1$ are birational, hence $\mathcal{T}\cong\PP^2$.
Since $\mathcal{T}$ is an ample divisor on $\Sigma(V)$, by the Lefschetz 
theorem,
the restriction map $\Pic (\Sigma(V))\to\Pic(\mathcal{T})$ is injective, a 
contradiction.
\end{proof}

\begin{proof}[Proof of Proposition~\xref{lem:fixed-pt}.]
By Lemma~\xref{lem:cubic-cones}, for $i=1,2$
there is an
$\Aut^0(V)$-invariant
Zariski dense, open subset $\mathcal{U}_i$
of $\SSS_i(V)$, whose points correspond to smooth
cubic scrolls in $V$. So, each point in $\SSS_i(V)\setminus\mathcal{U}_i$
corresponds to a cubic cone. Any smooth cubic
scroll contains a unique distinguished line, which is its
exceptional section.
This provides an $\Aut^0(V)$-equivariant rational map
\begin{equation}
\label{eq:scrolls-to-lines-map}
\varsigma:\SSS_i(V)\cong\PP^2\dashrightarrow\Sigma(V)\hookrightarrow\PP^2\times\
\PP^2
\end{equation}
sending
a smooth cubic scroll to its exceptional section. This map is regular
on $\mathcal{U}_i$. By Lemma~\xref{cor:common-exc-section} the
restriction $\varsigma|_{\mathcal{U}_i}:\mathcal{U}_i\to\varsigma(\mathcal
U_i)$ has finite fibers. Hence $\dim\varsigma(\mathcal{U}_i)=2$, $i=1,2$.

In particular, $\varsigma$ is non-constant.
By Lemma~\xref{lem:no-morphism}, $\varsigma$
cannot be a morphism. So, its indeterminacy set is a nonempty subset of $
\SSS_i(V)\setminus\mathcal{U}_i$ of dimension zero. It consists of isolated 
fixed
points of $\Aut^0(V)$ that
correspond to $\Aut^0(V)$-invariant cubic cones in $V$.
\end{proof}

\begin{slemma}
\label{lemma-nul-intersection-of-scrolls}
Two general cubic scrolls in $V$ either are disjoint, or meet
transversely in a single point.
\end{slemma}

\begin{proof}
By Lemma~\xref{lem:cubic-cones} it is enough to deal with a pair of smooth cubic 
scrolls in $V$.
Given a smooth cubic scroll $S$ in $V$, consider the family
$\mathcal{S}(S)$ of all $S'\in\SSS(V)$ such that $S'\subset A_S$.
We claim that $\mathcal{S}(S)$ is at most one-dimensional.
Indeed,
starting with the pair $(V,S)$ one can produce a diagram~\eqref{diagram-2}. If
$\tilde{S'}\subset\tilde{A_S}$ is the proper transform of $S'$ in $\tilde W$,
then $\eta(\tilde{S'})\subset\eta(\tilde{A_S})=F$. Suppose that $\tilde{S'}$
dominates $F$ under $\eta$. Then $\tilde{S'}$ meets each ruling $f_p=\tilde
\eta^{-1}(p)$, $p\in F$, of $\eta:\tilde{A_S}\to F$. Hence $S'$ meets each
line $\xi(f_p)\subset A_S$. Thus, $\xi(f_p)\subset A_{S'}$ for any
$p\in F$. This family of lines covers $A_S$, and so, $A_S=A_{S'}$. Then also the
singular loci $S$ and $S'$ of $A_S=A_{S'}$ coincide, a contradiction.

Thus, $\eta(\tilde{S'})=\theta(S')$ is an irreducible curve in $F$, which can
be either a line, or a smooth conic. Anyway, $\tilde{S'}$ belongs to a
one-parameter family of ruled surfaces in $\tilde{A_S}$. Hence 
$\dim\mathcal{S}(S)\le 1$.

It follows that $S'\not\subset A_S$ for a general pair of cubic scrolls $(S,S')$ 
in $V$. Then $\theta|_{S'}: S'\to
\theta(S')\subset W$ is birational. Hence $\dim(\langle S\rangle\cap\langle
S'\rangle)\le 1$. If $\dim(\langle S\rangle\cap\langle S'\rangle)=1$, then
$\theta (\langle S'\rangle)$ is a plane in $\PP^7$. Then
$\theta(S')\subset W$ is as well a plane, which belongs to a one-parameter 
family of
planes in $W$. Therefore, for a general $S'\in\SSS(V)$ one has $\dim(\langle
S\rangle\cap\langle S'\rangle)\le 0$. So, either $S$ and $S'$ are disjoint, or
they meet transversely in a single point.
\end{proof}

Next we describe the middle cohomology group of $V$.

\begin{proposition}
\label{lem:cohomology}
The group $H^4(V,\ZZ)$ is generated by the classes of two cubic scrolls
$S_1$ and $S_2$ from different components of $\SSS(V)$. In the
cohomology ring $H^*(V,\ZZ)$ these classes satisfy the relations
$[S_1]^2=[S_2]^2=1$, $[S_1]\cdot [S_2]=0$.
\end{proposition}

The proof is preceded by the following lemma.

\begin{slemma}
\label{lem:square}
For any cubic scroll $S$ in $V$ one has $S^2=1$ in $H^*(V,\ZZ)$.
\end{slemma}
\begin{proof}
By Lemma~\xref{lem:cubic-cones} one can choose $S\in\SSS(V)$ to be smooth. 
Consider
the corresponding diagram~\eqref{diagram-2}. For the exceptional divisor
$\tilde B\sim H^*-\tilde A$ of $\xi:\tilde W\to V$ one has 
(\cite[Lem.~2.3]{Prokhorov-Zaidenberg-4-Fano})
\begin{equation}
\label{eq:relation}
\tilde B^4=c_2(V)\cdot S+K_V|_S\cdot K_S-c_2(S)-K_V^2\cdot S.
\end{equation}
Letting $s_0$ and $f$ be the exceptional section and a ruling of 
$S\cong\FF_1\to\PP^1$, respectively, one computes
\begin{equation}
\label{eq:numbers}
K_V^2\cdot S=(2H^2)\cdot S=12,\quad c_2(S)=4,\quad K_V|_S\cdot
K_S=-2(s_0+2f)(-2s_0-3f)=10.
\end{equation}
Plugging~\eqref{eq:numbers} in~\eqref{eq:relation} gives
\begin{equation}
\label{eq:numbers-1} B^4=S^2-6.
\end{equation}
From the exact sequence
\begin{equation*}
0\longrightarrow T_S\longrightarrow 
T_V|_S\longrightarrow\NNN_{V/S}\longrightarrow 0
\end{equation*}
one deduces a relation for the Chern classes
\begin{equation*}(1-K_St+c_2(S)t^2)(1+c_1(\NNN_{V/S})t+c_2(\NNN_{V/S})t^2)=
1-{K_V}|_St+c_2(V)|_St^2.
\end{equation*}
This yields a system
\begin{equation*}
\begin{aligned} c_1(\NNN_{V/S})&=K_S-K_V|_S,
\\[5pt]
c_2(V)\cdot S&=c_2(S)-K_S\cdot
c_1(\NNN_{V/S})+c_2(\NNN_{V/S})
\\
&
=c_2(S)-K_S^2+K_S\cdot K_V|_S+S^2.
\end{aligned}
\end{equation*}
Using~\eqref{eq:numbers} one obtains the relation $c_2(V)\cdot S=S^2+6$. It 
follows
by~\eqref{eq:numbers-1} that
$
\tilde B^4=S^2.
$
Finally, from~\cite[Lem.~5.5]{Prokhorov-Zaidenberg-2015} one gets
\begin{equation*}
S^2=\tilde B^4=(H^*-\tilde A)^4=(H^*)^4+6(H^*)^2\cdot\tilde A^2-4H^*\cdot\tilde
A^3+\tilde A^4=1.
\end{equation*}
\end{proof}

\begin{slemma}
\label{lem:zero-intersection}
For any two cubic scrolls $S_i\in\SSS_i(V)$, $i=1,2$, one has $S_1\cdot S_2=0$.
\end{slemma}

\begin{proof}
By Lemma~\xref{lem:square} any two disjoint cubic scrolls belong to different 
components of $\SSS(V)$. According to 
Proposition~\xref{prop:center-inv-scroll}\xref{prop:center-inv-scroll-b} and 
Lemma~\xref{lem:disjoint-cones}, there is a Fano-Mukai fourfold $V=V_{18}$ of 
genus 10, which contains two disjoint cubic cones. By Theorem~\xref{thm:Mukai}, 
any Fano-Mukai fourfold of genus 10 appears as a smooth hyperplane section of 
the adjoint variety $\Omega\subset\PP^{13}$. Since the family of such 
hyperplane 
sections of $\Omega$ is irreducible, the lemma follows.
\end{proof}

\begin{proof}[Proof of Proposition~\xref{lem:cohomology}]
Using diagram~\eqref{diagram-2} one derives isomorphisms
\begin{equation*}
H^4(V,\ZZ)\cong H^4(W,\ZZ)\cong\ZZ\oplus\ZZ,
\end{equation*} 
see~\cite[Lem.~3.4]{Prokhorov-Zaidenberg-2015}.
By Lemmas~\xref{lem:square} and~\xref{lem:zero-intersection} the classes 
$[S_1]$ and $[S_2]$ are independent and generate $H^4(V,\ZZ)$.
\end{proof}

Let us introduce the following notation.

\begin{notation}
\label{nota:sigma}
Given a cubic scroll $S\in\SSS(V)$, consider the following objects:
\begin{itemize}

\item
the variety $\Lambda(S)\subset\Sigma(V)$ of rulings of $S$ (clearly,
$\Lambda(S)\cong\PP^1$);

\item
the variety $\Sigma (S)\subset\Sigma(V)$ of lines in $V$
which meet $S$.
\end{itemize}
\end{notation}

\begin{slemma}
\label{lem:div-sigma-S}
$\Sigma (S)$ is a divisor on $\Sigma(V)$.
\end{slemma}

\begin{proof}
The statement is immediate from Lemma~\xref{lemma--SS}.
\end{proof}

\begin{slemma}
\label{claim:lines-in-Sigma}
Consider the
standard projections $\pr_i:\Sigma(V)\to\PP^2$, $i=1,2$ \textup(see 
Theorem~\xref{prop:KaRa}\xref{prop:KaRa-a}\textup). Then for
any $S\in\SSS(V)$ the variety $\Lambda(S)\subset\Sigma(V)$ is a fiber of one
of these projections, while $\Sigma(S)$ is the pull-back of a line in $\PP^2$
under the other one. 
In particular, $\Sigma(S)\cong\FF_1$, and $\Lambda(S)\subset\Sigma(S)$ is the 
exceptional section.
\end{slemma}

\begin{proof}
Let $S_1\in\SSS_1(V)$.
For a general cubic scroll $S_2\in\SSS_2(V)$,
by Lemma~\xref{lemma-nul-intersection-of-scrolls}
and Proposition~\xref{lem:cohomology} one has $S_1\cap S_2=\emptyset$.
Therefore, $\Sigma(S_1)\cdot\Lambda(S_2)=0$. Similarly, for a general 
$S_1'\in\SSS_1(V)$
one has $\Sigma(S_1)\cdot\Lambda(S_1')=1$. In the notation of 
Remark~\xref{rem:ext-action}\xref{rem:ext-action-1}, modulo reindexing 
$\SSS_i(V)$ one obtains $\Sigma(S_1)\sim\mathcal{F}_1$. Thus, $\Sigma(S_1)$ is 
the pull-back of a line
under a projection $\Sigma(V)\to\PP^2$. This shows as well that $\Lambda(S_1)$ 
is a fiber of the other projection.
\end{proof}

The following corollary of~\eqref{eq:relations-in-homologies} and
\xref{claim:lines-in-Sigma} is immediate, cf.~\cite[proof of 
Prop.~2]{KapustkaRanestad2013}.

\begin{scorollary}
\label{claim:common-ruling}
\begin{enumerate}
\renewcommand\labelenumii{\rm (\roman{enumii})}
\renewcommand\theenumii{\rm (\roman{enumii})}

\item
\label{claim:common-ruling-a}
Any line $l\in\Sigma(V)$ is a common ruling of exactly two cubic scrolls
$S_1(l)$ and $S_2(l)$, where $S_i(l)\in\SSS_i(V)$, $i=1,2$. In particular, $V$ 
is covered by the cubic scrolls contained in $V$.

\item
\label{claim:common-ruling-b}
If $l$ is contained in a third cubic scroll $S\neq S_i(l)$, $i=1,2$, then $S$ is
smooth, and $l$ is the exceptional line of $S$.

\item
\label{claim:common-ruling-d} 
Two distinct cubic scrolls $S'\neq S$ in $V$
can have at most one common ruling.
\end{enumerate}
\end{scorollary}

The possible intersections of two cubic cones are as follows.

\begin{scorollary}
\label{cor:intersection-of-cones}
Two distinct cubic cones $S$, $S'$ in $V$ either are disjoint, or meet 
transversally in one point, or finally contain a common ruling which is the only one-dimensional component of their intersection.
\end{scorollary}

\begin{proof}
By Proposition~\xref{lem:cohomology}, if the intersection $S\cap S'$ is finite, 
then either it is empty, or $S$ and $S'$ meet transversally at a single point, 
depending on whether $S, S'$ belong to different components of $\SSS(V)$ or not.
%

Assume that $\dim (S\cap S')=1$. 
Then any ruling of $S'$ meets $S$ and so it is contained in $A_S$.
Therefore, $S'\subset A_S$. Since through any point of $A_S\setminus S$ 
passes exactly one line in $A_S$ meeting $S$ (see Lemma~\xref{lemma--SS}), the vertex $v(S')$ lies on $S$.
By symmetry, $v(S)$ lies on $S'$. Thus, the line $l$ joining $v(S)$ and $v(S')$
is a common ruling. 

It follows that $v(S')\in\langle S\rangle$, and so, 
$\langle S\rangle \cap S'$ is a union of lines in $V$ passing through $v(S')$.
By symmetry, $\langle S'\rangle \cap S$ is a union of lines passing through $v(S)$. One concludes that 
$$S\cap S'=(\langle S'\rangle \cap S)\cap (\langle S\rangle \cap S')$$ 
is a union of lines in $V$ passing through both $v(S)$ and $v(S')$. Since $v(S)\neq v(S')$, see Proposition \ref{lem:a-d}(e), 
$l$ is the only one-dimensional component of $S\cap S'$.
\end{proof}

We can now reinterpret the embedding $\Sigma(V)\hookrightarrow\PP^2\times\PP^2$
as follows.

\begin{scorollary}
\label{claim:common-ruling-1}
In the notation of~\xref{claim:common-ruling}\xref{claim:common-ruling-a}, the
morphism
\begin{equation}
\label{eq:embedding}
\varrho:\Sigma(V)\longrightarrow\SSS_1(V)\times\SSS_2(V)\cong\PP^2\times\PP^2,
\quad
l\longmapsto (S_1(l),\,S_2(l)),
\end{equation}
realizes $\Sigma(V)$ as a smooth $(1,1)$-divisor in $\PP^2\times\PP^2$.
\end{scorollary}

We use below the following notation and terminology.

\begin{notation}
\label{nota:4cases}
Given a line $l$ in $V$, consider the following objects:
\begin{itemize}

\item
the family $\Sigma(l)\subset\Sigma(V)$ of
lines in $V$ which meet $l$;

\item
the union $\Theta(l)\subset V$ of lines from $\Sigma(l)$;

\item
the morphism $\pi:\Sigma(l)\setminus\{l\}\to l$,
$l'\longmapsto l'\cap l$.
\end{itemize}

It is known (\cite[proof of Prop.~2]{KapustkaRanestad2013}) that
$\Sigma(l)\setminus\{l\}$ is a curve in $\Sigma(V)$ of bidegree $(1,1)$. Due
to Proposition~\xref{lem:a-d}\xref{lem:a-d(d)} and Lemma
\xref{cor:common-exc-section}, $\pi$ is generically $2:1$ or $1:1$. So,
$\Sigma(l)$ consists of at most two components. Then also $\Theta(l)$ consists
of at most two components of pure dimension $2$.

A line $l\in\Sigma(V)$ will be called a \emph{splitting line}, or an
\emph{$s$-line} for short, if $\Theta(l)$ splits into two components
$\Theta_1(l)$ and $\Theta_2(l)$.
We let $\Sigma_{\s}(V)\subset\Sigma(V)$ denote the subvariety of $s$-lines in 
$V$. Clearly, $\Sigma_{\s}(V)$ is a closed subset in $\Sigma(V)$. We consider it 
with its reduced structure.

The proof of Proposition~2 in~\cite{KapustkaRanestad2013} shows that
$\Theta_1(l)$ and $\Theta_2(l)$ belong to different components of $\SSS(V)$.
We adopt the convention $\Theta_i(l)\in\SSS_i(V)$, $i=1,2$. Adopting a 
terminology from~\cite{KapustkaRanestad2013}, we call a line $l$ in a cubic 
scroll $S$ \emph{exceptional} if either $S$ is smooth and $l$ is the exceptional 
section of $S$, or $S$ is a cone and $l$ is a ruling of $S$.
\end{notation}

Since $(\Theta(l))_{l\in\Sigma(V)}$ is a flat family over an irreducible base, 
the function $l\longmapsto\deg\Theta(l)$
on $\Sigma(V)$ is constant. In statement~\xref{lem:degree-of-theta-a} of the 
following lemma we determine this constant.

\begin{lemma}
\label{lem:degree-of-theta}
\begin{enumerate}
\renewcommand\labelenumi{\rm (\alph{enumi})}
\renewcommand\theenumi{\rm (\alph{enumi})}

\item
\label{lem:degree-of-theta-a}
For any line $l\in\Sigma(V)$
one has $\deg\Theta(l)=6$. If $l$ is an $s$-line with
$\Theta(l)=\Theta_1(l)\cup\Theta_2(l)$, then $\Theta_i(l)$, $i=1,2$, are cubic
scrolls.

\item
\label{claim:s-lines-1}
Let $l\in\Sigma(V)$. Then $l\in\Sigma_{\s}(V)$, that is,
$l$ is an $s$-line,
if and only if $l$ is an exceptional line of a cubic scroll.

\item
\label{claim:s-lines-2}
An $s$-line $l$ is contained in $\mathcal{B}$ if and only if $l$ is a ruling of
a cubic cone.

\end{enumerate}
\end{lemma}

\begin{proof}
\xref{lem:degree-of-theta-a} Consider the rational surfaces 
$T_i=\overline{\varsigma(\mathcal U_i)}\subset\Sigma(V)$, $i=1,2$, where 
${\mathcal{U}_i}\subset\SSS_i(V)$ is the subset of smooth scrolls, and 
$\varsigma:\mathcal{U}_i\to\Sigma(V)$ sends a scroll $S\in{\mathcal{U}_i}$ to 
its 
exceptional section, cf.\ the proof of Proposition~\xref{lem:fixed-pt}. 
It is easy to see that any two exceptional divisors on $\Sigma(V)$ meet.
Hence, $T_1\cap T_2\neq\emptyset$.

Let $[l]\in T_1\cap T_2\subset\Sigma(V)$. We claim that $\Theta(l)=S_1\cup S_2$
for some $S_i\in\SSS_i(V)$, $i=1,2$. Indeed, one has $l=\lim_{j\to\infty} 
l_{i,j}$, where for any $j\ge 1$, $l_{i,j}\in\varsigma(\mathcal{U}_i)$ is the
exceptional line of a smooth cubic scroll $S_{i,j}\in\SSS_i$, $i=1,2$. Passing
to subsequences one may assume that $\lim_{j\to\infty} S_{i,j}=S_i\in\SSS_i$,
$i=1,2$. If $S_i$ is smooth, then $l\subset S_i$ is the exceptional line of
$S_i$. In any case, $S_1\cup S_2\subset\Theta(l)$. Since $\Theta(l)$ consists of
at most two components, see~\xref{nota:4cases}, one has $\Theta(l)=S_1\cup
S_2$, as claimed. So, $\deg\Theta(l)=6$ $\forall l\in\Sigma(V)$. Now the second
assertion follows as well. Indeed, $\Omega$ (and then also $V$) contains neither
a plane, nor a quadric surface (\cite[Lem.~3]{KapustkaRanestad2013}). Moreover,
since $V$ is an intersection of quadrics (\cite[Lem.~2.10]{Iskovskih1977a}), $V$
does not contain any cubic surface $F$ with $\langle F\rangle\cong\PP^3$.

\xref{claim:s-lines-1} 
To show the ``if'' part, observe that
in both cases the corresponding cubic scroll is a component of $\Theta(l)$. Due
to~\xref{lem:degree-of-theta-a}, $l$ is an $s$-line.
The ``only if'' part follows from the definition. 

\xref{claim:s-lines-2} 
A line $l$ on $V$ is contained in $\mathcal{B}$ if and
only if through a general point
of $l$ pass at most $2$ lines in $V$ including $l$ itself. For an $s$-line $l$
this is the case if and only if at least one of the components $\Theta_1(l)$ and
$\Theta_2(l)$ is a cubic cone.
\end{proof}

\begin{scorollary}
\label{cor:degree-of-theta}
One has $\dim\Sigma_{\s}(V)=2$.
\end{scorollary}

\begin{proof}
This follows from 
Lemma~\xref{lem:degree-of-theta}\xref{lem:degree-of-theta-a}--\xref{claim:s-lines-1}. 
Indeed, $\Sigma_{\s}(V)$ consists of the exceptional lines of 
smooth cubic scrolls and of the rulings of cubic cones. The first family of 
lines is purely two-dimensional, and the second is at most two-dimensional due 
to Lemma~\xref{lem:cubic-cones}.
\end{proof}

\begin{proposition}
\label{claim:ruling-in-scroll}
For any $S\in\SSS(V)$ one has
\begin{enumerate}
\renewcommand\labelenumi{\rm (\alph{enumi})}
\renewcommand\theenumi{\rm (\alph{enumi})}

\item
\label{claim:ruling-in-scroll-a}
$\Lambda(S)\subset\Sigma_{\s}(V)$ if and only if $S$ is a cubic cone;

\item
\label{claim:ruling-in-scroll-b}
$\Lambda(S)\cdot\Sigma_{\s}(V)=1$ for any $S\in\SSS(V)$;

\item
\label{claim:ruling-in-scroll-c}
$\Sigma_{\s}(V)\sim\mathcal{F}_1+\mathcal{F}_2$.
\end{enumerate}
\end{proposition}

\begin{proof}
\xref{claim:ruling-in-scroll-a} is immediate from 
Lemma~\xref{lem:degree-of-theta}\xref{lem:degree-of-theta-a}--\xref{claim:s-lines-1}.

\xref{claim:ruling-in-scroll-b} 
Let $S\in\SSS(V)$ be smooth. Let us show that
exactly one ruling of $S$ is an $s$-line. For such a ruling $l$ one has $\Theta(l)=S_1\cup S_2$, 
where $S_i\in\SSS_i(V)$, $S_i\neq S$,
$i=1,2$.
Due to Corollary
\xref{claim:common-ruling}\xref{claim:common-ruling-a}, at least one of $S_1$, 
$S_2$ is smooth. Since each ruling of $S_i$ meets $S$,
one has $S_i\subset A_S$, $i=1,2$.

Starting with the pair $(V,S)$ we construct diagram~\eqref{diagram-2}. The
proper transform $\tilde S_i$ of $S_i$ in $\tilde W$ is a ruled surface
contained in $\tilde A_S$, and its image $\theta(S_i)\subset W$ under the linear
projection
$\theta:\PP^{12}\dashrightarrow\PP^7$ with center $\langle S\rangle$
is a curve contained in $F$.

Assume first that $\langle S\rangle\cap\langle S_i\rangle=l$ for some 
$i\in\{1,2\}$. Then $J:=\theta(S_i)$ is a smooth
conic. Indeed, if $S_i$ is smooth, then $\theta(S_i)$ is the image of a conic
section of $S_i\to\PP^1$ disjoint with $l$, and if $S_i$ is a cubic cone, then
$\theta(S_i)$ is the image of a twisted cubic in $S_i$, which meets $l$ once and
transversely. The rational normal quintic scroll $F$, which contains a 
nondegenerate conic $J$, is isomorphic to $\FF_1$, and
$J$ is its exceptional section. Hence $S_i$ is the only
component of $\Theta(l)$ with $\langle S\rangle\cap\langle
S_i\rangle=l$. Moreover, $l$ is a unique splitting line in $S$, because $S_i$
does not contain two distinct rulings of $S$, see Corollary
\xref{claim:common-ruling}\xref{claim:common-ruling-d}. Thus, in this case
$\Lambda(S)\cap\Sigma_{\s}(V)$ consists of a single point.

Assume further that $\dim(\langle S\rangle\cap\langle S'_i\rangle)=2$ for
$i=1,2$. Then $f_i:=\theta(S_i)$ is a line in $F$ for $i=1,2$. These lines
cannot be both rulings of $F$. Indeed, otherwise
$\eta^{-1}(f_i)$, $i=1,2$, would be ruled surfaces in $\tilde W$ from the same 
irreducible
family of ruled surfaces. Hence also $S_1$ and $S_2$ would belong to the same 
family of surfaces
in $V$ over an irreducible base. However, $S_1$ and $S_2$ belong to different 
components of $\SSS(V)$
(\cite[proof of Prop.~2]{KapustkaRanestad2013}). Due to Proposition
\xref{lem:cohomology} we obtain a contradiction.

So, one of the $f_i$, say, $f_1$ should be the exceptional line of
$F\cong\FF_3$, while $f_2$ should be a ruling of $F$. Once again, $S_1$ is
uniquely determined and does not contain two distinct rulings of $S$. Hence
$\Lambda(S)\cap\Sigma_{\s}(V)$ is a single point in this case as well.

\xref{claim:ruling-in-scroll-c} 
This follows by~\xref{claim:ruling-in-scroll-b} and~\xref{claim:lines-in-Sigma}.
\end{proof}

\begin{scorollary}
\label{cor:disjoint-cones}
If two distinct cubic cones $S$ and $S'$ in $V$ have a common ruling $l$, then 
they belong to different components of $\SSS(V)$.
\end{scorollary}

\begin{proof}
Indeed, $\Lambda(S)$ and $\Lambda(S')$ are two lines in $\Sigma_{\s}(V)$, which 
meet in a point $[l]$. If $S$ and $S'$ were in the same component, say, 
$\SSS_i(V)$ of $\SSS(V)$, then these lines would project to the same point 
$\pr_j([l])\in\PP^2$, where $j\neq i$. The latter means that $S=S'$, a 
contradiction.
\end{proof}

\section{Cubic cones in $V_{18}$ and lines on singular del Pezzo
sextics}
\label{sec:del-Pezzo}

In this section we associate to each Fano-Mukai fourfold $V=V_{18}$ of
genus 10 a certain (singular) del Pezzo sextic surface. The cubic cones in $V$
occur to be in a one-to-one correspondence with
the lines on this sextic.

\begin{lemma}[{\cite[\S~8.1.1, \S~8.4.2]{Dolgachev-ClassicalAlgGeom}}, 
{\cite[Prop.~ 8.3]{Coray-Tsfasman-1988}}]
\label{claim:Du-Val-del-Pezzo}
Let $X\subset\PP^6$ be a linearly nondegenerate 
normal surface of degree $6$, and let $\tilde X$ be the 
minimal desingularization of $X$. Let 
$\Gamma(\tilde X)$ be the dual graph of the configuration of $(-1)$- and 
$(-2)$-curves in $\tilde X$. Suppose that $X$ admits two different birational
contractions onto $\PP^2$. Then the following hold.
\begin{enumerate}
\renewcommand\labelenumi{\rm (\alph{enumi})}
\renewcommand\theenumi{\rm (\alph{enumi})}
\item 
$X$ has at worst Du Val singularities.
\item 
$X$ admits a canonical embedding in $\PP^2\times\PP^2$. 
\item\label{claim:Du-Val-del-Pezzo-c}
$X$ is one of the following:
\begin{enumerate}
\renewcommand\labelenumii{\rm (\roman{enumii})}
\renewcommand\theenumii{\rm (\roman{enumii})}

\item
\label{claim:Du-Val-del-Pezzo-i} 
a smooth del Pezzo sextic surface, with $\Gamma(\tilde X)$ being a cycle of six 
$(-1)$-vertices;

\item
\label{claim:Du-Val-del-Pezzo-ii} a del Pezzo sextic surface with a unique 
singular
point of type $A_1$ and with
\begin{equation}
\label{eq:A1}
\Gamma(\tilde X):
\qquad
\vcenter{
\xymatrix
{
\overset{-1}\bullet\ar@{-}[r]&\overset{-1}\bullet\ar@{-}[r]&
\overset{-2}\circ\ar@{-}[r]&
\overset{-1}\bullet\ar@{-}[r]&\overset{-1}\bullet
}}
\end{equation}
where the black vertices correspond to the lines in $X$;

\item
\label{claim:Du-Val-del-Pezzo-iii} a del Pezzo sextic surface with a unique 
singular
point of type $A_2$ and with
\begin{equation}
\label{eq:A2}
\Gamma(\tilde X):\quad\vcenter{
\xymatrix@R=6pt@C=50pt{
&\overset{-2}{\circ}\ar@{-}[d]&
\\
\underset{-1}\bullet\ar@{-}[r]&\underset{-2}\circ\ar@{-}[r]&\underset{-1}\bullet
}}
\end{equation}
\end{enumerate}
\item\label{claim:Du-Val-del-Pezzo-d}
The number of lines in $X$ is $6$ for type~\xref{claim:Du-Val-del-Pezzo-i}, $4$ 
for 
type~\xref{claim:Du-Val-del-Pezzo-ii}, and $2$ for type 
\xref{claim:Du-Val-del-Pezzo-iii}. 
There exists a unique divisor $\mathfrak{G}_X\in |-K_X|$ whose support coincides 
with the union of lines
on $X$.
\end{enumerate}
\end{lemma}

Abusing the language we say that a del Pezzo 
sextic $X$ in~\xref{claim:Du-Val-del-Pezzo-ii} (resp., $X$ 
in~\xref{claim:Du-Val-del-Pezzo-iii}) is \emph{of type $A_1$} (resp., \emph{of 
type $A_2$}).

One has the following result.

\begin{proposition}
\label{prop:dP-6}
$\Sigma_{\s}(V)$ is a sextic surface. If it is reducible, then this is the union 
$\mathcal{F}_1\cup\mathcal{F}_2$ of two cubic scrolls, where $\mathcal{F}_i$ is 
the pullback of a line in $\PP^2$
under the projection $\pr_i:\Sigma(V)\to\PP^2$, $i=1,2$. If it is irreducible, 
then $\Sigma_{\s}(V)$ is
a del Pezzo sextic which is either smooth, or of type $A_1$, or of type $A_2$.
\end{proposition}

\begin{proof}
By Proposition~\xref{claim:ruling-in-scroll}\xref{claim:ruling-in-scroll-c}\ 
$\Sigma_{\s}(V)$ is
a sextic hyperplane section of 
$\Sigma(V)$ with respect to the polarization on $\Sigma(V)$ induced by the 
Segre 
embedding $\PP^2\times\PP^2\hookrightarrow\PP^{8}$.
Moreover, if the surface $\Sigma_{\s}(V)$ is reducible, then 
$\Sigma_{\s}(V)=\mathcal{F}_1\cup\mathcal{F}_2$. Assume that
$\Sigma_{\s}(V)$ is irreducible. Again by 
Proposition~\xref{claim:ruling-in-scroll}\xref{claim:ruling-in-scroll-c}
the restrictions $p_i|_{\Sigma_{\s}(V)}: 
\Sigma_{\s}(V)\to\SSS_i(V)=\PP^2$, $i=1,2$, yield two birational contractions, 
whose exceptional divisors do not possess any common component. Moreover, any 
ruling of $p_i$ passing 
through a singular point of $\Sigma_{\s}(V)$ is contained in 
$\Sigma_{\s}(V)$. Therefore $\Sigma_{\s}(V)$ is normal.
The rest of the proof is straightforward from 
Lemma~\xref{claim:Du-Val-del-Pezzo}\xref{claim:Du-Val-del-Pezzo-c}.
\end{proof}

\begin{scorollary}
\label{cor:family-of-rulings} 
The set of cubic cones in $V$ is finite if and only if $\Sigma_{\s}(V)$ is 
irreducible.
If it is finite, then any cubic cone $S\subset V$ is invariant under the 
$\Aut^0(V)$-action.
\end{scorollary}

\begin{notation}
\label{nota:two-kinds}
Up to an automorphism of $(\PP^2)^\vee\times\PP^2$, which does not interchange
the factors, one may assume that $\Sigma(V)$ coincides with the variety $\Sigma$
of complete flags in $\PP^2$ given in $(\PP^2)^\vee\times\PP^2$ by equation
\eqref{eq:equation-of-Sigma}, cf.~\xref{rem:ext-action}. Up to an automorphism
of $\Sigma(V)$ one may suppose that $(\mathcal{F}_1,\mathcal{F}_2)$ is one of
the following:
\begin{enumerate}
\item\label{(i_1)}
$\mathcal{F}_1=\{x_0=0\}$ and $\mathcal{F}_2=\{y_0=0\}$;

\item\label{(i_2)}
$\mathcal{F}_1=\{x_0=0\}$ and $\mathcal{F}_2=\{y_1=0\}$.
\end{enumerate}
We say that a pair $(\mathcal{F}_1,\mathcal{F}_2)$ is \emph{of the first}
(resp., \emph{second}) \emph{kind} if it is equivalent to a pair~\xref{(i_1)} 
(resp.,
\xref{(i_2)}). For a pair $(\mathcal{F}_1,\mathcal{F}_2)$ of the first kind in
\xref{(i_1)}, the intersection $C=\mathcal{F}_1\cap\mathcal{F}_2$ is the
$(1,1)$-conic in $(\PP^1)^\vee\times\PP^1$ given by equation $x_1y_1+x_2y_2=0$.
The projection $p_2|_{\mathcal{F}_1}:\mathcal{F}_1\to\SSS_2(V)\cong\PP^2$
sends birationally $\mathcal{F}_1$ onto $\PP^2$ contracting the exceptional
section $s_1$ of $\mathcal{F}_1\cong\FF_1$ to the point $q_2=(1:0:0)$, and sends
the conic section $C\subset\mathcal{F}_1$ (disjoint with $s_1$) to the line
$h_2=\{y_0=0\}$, and symmetrically for the projection
$p_1|_{\mathcal{F}_2}:\mathcal{F}_2\to\SSS_1(V)\cong\PP^2$.

For a pair $(\mathcal{F}_1,\mathcal{F}_2)$ of the second kind the intersection
$C=\mathcal{F}_1\cap\mathcal{F}_2$ is reducible and consists of two lines
$s_1\cup s_2$,
where $s_1$ (resp. $s_2$) is the exceptional section of $\mathcal{F}_1$ (resp.
$\mathcal{F}_2$) and a ruling of $\mathcal{F}_2$ (resp. $\mathcal{F}_1$). Thus,
$p_2(s_1)=q_2\in h_2=\{y_1=0\}$, and $p_1(s_2)=q_1\in h_1=\{x_0=0\}$. The lines
$s_1$ and $s_2$ meet at a single point.
\end{notation}

\begin{slemma}
\label{lem:first-kind}
Assume that
$\Sigma_{\s}(V)=\mathcal{F}_1\cup\mathcal{F}_2$. Let $(W,F)$ be the pair linked 
to $(V,S_2)$.
Suppose also that the group $\Aut^0(W,F)$ contains a singular torus $\z(\Levi)$ 
of $\Aut(W)$, where $\Levi\subset\Aut(W)$ is a Levi subgroup.
Then $(\mathcal{F}_1,\mathcal{F}_2)$ is a pair of the first kind. This holds, in 
particular, if $\Aut(V)\supset (\Gm)^2$.
\end{slemma}

\begin{proof}
Suppose to the contrary that $(\mathcal{F}_1,\mathcal{F}_2)$ is a pair of the 
second kind. By Corollary~\xref{cor:finiteness}\xref{cor:finiteness-b} the cubic 
cones $S_1$ and $S_2$ are $\Aut^0(V)$-invariant. Since 
$\z(\Levi)\subset\Aut^0(W,F)$, then $\SSS_1(V)$ contains a cubic cone $S_1'$ 
disjoint with $S_2$, see Lemma~\xref{lem:disjoint-cones}. The cones $S_1$ and 
$S_2$ have a common ruling, hence $S_1'\neq S_1$. So, $S_1'$ corresponds to a 
rulings $\Lambda(S_{1,t})$ of $\mathcal{F}_1$ for some $t\in\PP^1$, see 
Corollary~\xref{cor:finiteness}\xref{cor:finiteness-b}. However, the cones $S_2$ 
and $S_1'=S_{1,t}$ as well have a common ruling. This gives a desired 
contradiction.
\end{proof}

From Proposition~\xref{prop:dP-6} we deduce

\begin{scorollary}
\label{cor:finiteness}
\begin{enumerate}
\renewcommand\labelenumi{\rm (\alph{enumi})}
\renewcommand\theenumi{\rm (\alph{enumi})}

\item
\label{cor:finiteness-a}
If $\Sigma_{\s}(V)$ is irreducible, then every component $\SSS_i(V)$, $i=1,2$, 
contains exactly $3$, $2$ or $1$
cubic cone\textup(s\textup) according to types~\xref{claim:Du-Val-del-Pezzo-i},
\xref{claim:Du-Val-del-Pezzo-ii}, and~\xref{claim:Du-Val-del-Pezzo-iii}
in Lemma~\xref{claim:Du-Val-del-Pezzo}, respectively.

\item
\label{cor:finiteness-b}
If $\Sigma_{\s}(V)=\mathcal{F}_1\cup\mathcal{F}_2$, then the following hold.
\begin{itemize}

\item
For $i=1,2$ the subvariety $\mathcal{C}_i(V)$ of $\SSS_i(V)\cong\PP^2$ whose 
points correspond to cubic cones in $V$ consists of a line $h_{i}\subset\PP^2$ 
and a reduced
$\Aut^0(V)$-invariant point $q_i=[S_i]$. For a pair 
$(\mathcal{F}_1,\mathcal{F}_2)$ of the first kind, $q_i\notin h_i$, otherwise 
$q_i\in h_{i}$.

\item
Each ruling $l_{i,t}$ of $S_i$, $t\in\PP^1$, is a common ruling of $S_i$ and of 
a unique cubic cone $S_{j,t}\in\mathcal{C}_{j}(V)$, $j\neq i$.

\item
If $(\mathcal{F}_1,\mathcal{F}_2)$ is a pair of the first kind, then
$S_j$ does not contain the vertex of $S_i$, $i\neq j$.
\end{itemize}
\end{enumerate}
\end{scorollary}

\begin{proof}
Statement~\xref{cor:finiteness-a} is straightforward from 
Proposition~\xref{claim:ruling-in-scroll}\xref{claim:ruling-in-scroll-a}.

\xref{claim:ruling-in-scroll-b}
By Proposition~\xref{claim:ruling-in-scroll}\xref{claim:ruling-in-scroll-a}
the cubic cones in $\SSS_i(V)$ correspond to the lines in
$\Sigma_{\s}(V)=\mathcal{F}_1\cup\mathcal{F}_2$ contracted under the projection
$p_i:\Sigma_{\s}(V)\to\SSS_i(V)=\PP^2$, $i,j=1,2$. Hence they correspond to the 
rulings of $\mathcal{F}_i$ and the exceptional sections $s_i$ of 
$\mathcal{F}_i$, $i=1,2$. This gives an isomorphism $\mathcal{C}_{i}(V)\cong 
h_i\cup\{q_i\}$ and the equalities $[S_i]=\{q_i\}$ and 
$s_j=\Lambda(S_i)=p_i^{-1}(q_i)$, $i,j=1,2$, $j\neq i$. Any ruling of 
$p_i:\mathcal{F}_i\to h_i$ meets the exceptional section $s_i$ in a point, which 
corresponds to the unique common ruling of the assigned cubic cones, cf.\ 
Corollary~\xref{claim:common-ruling}. This proves the first two assertions 
of~\xref{claim:ruling-in-scroll-b}.

To show the last one, we let $v_i$ be the vertex of the cone $S_i$, $i=1,2$. 
Suppose to the contrary that $v_i\in S_j$ for some choice of $i,j\in\{1,2\}$, 
$i\neq j$. Let $l$ be the unique ruling of $S_j$ passing through $v_i\neq v_j$. 
Since any line in $V$ through $v_i$ is a ruling of $S_i$, see 
Proposition~\xref{lem:a-d}\xref{lem:a-d(d)}, $l$ is a common ruling of $S_i$ and 
$S_j$. However, for a pair $(\mathcal{F}_1,\mathcal{F}_2)$ of the first kind 
both $\Lambda(S_1)\subset\mathcal{F}_2$ and $\Lambda(S_2)\subset\mathcal{F}_1$ 
are disjoint with the common conic section $C=\mathcal{F}_1\cap\mathcal{F}_2$. 
Hence, $S_1$ and $S_2$ do not have any ruling in common, a contradiction.
\end{proof}

\section{Automorphism groups of singular del Pezzo sextics}
\label{sec:1.3-1.4}
In this section we prove the following proposition.

\begin{proposition}
\label{thm:toric-case}
For any Fano-Mukai fourfold $V=V_{18}$ of genus $10$ one of the following 
cases~\xref{thm:toric-case-i}--\xref{thm:toric-case-iv} occurs.
\par\medskip\noindent
\setlength{\extrarowheight}{1pt}
\newcommand{\heading}[1]{\multicolumn{1}{c|}{#1}}
\newcommand{\headingl}[1]{\multicolumn{1}{c}{#1}}
{\rm
\begin{tabularx}{\textwidth}{p{0.02\textwidth}|p{0.37\textwidth}|p{
0.53\textwidth}}
&\heading{$\Sigma_{\s}(V)$}&\headingl{$\Aut(V)$}
\\\hline
\mbox{\nr\label{thm:toric-case-i}} &a union of two smooth cubic scrolls
meeting along a smooth conic&$\GL_2(\CC)\subset\Aut(V)\subset\GL_2(\CC)\rtimes 
(\ZZ/2\ZZ)$
\\
\mbox{\nr\label{thm:toric-case-ii}}
& a del Pezzo sextic of type $A_1$
& $\Ga\times\Gm\subset\Aut(V)\subset (\Ga\times\Gm)\rtimes (\ZZ/2\ZZ)$
\\\mbox{\nr\label{thm:toric-case-iii}}
& a smooth del Pezzo sextic
& 
\mbox{$(\Gm)^2\subset\Aut(V)\subset (\Gm)^2\rtimes (\ZZ/6\ZZ)$}
\\
\mbox{\nr\label{thm:toric-case-iv}}
& a union of two smooth cubic scrolls
meeting along a pair of intersecting lines
& $\Ga\times\Gm\subset\Aut(V)\subset B(\PGL_3(\CC))\rtimes 
(\ZZ/2\ZZ)$,\linebreak
where $B(\PGL_3(\CC))$ is a Borel subgroup of $\PGL_3(\CC)$\\\hline
\end{tabularx}
}\par\vspace{10pt}\noindent
In Cases~\xref{thm:toric-case-i} and~\xref{thm:toric-case-iv} the variety
$V$ contains two one-parameter families of cubic cones and two
$\Aut^0(V)$-invariant cubic cones $S_i\in\SSS_i(V)$, $i=1,2$. The number of
cubic cones in $V$ equals $4$ in Case~\xref{thm:toric-case-ii} and $6$ in Case
\xref{thm:toric-case-iii}; all of these cones are $\Aut^0(V)$-invariant.
\end{proposition}

In fact, Case~\xref{thm:toric-case-iv} does not occur, see 
Corollary~\xref{cor:stable-cones} in the next section.
The proof of Proposition~\xref{thm:toric-case} is done 
in~\xref{proof-of-thm:toric-case}. It is preceded by some preliminary facts and 
constructions.

\begin{lemma}
\label{lem:group-embedding}
\begin{enumerate}
\renewcommand\labelenumi{\rm (\alph{enumi})}
\renewcommand\theenumi{\rm (\alph{enumi})}

\item
\label{lem:group-embedding-a}
The induced $\Aut(V)$-action on $\Sigma(V)$ is effective and leaves invariant 
the divisor $\Sigma_{\s}(V)\subset\Sigma(V)$.

\item
\label{lem:group-embedding-b}
$\Aut^0(V)$ acts effectively on any component of $\Sigma_{\s}(V)$.

\item
\label{lem:group-embedding-c}
The action of $\Aut(V)$ on $\Sigma_{\s}(V)$ is effective.
\end{enumerate}
\end{lemma}

\begin{proof}
\xref{lem:group-embedding-a} By Proposition~\xref{lem:a-d}\xref{lem:a-d(b)}, 
there are exactly three lines passing through a general point of $V$. This 
allows to reconstruct the $\Aut(V)$-action on $V$ from the induced 
$\Aut(V)$-action on $\Sigma(V)$. So, the latter action is effective. Now
the assertion is straightforward.

\xref{lem:group-embedding-b} One may identify 
$\Sigma(V)\subset(\PP^2)^\vee\times\PP^2$ with the variety $\Sigma$ of complete 
flags in $\PP^2$, see~\xref{nota:two-kinds}. The duality permutes the factors 
$(\PP^2)^\vee$ and $\PP^2$ inducing an involution $\iota$ of $\Sigma$. Thus, 
$\Aut(\Sigma)=\Aut^0(\Sigma)\rtimes (\ZZ/2\ZZ)$, where 
$\Aut^0(\Sigma)\cong\Aut(\PP^2)$.

For $i=1,2$ the action of $\Aut^0(\Sigma(V))$ on the ruling $p_i: 
\Sigma(V)\to\SSS_i(V)$ induces an isomorphism 
$\Aut^0(\Sigma(V))\cong\Aut(\SSS_i(V))$ and also an injection 
$\Aut^0(V)\hookrightarrow\Aut(\SSS_i(V))$. For any irreducible component 
$\mathcal{T}$ of the divisor $\Sigma_{\s}(V)$, at least one of the projections 
$p_i|_{\mathcal{T}}:\mathcal{T}\to\SSS_i(V)$, $i=1,2$, is dominant. Hence 
the representation $\Aut^0(V)\to\Aut(\mathcal{T})$ is faithful.

\xref{lem:group-embedding-c}
Assume that $\alpha\in\Aut(V)$ induces the identity on $\Sigma_{\s}(V)$. In
particular, any ruling of $p_i:\Sigma(V)\to\SSS_i(V)$ contained in
$\Sigma_{\s}(V)$ is invariant under $\alpha$. The rulings on $\Sigma_{\s}(V)$
correspond to the cubic cones in $V$, and they do exist, see, e.g., Corollary
\xref{cor:finiteness}. It follows that $\alpha(\SSS_i(V))=\SSS_i(V)$, $i=1,2$,
that is, $\alpha$ does not interchange the factors of $(\PP^2)^\vee\times\PP^2$.
Since the projections $p_i|_{\Sigma_{\s}(V)}:\Sigma_{\s}(V)\to\SSS_i(V)$,
$i=1,2$, are dominant and $\alpha$-equivariant, $\alpha$ acts identically on the
factors of $(\PP^2)^\vee\times\PP^2$, hence also on $\Sigma$. By
\xref{lem:group-embedding-a}, $\alpha=\id_V$.
\end{proof}

The following corollary is straightforward.

\begin{scorollary}
\label{cor:embedding-Aut-V-to-Aut-X}
There is an embedding 
\begin{equation*}
\Aut(V)\hookrightarrow\Aut((\PP^2)^\vee\times\PP^2,\,\Sigma(V),\Sigma_{\s}
(V))=\Aut(\Sigma(V),\Sigma_{\s}(V)).
\end{equation*}
\end{scorollary}

\begin{proof}
Indeed, the pair of projections of $(\PP^2)^\vee\times\PP^2$ to the factors 
being canonical (see~\xref{rem:ext-action}), the inclusion 
$\Aut((\PP^2)^\vee\times\PP^2,\,\Sigma(V))\subset\Aut(\Sigma(V))$ is in fact the 
equality.
\end{proof}

The Hilbert scheme of lines $\Sigma(V)$ admits the following alternative 
description.

\begin{sconstruction}
\label{constr:tensor-product} 
\setenumerate[0]{leftmargin=8pt,itemindent=7pt}
\begin{enumerate}
\renewcommand\labelenumi{\bf\arabic{enumi}.}
\renewcommand\theenumi{\bf\arabic{enumi}}
\item
For a vector $\y\neq 0$ in a vector space $U=\CC^{n+1}$, we let $[\y]$ denote 
the image of $\y$ in $\PP(U)=\PP^n$. Consider the Segre embedding
\begin{equation*}
\nu: (\PP^n)^\vee\times\PP^n\hookrightarrow\PP^{n^2-1},\quad 
([\x^*],[\y])\longmapsto [\x^*\otimes\y],\quad\x^*\in (U)^\vee,\,\,\y\in 
U,\,\,\x^*,\y\neq 0.
\end{equation*}
Using the standard isomorphism $U^\vee\otimes 
U\cong\operatorname{End}(U)=\gl(U)$, we regard
$U^\vee\otimes U$ as the vector space of square matrices of order $n+1$ over 
$\CC$. We let $(e_0,\ldots,e_n)$ be the standard basis of $U=\CC^{n+1}$ and 
$(e^*_0,\ldots,e^*_n)$ be its dual, so that $e_i^*\otimes e_j=E_{i,j}$ is the 
elementary matrix with the only nonzero entry $e_{i,j}=1$. Under this 
identification, the image of $U^\vee\times U$ in $\gl(U)$ consists of matrices 
of rank 1:
\begin{equation}
\label{eq:matrix}
\x^*\otimes\y=
\begin{pmatrix}
x_0y_0 & \ldots & x_0y_n \\
\vdots & \vdots & \vdots \\
x_ny_0 & \ldots & x_ny_n
\end{pmatrix},
\end{equation}
where $\x^*=x_0e_0^*+\ldots+x_ne_n^*$ and $\y=y_0e_0+\ldots+y_ne_n$.

In the sequel we let $n=2$, so that $\gl(U)=\gl_3(\CC)$.
Assuming that $\Sigma(V)$ is realized as the subvariety of complete flags in 
$\PP^2$ embedded in $(\PP^2)^\vee\times\PP^2$ with 
equation~\eqref{eq:equation-of-Sigma}
\begin{equation*}
x_0y_0+x_1y_1+x_2y_2=0,
\end{equation*}
its image $\Sigma$ in $\PP^8=\PP(\gl(U))$ lies in the hyperplane 
$\sll(U)=\sll_3(\CC)$ of matrices with zero trace. We identify $\Sigma(V)$ with 
$\Sigma$. In this way, $\Sigma(V)$ is realized as a smooth hyperplane section of
$\nu((\PP^2)^\vee\times\PP^2)\subset\PP^8$, where $\nu$ stands for the Segre 
embedding, and, alternatively, as the projectivization of the cone of $(3\times 
3)$-matrices of rank $1$ with zero trace.

For a square matrix $M$ of order $3$ and of rank $1$ one has:
\begin{equation*}
\tr(M)=0\Longleftrightarrow M^2=0\Longleftrightarrow\operatorname{im} 
(M)\subset\ker(M).
\end{equation*}
Thus, for $[M]\in\Sigma(V)$, $\operatorname{im} (M)$ is a line in the plane 
$\ker(M)$, and
\begin{equation*}
\PP(\operatorname{im} (M))\subset\PP(\ker(M))\subset\PP^2
\end{equation*}
is a complete flag in $\PP^2$. The maps $M\longmapsto\ker(M)$ and 
$M\longmapsto\operatorname{im}(M)$ yield the projections $\Sigma(V)\to 
(\PP^2)^\vee$ and $\Sigma(V)\to\PP^2$, respectively.

\item
The $\GL_3(\CC)$-action on $\gl_3(\CC)$ by conjugation: $(A,M)\longmapsto 
AMA^{-1}$ descends to a $\PGL_3(\CC)$-action on $\PP^8=\PP(\gl_3(\CC))$. The 
Segre embedding is equivariant with respect to the latter $\PGL_3(\CC)$-action 
on $\PP^8$ and the $\PGL_3(\CC)$-action on $(\PP^2)^\vee\times\PP^2$ given by
\begin{equation*}
(A, ([\x^*],[\y]))\longmapsto ([{(A^{-1})}^t\x^*],\,[A\y]).
\end{equation*}
Any square matrix $M$ of order $3$ and of rank $1$ with zero trace has Jordan 
form 
$\left(\begin{smallmatrix}
0 & 1 & 0\\
0 & 0 & 0\\
0 & 0 & 0
\end{smallmatrix}\right)$.
Therefore, the induced $\PGL_3(\CC)$-action on $\Sigma$ is transitive. In fact, 
$\Sigma\cong\PGL_3(\CC)/B_0$, where $B_0\subset\PGL_3(\CC)$ is the Borel 
subgroup of upper triangular matrices.

Recall that $\Aut(\Sigma)\cong\PGL_3(\CC)\rtimes (\ZZ/2\ZZ)$, see the proof of 
Lemma~\xref{lem:group-embedding}.
One can take the matrix transposition for the generator of $\ZZ/2\ZZ$ 
interchanging the factors of $(\PP^2)^\vee\times\PP^2$.
The (maximal) diagonal torus $\TT_0\cong (\Gm)^2$ of $B_0$ acts effectively on 
$\Sigma$. Any $(\Gm)^2$-subgroup of $\Aut(\Sigma)$ is conjugate to $\TT_0$ in 
$\PGL_3(\CC)$.

\item
To a square matrix $C\neq 0$ of order $3$ one can associate a hyperplane $H_C$ 
in the vector space $\gl_3(\CC)$, where
\begin{equation*}
H_C=\{M\in\gl_3(\CC)\mid\tr (M\cdot C)=0\}.
\end{equation*}
If $C\neq 0$ is a scalar matrix, then $H_C$ coincides with the subspace
$\sll_3(\CC)\subset\gl_3(\CC)$ of matrices with zero trace. Since the bilinear
form $(A,B)\longmapsto\tr (A\cdot B)$ on $\sll_3(\CC)$ is
nondegenerate, any hyperplane in $\sll_3(\CC)$ coincides with
$H_C\cap\sll_3(\CC)$ for a suitable $C\in\sll_3(\CC)$.

\item
According to 
Proposition~\xref{claim:ruling-in-scroll}\xref{claim:ruling-in-scroll-c}, 
$\Sigma_{\s}(V)$ is
a hyperplane section of $\Sigma(V)=\Sigma$ in $\PP^8$. Therefore, there exists a
$(3\times 3)$-matrix $C=C(V)$ with zero trace such that
\begin{equation*}
\Sigma_{\s}(V)=\Sigma\cap\PP(H_{C})=\nu\left((\PP^2)^\vee\times\PP^2\right)\cap
\PP\left(\sll_3(\CC)\cap H_{C}\right).
\end{equation*}
Such a matrix $C$ is defined uniquely
up to a nonzero scalar factor. The $\PGL_3(\CC)$-action on
$\PP^8=\PP(\gl_3(\CC))$ by conjugation leaves the pair
$(\nu((\PP^2)^\vee\times\PP^2),\,\Sigma)$ invariant. Replacing $C$ (defined up
to a scalar factor) by its Jordan form, any hyperplane section $X$ of $\Sigma$
can be sent by an automorphism of $\Sigma$ to one of the sections
\begin{equation}
\label{eq:pencil}
X_{(a:b)}=\Sigma\cap H_{C_{(a,b)}},\,\,(a,b)\neq (0,0),\quad\mbox{and}\quad 
X_i=\Sigma\cap H_{C_i},\,\,i=1,2,3,
\end{equation}
where
\begin{equation}
\label{eq:two-matrices}
\begin{array}{lll}
&C_{(a,b)}=\diag(a,b,-a-b),
\\[5pt]
C_1=
\begin{pmatrix}1&
1&
0\\
0&
1&
0\\
0&
0&
{-2}\\
\end{pmatrix},
&
C_2=
\begin{pmatrix}0&
1&
0\\
0&
0&
1\\
0&
0&
0\\
\end{pmatrix},
&
C_3=
\begin{pmatrix}0&
1&
0\\
0&
0&
0\\
0&
0&
0\\
\end{pmatrix}.
\end{array}
\end{equation}

\item
Consider the group $\tilde{\GL_n(\CC)}=\GL_n(\CC)\rtimes (\ZZ/2\ZZ)$, where
the generator $\tau$ of $\ZZ/2\ZZ$ acts on $\GL_n(\CC)$ via the Cartan involution 
$A\longmapsto(A^t)^{-1}$. The action of $\GL_n(\CC)$ via conjugation on $\gl(U)$ with
$U=\CC^n$ extends to an action of $\tilde{\GL_n(\CC)}$, where $\tau$ acts on
$\gl(U)$ via $C\longmapsto C^t$. Let also
\begin{eqnarray*}
\operatorname{PCent}_{\tilde{\GL_n(\CC)}}(C)&=&\left\{\tilde 
A\in\tilde{\GL_n(\CC)}\,\left|\,
\tilde A . C=\alpha C\quad\mbox{for some}\quad\alpha\in\CC^*\right.\right\},
\\
\tilde{\PGL_n(\CC)}&=&\tilde{\GL_n(\CC)}/\z(\tilde{\GL_n(\CC)}
)\cong\PGL_n(\CC)\rtimes (\ZZ/2\ZZ),
\\
\operatorname{PCent}_{\tilde{\PGL_n(\CC)}}(C)&=&\operatorname{PCent}_{\tilde{
\GL_n(\CC)}}(C)/\z(\tilde{\GL_n(\CC)}).
\end{eqnarray*}
\end{enumerate}
\end{sconstruction}

\begin{lemma}
\label{lem:pcentralizer}
For a hyperplane section $X_C=\Sigma\cap H_C$, where $C\in\sll_3(\CC)$, $C\neq 
0$, one has
\begin{equation*}
\Aut(\Sigma, 
X_C)=\operatorname{PCent}_{\tilde{\PGL_3(\CC)}}(C)\subset\tilde{\PGL_3(\CC)}.
\end{equation*}
\end{lemma}

\begin{proof}
We have $\Aut(\Sigma)=\tilde{\PGL_3(\CC)}$.
Clearly, $g\in\tilde{\GL_3(\CC)}$ induces an automorphism of $(\Sigma,X_C)$ if 
and only if $g(H_{C})$ is a member of the pencil generated by $H_C$ and $H_{0}$,
if and only if $g . C=\alpha C+\beta\operatorname{E}$ for some 
$\alpha,\beta\in\CC$. For $A\in\GL_3(\CC)$ one has $A . C=ACA^{-1}$ and $A\tau . 
C=AC^tA^{-1}$. Anyway, under the latter condition, $\tr(\alpha 
C+\beta\operatorname{E})=\tr(C)=0$, hence $\beta=0$. Now the claim follows.
\end{proof}

\begin{lemma}
\label{lem:centralizers}
Given a matrix $C$ as in~\eqref{eq:two-matrices}, the following hold.
\par\medskip\noindent
{\rm
\setlength{\extrarowheight}{1pt}
\newcommand{\heading}[1]{\multicolumn{1}{c|}{#1}}
\newcommand{\headingl}[1]{\multicolumn{1}{c}{#1}}
\begin{tabularx}{0.97\textwidth}{p{0.03\textwidth}|p{0.36\textwidth}|p{
0.2\textwidth}|p{0.35\textwidth}}
&\heading{$C$} &\heading{$\Sing(X_C)$} &\headingl{$\Aut(\Sigma, X_C)$}
\\\hline
\mbox{\nnr\label{lem:centralizers-i}} &
$C_{(1,1)}$, $C_{(1,-2)}$, $C_{(-2,1)}$ & a smooth conic
&
$\tilde{\GL_2(\CC)}=\GL_2(\CC)\rtimes (\ZZ/2\ZZ)$
\\
\mbox{\nnr\label{lem:centralizers-ii}} &
$C_{(1,\zeta)}$,\quad $\zeta\neq 1$, $\zeta^3=1$& $\emptyset$&
$(\Gm)^2\rtimes (\ZZ/6\ZZ)$
\\
\mbox{\nnr\label{lem:centralizers-iii}} &
$C_{(1,b)}$,\quad $b\notin\{-2,\, -1/2\}$, $b^3\neq 1$& $\emptyset$&
$(\Gm)^2\rtimes (\ZZ/2\ZZ)$
\\
\mbox{\nnr\label{lem:centralizers-iv}} &
$C_1$& $A_1$& $(\Ga\times\Gm)\rtimes (\ZZ/2\ZZ)$
\\
\mbox{\nnr\label{lem:centralizers-v}} & $C_2$& $A_2$ &
$((\Ga)^2\rtimes\Gm)\rtimes (\ZZ/2\ZZ)$
\\
\mbox{\nnr\label{lem:centralizers-vi}} &
$C_3$&two intersecting lines &\mbox{$B(\PGL_3(\CC))\rtimes (\ZZ/2\ZZ)$}\\\hline
\end{tabularx}
}
\par\vspace{10pt}\noindent
where
$B(\PGL_3(\CC))$ stands for a Borel subgroup of $\PGL_3(\CC)$,
in Cases~\xref{lem:centralizers-iii} and~\xref{lem:centralizers-iv}
the generator of $\ZZ/2\ZZ$ acts by the inversion $g\mapsto g^{-1}$ on the first 
factor,
and in Case~\xref{lem:centralizers-v} the group $\Aut(\Sigma, X_C)$
does not contain the product $\Ga\times\Gm$.
\end{lemma}

\begin{proof}
\xref{lem:centralizers-i}
Let $C\in\{\diag(1,1,-2),\diag(1,-2,1),\diag(-2,1,1)\}$. 
The hyperplane section $X_{(1:1)}=\{x_2y_2=0\}\cap\Sigma$ has two irreducible 
components given by $x_2=0$ and $y_2=0$, respectively. Hence
$X_{(1:1)}$ is a union $\mathcal{F}_1\cup\mathcal{F}_2$ of the first kind 
(see~\xref{nota:two-kinds}), and similarly $X_{(1:-2)}$ and $X_{(-2:1)}$ are.
By the spectral mapping theorem (\cite{Ekedahl2004}) one obtains
\begin{equation*}
\operatorname{PCent}_{\tilde{\GL_3(\CC)}}(C)=\operatorname{Cent}_{\tilde{
\GL_3(\CC)}}(C)\cong\tilde{\GL_2(\CC)}\times\z(\tilde{\GL_3(\CC)}).
\end{equation*}
Now the assertion follows from Lemma~\xref{lem:pcentralizer}.

\xref{lem:centralizers-ii}
For $C=\diag(1,\zeta,\zeta^2)$, where $\zeta\neq 1$ and $\zeta^3=1$, one has
\begin{equation*}
\operatorname{PCent}_{\tilde{\GL_3(\CC)}}(C)=\TT(3)\rtimes ((\ZZ/3\ZZ)\times 
(\ZZ/2\ZZ)),
\end{equation*}
where $\TT(3)=\operatorname{Cent}_{\GL_3(\CC)}(C)$ is the diagonal 3-torus of 
$\GL_3(\CC)$, the factor $\ZZ/3\ZZ$ is generated by the cyclic permutation 
matrix $
\left(\begin{smallmatrix}
0 & 0 & 1\\
1 & 0 & 0\\
0 & 1 & 0
\end{smallmatrix}\right)
$,
and $\ZZ/2\ZZ=\langle\tau\rangle$ acts on $\TT(3)$ by the inversion. Passing to 
the quotient by the center $\z(\tilde{\GL_3(\CC)})\subset\TT(3)$ gives the 
result, see Lemma~\xref{lem:pcentralizer}. One can easily check that $X_C$ is 
smooth.

The proof of~\xref{lem:centralizers-iii} is similar.

\xref{lem:centralizers-iv} 
For $C_1$ as in~\eqref{eq:two-matrices}
one has
\begin{equation*}
\operatorname{PCent}_{\tilde{\GL_3(\CC)}}(C_1)=\operatorname{Cent}_{\GL_3(\CC)}
(C_1)=\left\{\left.
\begin{pmatrix}
\lambda & a & 0\\
0 &\lambda & 0\\
0 & 0 &\mu
\end{pmatrix}
,\,\,
\begin{pmatrix}
a &\lambda & 0\\
\lambda & 0 & 0\\
0 & 0 &\mu
\end{pmatrix}
\cdot\tau\,\right|\,\lambda ,\mu\in\CC^*,\,\,a\in\CC\right\}.
\end{equation*}
This group contains the center $\z(\tilde{\GL_3(\CC)})\cong\Gm$ and an element $
\left(\begin{smallmatrix}
0 & 1 & 0\\
1 & 0 & 0\\
0& 0 & 1
\end{smallmatrix}\right)
\cdot\tau$ of order $2$. The quotient by the center is isomorphic to $(\Ga\times
\Gm)\rtimes (\ZZ/2\ZZ)$. The generator of the factor $\ZZ/2\ZZ$ acts by the
inversion $g\mapsto g^{-1}$ on $\Ga\times\Gm$.

We use the notation $X_i=X_{C_i}$, $i=1,2,3$, see~\eqref{eq:pencil}.
The hyperplane section $\mathfrak{G}_{X_1}\in |-K_{X_1}|$ as in 
Lemma~\xref{claim:Du-Val-del-Pezzo}\xref{claim:Du-Val-del-Pezzo-d}
is given by $X_1\cap H_{C_{(1,1)}}$. It is easy to see that $X_1\cap 
H_{C_{(1,1)}}$
consists of four lines. Then by 
Lemma~\xref{claim:Du-Val-del-Pezzo}\xref{claim:Du-Val-del-Pezzo-d}
the surface $X_1$ is a singular del Pezzo 
sextic of type $A_1$.

\xref{lem:centralizers-v} Similarly, for $C_2$ as in~\eqref{eq:two-matrices} one 
has
\begin{equation*}
\operatorname{PCent}_{\tilde{\GL_3(\CC)}}(C_2)=\left\langle
\left.
\begin{pmatrix}
\mu^2\lambda &\mu a & b\\
0 &\mu\lambda & a\\
0 & 0 &\lambda
\end{pmatrix},\,\,
\begin{pmatrix}
b &\mu a &\mu^2\lambda\\
a &\mu\lambda & 0\\
\lambda & 0 & 0
\end{pmatrix}
\cdot\tau\,\right|\,\lambda,\mu\in\CC^*,\,\,a,b\in\CC\right\rangle .
\end{equation*}
This group contains the center 
$\z(\tilde{\GL_3(\CC)})=\{\lambda\operatorname{E}\}_{\lambda\in\CC^*}$ and an 
element $
\left(\begin{smallmatrix}
0 & 0 & 1\\
0 & 1 & 0\\ 
1 & 0 & 0
\end{smallmatrix}\right)
\cdot\tau$ of order~$2$. The quotient 
$\operatorname{PCent}_{\tilde{\GL_3(\CC)}}(C_2)/\z(\tilde{\GL_3(\CC)})$ is 
isomorphic to $((\Ga)^2\rtimes\Gm)\rtimes (\ZZ/2\ZZ)$. None of the 
one-parameter unipotent subgroups is centralized by a $\Gm$-subgroup, as stated.
The corresponding surface $X_2$ contains exactly two lines 
given by $X_2\cap H_{C_{2}^2}$.
We conclude as in~\xref{lem:centralizers-iv} that $X_2$ is an $A_2$-surface.

\xref{lem:centralizers-vi} The matrix $C_3$ is conjugate to $C_2^2$
with centralizer
\begin{equation*}
\operatorname{PCent}_{\tilde{\GL_3(\CC)}}(C_2^2)=\left\{\left.
\begin{pmatrix}
\lambda_0 & a_0 & a_1\\
0 &\lambda_1 & a_2\\
0 & 0 &\lambda_2
\end{pmatrix},\,\,
\begin{pmatrix}
a_1 & a_0 &\lambda_0\\
a_2 &\lambda_1 & 0\\
\lambda_2 & 0 & 0
\end{pmatrix}\cdot\tau
\,\right|\,\lambda_i\in\CC^*,\,a_i\in\CC,\,\,i=0,1,2\right\}.
\end{equation*}
Thus, $\operatorname{PCent}_{\tilde{\PGL_3(\CC)}}(C_3)\cong 
B(\PGL_3(\CC))\rtimes (\ZZ/2\ZZ)$. By Lemma~\xref{lem:pcentralizer} the latter 
group is isomorphic to $\Aut(\Sigma, X_3)$. The hyperplane section 
$X_{3}=\{x_1y_0=0\}\cap\Sigma$ has two irreducible components given by $x_1=0$ 
and $y_0=0$, respectively. Hence $X_{3}$ is a union 
$\mathcal{F}_1\cup\mathcal{F}_2$ of the second kind, see~\xref{nota:two-kinds}.
\end{proof}

\begin{mdefinition}\emph{Proof of Proposition~\xref{thm:toric-case}.}
\label{proof-of-thm:toric-case}
According to~\xref{constr:tensor-product}.4, the pair 
$(\Sigma(V),\Sigma_{\s}(V))$ is equivalent to one of the pairs 
$(\Sigma,X_{(a:b)})$ and $(\Sigma,X_i)$, $i=1,2,3$.
The automorphism groups of the latter pairs are described in 
Lemma~\xref{lem:centralizers}. By 
Corollary~\xref{cor:embedding-Aut-V-to-Aut-X}, 
$\Aut(V)$ embeds in one of this groups. On the other hand, due to Corollary 
\xref{cor:aut-orbits} and Proposition~\xref{lem:fixed-pt}, $\Aut(V)$ contains 
one of the groups $\GL_2(\CC)$, $\Ga\times\Gm$, and $(\Gm)^2$, which is not the 
case for $\Aut(\Sigma,X_2)$. Hence either 
$\Aut(V)\hookrightarrow\Aut(\Sigma,X_{(a:b)})$, or 
$\Aut(V)\hookrightarrow\Aut(\Sigma,X_i)$, $i\in\{1,3\}$.

Inspecting Lemma~\xref{lem:centralizers} we see that, if 
$\GL_2(\CC)\subset\Aut(V)$, then there is an embedding
$\Aut(V)\hookrightarrow\Aut(\Sigma,X_{(1:1)})\cong\tilde{\GL_2(\CC)}$. 
Furthermore, in this case $\Aut^0(V)\cong\GL_2(\CC)$, see 
Lemma~\xref{lem:j=upsilon}, and $\Sigma_{\s}(V)$ is
a union $\mathcal{F}_1\cup\mathcal{F}_2$ of the first kind, see 
Lemma~\xref{lem:centralizers}~\xref{lem:centralizers-i}. The $\GL_2(\CC)$-action 
on $ 
\Sigma_{\s}(V)=\mathcal{F}_{1}\cup\mathcal{F}_{2}$ preserves two disjoint 
lines, 
which are the exceptional sections of the cubic scrolls 
$\mathcal{F}_{i}\cong\FF_1$, $i=1,2$. These lines correspond to the families of 
rulings $\Lambda(S_i)$, $i=1,2$, where the cubic cones $S_1\in\SSS_1(V)$ and 
$S_2\in\SSS_2(V)$ are $\Aut^0(V)$-invariant and have no common ruling, see 
Proposition~\xref{claim:ruling-in-scroll}.
Since $S_1\cdot S_2=0$ in $H^*(V,\ZZ)$, see Proposition~\xref{lem:cohomology}, 
then
$S_1\cap S_2=\emptyset$ due to Corollary~\xref{cor:intersection-of-cones}. This 
corresponds to Case~\xref{thm:toric-case-i} of Proposition 
\xref{thm:toric-case}.

Suppose now that $\rk\Aut(V)\ge 2$, but
$\GL_2(\CC)\not\subset\Aut(V)$. By Lemma~\xref{lem:centralizers}
one of the following holds: either, for a suitable $(a:b)\in\PP^1$,
\begin{equation*}
(\Gm)^2\subset\Aut(V)\subset\Aut(\Sigma,X_{(a:b)})\subset (\Gm)^2\times 
(\ZZ/6\ZZ),
\end{equation*} 
or
\begin{equation}
\label{eq:second-kind}
(\Gm)^2\subset\Aut(V)\subset\Aut(\Sigma,X_3)=
B(\PGL_3(\CC))\rtimes (\ZZ/2\ZZ).
\end{equation}
In the latter case, by Lemma~\xref{lem:centralizers}~\xref{lem:centralizers-iv}
one has $\Sigma_{\s}(V)\cong X_3=Y_{(0:1)}=\mathcal{F}_1\cup\mathcal{F}_2$ is a 
union of the second kind. However, the latter contradicts 
Lemma~\xref{lem:first-kind}.
In the former case, by 
Lemma~\xref{lem:centralizers}~\xref{lem:centralizers-ii}~\xref{lem:centralizers-iii}, 
$\Sigma_{\s}(V)$ is a smooth del Pezzo sextic. This corresponds to Case 
\xref{thm:toric-case-iii} of Proposition~\xref{thm:toric-case}.

Assume further that $\rk\Aut(V)=1$. By Lemma~\xref{lem:centralizers} either 
\eqref{eq:second-kind} holds and we are in Case~\xref{thm:toric-case-iv}, or 
there are embeddings
\begin{equation*}
\Ga\times\Gm\hookrightarrow\Aut(V)\hookrightarrow\Aut(\Sigma,
X_1)=(\Ga\times\Gm)\rtimes (\ZZ/2\ZZ),
\end{equation*} 
and $\Sigma_{\s}(V)\cong X_1$ is a normal del Pezzo sextic of type $A_1$, see 
Lemma~\xref{lem:centralizers}\xref{lem:centralizers-iv}. The latter
corresponds to Case~\xref{thm:toric-case-ii} of Proposition 
\xref{thm:toric-case}. \qed
\end{mdefinition}

\begin{sremark}
\label{lem:Aut-sing-DP}
According to Lemma~\xref{lem:centralizers}\xref{lem:centralizers-ii}, for the
matrix $C=\diag(1,\zeta,\zeta^2)$, where $\zeta$ is a primitive cubic root of
unity, one has $\Aut(\Sigma, X_C)\cong (\Gm )^2\times (\ZZ/6\ZZ)$. However, we
do not know whether this group can be realized as the automorphism group
$\Aut(V)$ for some Fano-Mukai fourfold $V=V_{18}$. In other words, we ignore
whether $X_C$ is equivalent to some $\Sigma_{\s}(V)$ under the
$\Aut(\Sigma)$-action on $\Sigma$.
\end{sremark}

\section{Automorphism groups of $V_{18}$}
\label{sec:1.5-1.6} Let $V=V_{18}$ be a Fano-Mukai fourfold of genus $10$.
By Proposition~\xref{thm:toric-case}, $\Aut^0(V)$ is one of the groups 
$\GL_2(\CC)$, $\Ga\times\Gm$, and $(\Gm)^2$.
The central result of this section is the following theorem.

\begin{theorem}
\label{cor:final}
Given an $\Aut^0(V)$-invariant cubic cone $S\subset V$, consider the pair 
$(W,F_S)$ linked to $(V,S)$, and let $J_S=F_S\cap\Xi$.
With this notation, 
Case~\xref{thm:toric-case-iv} of Proposition~\xref{thm:toric-case} does not 
occur, that is, $\Sigma_{\s}(V)$ 
cannot be a union $\mathcal{F}_1\cup\mathcal{F}_2$ of the second type. Moreover, 
one of 
the following possibilities occurs:
\par\medskip\noindent
{\rm
\setlength{\extrarowheight}{1pt}
\newcommand{\heading}[1]{\multicolumn{1}{c|}{#1}}
\newcommand{\headingl}[1]{\multicolumn{1}{c}{#1}}
\begin{tabularx}{\textwidth}{p{0.03\textwidth}|p{0.1\textwidth}|p{0.12\textwidth
}|p{0.28\textwidth}|X}
&\heading{$\Aut^0(V)$} 
&\heading{$V$}&\heading{$\Sigma_{\s}(V)$}&\headingl{$(\Upsilon,J_S)$}
\\\hline
\mbox{\nnnr
\label{cor:final-i}}&
$\GL_2(\CC)$ & 
$V^{\s}_{18}$ & 
$\mathcal{F}_1\cup\mathcal{F}_2$ of the first kind & 
$J_S=\Upsilon$
\\
\mbox{\nnnr
\label{cor:final-ii}}&
$\Ga\times\Gm$ & 
$V^{\aaa}_{18}$ & 
an $A_1$-del Pezzo sextic & 
\xref{rem:touching-conics}\xref{lem:touching-conics-Ga} for a suitable 
choice\footnote{The cubic cones $S\subset V$ such that $(\Upsilon,J_S)$ 
is of type~\xref{rem:touching-conics}\xref{lem:touching-conics-Ga} correspond 
actually to the two outer $(-1)$-vertices in diagram~\eqref{eq:A1}. The two 
inner $(-1)$-vertices in~\eqref{eq:A1} correspond to the cubic cones $S\subset 
V$ with $(\Upsilon,J_S)$ of 
type~\xref{rem:touching-conics}\xref{lem:touching-conics-Gm}; see 
Remark~\xref{rem:two-linked-actions}.} of $S$
\\
\mbox{\nnnr
\label{cor:final-iii}}&
$(\Gm)^2$ & 
$\not\cong V^{\s}_{18},\,V^{\aaa}_{18}$ & 
a smooth del Pezzo sextic & 
\xref{rem:touching-conics}\xref{lem:touching-conics-Gm} for any choice of 
$S$\\\hline
\end{tabularx}
}\par\vspace{4pt}\noindent
\end{theorem}

For the proof of the first assertion see Corollary~\xref{cor:stable-cones}. 
Then the equivalence between the first and the third 
properties in~\xref{cor:final-i}--\xref{cor:final-iii} follows from Proposition 
\xref{thm:toric-case}. The proof of the remaining statements is done in 
\xref{proof:thm-cor:final} after a certain preparation and several auxiliary 
results.

\begin{lemma}
\label{lem:fixed-pt-unique}
Suppose that $\Aut(V)\supset\SL_2(\CC)$. Then any component
$\SSS_i(V)$, $i=1,2$, of $\SSS(V)$ contains a unique $\Aut^0(V)$-fixed point, 
and this point
corresponds to an $\Aut^0(V)$-invariant cubic cone $S_i\in\SSS_i(V)$. 
Furthermore, one has $S_1\cap S_2=\emptyset$.
\end{lemma}

\begin{proof} The existence of an $\Aut^0(V)$-fixed point in $\SSS_i(V)$ that 
corresponds to a cubic cone is established in Proposition~\xref{lem:fixed-pt}. 
We claim that an
$\SL_2(\CC)$-fixed point in $\SSS_i(V)$ is unique; then, of course, this point 
is a unique
$\Aut^0(V)$-fixed point as well.

Suppose to the contrary that the
$\SL_2(\CC)$-action on $\SSS_i(V)\cong\PP^2$ admits at least two fixed points.
Since
the group $\SL_2(\CC)$ is simply connected, the induced $\SL_2(\CC)$-action on
$\PP^2$ can be lifted to a linear $\SL_2(\CC)$-action on $\CC^3$ trivial
on the one-dimensional subspaces
which correspond to the fixed points. Due to the
complete reducibility, such an $\SL_2(\CC)$-action on $\CC^3$ is trivial. This 
is a contradiction.

By Lemma~\xref{lem:j=upsilon}, under our assumptions one has 
$\Aut^0(V)=\GL_2(\CC)$.

Due to the uniqueness, the unordered pair $(S_1,S_2)$ coincides with the pair 
$(S,S')$ of $\Aut^0(V)$-invariant cubic cones as in 
Lemma~\xref{lem:disjoint-cones}\xref{nota:diagr-2-a}. However, the latter cones 
are disjoint.
\footnote{Alternatively, one can notice that the cones $S_1$ and 
$S_2$ correspond to the disjoint exceptional sections of the components $\mathcal{F}_1$ 
and $\mathcal{F}_2$ of $\Sigma_{\s}(V)$. It follows that 
$S_1\cap S_2$ is zero-dimensional, hence empty since $S_1\cdot S_2=0$ in 
$H^*(V,\ZZ)$, see Proposition \xref{lem:cohomology} 
and Corollary \xref{cor:intersection-of-cones}.}
\end{proof}

\begin{lemma}
\label{lem:F1-plus-F2}
Assume that $\Sigma_{\s}(V)=\mathcal{F}_1\cup\mathcal{F}_2$.
For $i=1,2$ let $S_i\in\SSS_i(V)$ be the $\Aut^0(V)$-invariant cubic cone 
provided by the exceptional section of the ruling $\mathcal{F}_j\to\PP^1$, 
$j\neq i$, see
Corollary~\xref{cor:finiteness}\xref{cor:finiteness-b}. Let $(W,F)$ be the pair 
linked to $(V,S_2)$. Suppose that
$\Aut^0(W,F)$ contains a singular torus of $\Aut^0(W)$.
Then the following hold.
\begin{enumerate}
\renewcommand\labelenumi{\rm (\alph{enumi})}
\renewcommand\theenumi{\rm (\alph{enumi})}

\item
\label{cor:finiteness-aa}
$(\mathcal{F}_1,\mathcal{F}_2)$ is a pair of the first kind;

\item
\label{cor:finiteness-bb}
$\Ru\cap\Aut^0(W,F)=\{1\}$, and $\Aut^0(V)$ is isomorphic to one of the groups 
$\GL_2(\CC),\,\Ga\times\Gm$, and $(\Gm)^2$.
\end{enumerate}
\end{lemma}

\begin{proof}
\xref{cor:finiteness-aa}
By Lemma~\xref{lem:disjoint-cones}\xref{nota:diagr-2-a} there exists a cubic 
cone
$S'\subset V$ disjoint with $S_2$. Since $S'\cdot S_2=0$, these cubic cones
belong to different components of $\SSS(V)$, see Proposition
\xref{lem:cohomology}. Thus, $S'\in\SSS_1(V)$, and so,
$\Lambda(S')\subset\Sigma_{\s}(V)$ is either a ruling of $\mathcal{F}_1$, or
the exceptional section of $\mathcal{F}_2$. However, if
$(\mathcal{F}_1,\mathcal{F}_2)$ were a pair of the second kind, then in both
cases $S'$ and $S_2$ would possess a common ruling, a contradiction. This
proves~\xref{cor:finiteness-aa}.

\xref{cor:finiteness-bb}
Therefore, $(\mathcal{F}_1,\mathcal{F}_2)$ is a pair of
the first kind. Since the cubic cone $S_1$ is $\Aut^0(V)$-invariant, its vertex
$v_1$ is fixed under the action of $\Aut^0(V)$ on $V$. By Corollary
\xref{cor:finiteness}\xref{cor:finiteness-b}, $v_1\notin S_2$. It follows that
$\theta(v_1)\notin R$, see diagram~\eqref{diagram-2}, where $B=R$. The
projection $\theta: V\dashrightarrow W$ with center $\langle S_2\rangle$ as
in diagram~\eqref{diagram-2} sends $v_1$ to a fixed point $\theta(v_1)$ of
$\Aut^0(W,F)$. However, by Proposition
\xref{proposition-GL2-action-c}\xref{proposition-GL2-action-c-c} the unipotent
radical $\Ru$ of $\Aut^0(W)$ acts freely in $W\setminus R$. Thus, $\Ru\cap
\Aut^0(W,F)=\{1\}$. The last assertion follows from Corollary
\xref{cor:aut-center}\xref{cor:aut-center-a}-\xref{cor:aut-center-b}.
\end{proof}

\begin{scorollary}
\label{cor:F1-plus-F2}
Suppose that $\Sigma_{\s}(V)=\mathcal{F}_1\cup\mathcal{F}_2$ is a pair of the 
second kind. Let $(W,F)$ be linked to $(V,S)$, where $S\subset V$ is an 
$\Aut^0(V)$-invariant cubic cone. Then the following hold.
\begin{itemize}
\item
$\rk\Aut(V)=1$;

\item
$\Aut(V)\supset\Ga\times\Gm$;

\item
Any $\Gm$-subgroup of $\Aut(W,F)$ acts nontrivially on $J=F\cap\Xi$.
\end{itemize}
\end{scorollary}

\begin{proof}
Suppose to the contrary that $\rk\Aut(V)=2$. Then $\Aut^0(W,F)$ contains a 
singular torus of $\Aut^0(W)$, and so, by 
Lemma~\xref{lem:F1-plus-F2}\xref{cor:finiteness-aa}, 
$(\mathcal{F}_1,\,\mathcal{F}_2)$ is a pair of the first kind, a contradiction.

Since $\Aut(V)$ of rank $1$ contains one of the groups $\GL_2(\CC)$, 
$\Ga\times\Gm$, $(\Gm)^2$, it contains $\Ga\times\Gm$.

Suppose that there is a $\Gm$-subgroup, say, $Z\subset\Aut(W,F)$ acting 
trivially on $J$. Then $Z$ is contained in the kernel $\Ru\rtimes\z(\Levi)$ of 
the homomorphism $\varrho:\Aut(W,F)\to\PGL_2(\CC)$ in 
\eqref{eq:2nd-sequence-1}. Hence $Z$ is a singular torus of $\Aut(W)$. Since 
$F$ 
is $Z$-invariant, this leads again to a contradiction with Lemma~
\xref{lem:F1-plus-F2}\xref{cor:finiteness-aa}.
\end{proof}

\begin{lemma}\label{lem:ineq-J}
Let $S\subset V$ be a cubic cone, let $(W,F)$ be linked to $(V,S)$, and
let $J=F\cap\Xi$ be the exceptional section of $F$.
Then the following hold.
\begin{enumerate}
\renewcommand\labelenumi{\rm (\alph{enumi})}
\renewcommand\theenumi{\rm (\alph{enumi})}

\item
\label{lem:ineq-J-a}
\begin{equation}
\label{eq:stabiliser}
\dim\Aut^0(W,F)=\dim\Aut^0(V,S)\ge
\begin{cases} 2 &\quad\mbox{if}\quad J\neq\Upsilon,
\\
4&\quad\mbox{if}\quad J=\Upsilon.
\end{cases}
\end{equation}

\item
\label{lem:ineq-J-b} Let $\mathfrak{F}_J$ be the family of all rational normal 
quintic scrolls $F'\subset R$ with exceptional section $J$. Then the equality 
in~\eqref{eq:stabiliser} holds if and only the orbit of $F\in\mathfrak{F}_J$ 
under the natural $\Stab_{\Aut(W)}(J)$-action on $\mathfrak{F}_J$ is open.
\end{enumerate}
\end{lemma}

\begin{proof}
We claim that $\mathfrak{F}_J$ is of pure dimension 4.
Indeed, by a deformation argument
the family of twisted cubics in $R\setminus\Xi=R\setminus\Sing(R)$ has dimension 
$6$.
Given a scroll $F\in\mathfrak{F}_J$, the twisted cubic curves $\Psi$ contained 
in $F$ are the sections of $F\to\PP^1$ disjoint with $J$, that is, the
sections with $\Psi^2=1$.
Therefore, the family of
twisted cubic curves $\Psi\subset F$ is two-dimensional. Now the claim follows.

The group $\Aut(\Upsilon)\cong\PGL_2(\CC)$ acts
naturally on the plane $\Xi\cong\PP^2$ viewed as the symmetric square of
$\Upsilon\cong\PP^1$. By Lemma~\xref{lem:aut-pencils}\xref{lem:aut-pencils-a} 
one has
\begin{equation*}
\dim\Stab_{\Aut(\Xi,\Upsilon)} (J)=
\begin{cases} 
1 &\quad\mbox{if}\quad J\neq\Upsilon,
\\
3&\quad\mbox{if}\quad J=\Upsilon.
\end{cases}
\end{equation*}
The stabilizer $\mathcal{G}:=\Stab_{\Aut(W)} (J)$ acts
on $\mathfrak{F}_J$. From~\eqref{eq:2nd-sequence-1} and 
Proposition~\xref{proposition-GL2-action-c}\xref{proposition-GL2-action-c-c} one 
can deduce the equalities
\begin{equation}
\label{eq:dim-stab}
\dim\mathcal{G}=5+\dim\Stab_{\Aut(\Xi,\Upsilon)} (J)=
\begin{cases} 
6&\quad\mbox{if}\quad J\neq\Upsilon,
\\
8&\quad\mbox{if}\quad J=\Upsilon.
\end{cases}
\end{equation}
Notice that $\Aut(W,F)$ is the stabilizer of $F\in\mathfrak{F}_J$
under the $\mathcal{G}$-action
on $\mathfrak{F}_J$. For the dimension of the orbit $\mathcal{G}.F$ of $F$ under 
this action one has
\begin{equation}
\label{eq:dim-orbit} 
\dim\mathcal{G}.F=\dim\mathcal{G}-\dim\Stab_{\mathcal{G}}(F)=\dim\mathcal{G}
-\dim\Aut^0(W,F)\le\dim\mathfrak{F}_J=4.
\end{equation}
Then~\eqref{eq:dim-stab} and~\eqref{eq:dim-orbit} imply~\eqref{eq:stabiliser}. 
This gives~\xref{lem:ineq-J-a}. Now~\xref{lem:ineq-J-b} is straightforward.
\end{proof}

\begin{lemma}
\label{lem:F1-plus-F2-elimination}
Suppose that $\Sigma_{\s}(V)$ is a union $\mathcal{F}_1\cup\mathcal{F}_2$ of the 
second kind. Then
$\Aut^0(V)\cong\Ga\times\Gm$.
\end{lemma}

\begin{proof}
According to Proposition~\xref{lem:fixed-pt} there exists an 
$\Aut^0(V)$-invariant cubic cone $S\subset V$. Let $(W,F_S)$ be the pair linked 
to $(V,S)$, and let $J_S=F_S\cap\Xi$.
Assume to the contrary that $\Aut^0(V)\cong\Aut^0(W,F_S)\not\cong\Ga\times\Gm$. 
By Corollary~\xref{cor:F1-plus-F2}, $\Aut^0(W,F_S)\cong\Aut^0(V)$ has rank 1,
and any $\Gm$-subgroup of $\Aut^0(W,F_S)$ acts nontrivially on $J_S$.
Hence either $J_S=\Upsilon$ is of 
type~\xref{rem:touching-conics}\xref{lem:touching-conics-Gl2}, or 
$(\Upsilon,J_S)$ is of 
type~\xref{rem:touching-conics}\xref{lem:touching-conics-Gm}. Furthermore, by 
Corollary~\xref{cor:F1-plus-F2}, $\Aut^0(W,F_S)$ contains $\Ga\times\Gm$.

Suppose first that $J_S=\Upsilon$. Consider the unipotent radical 
$\Ru(\Aut^0(W,F_S))$ of the solvable group $\Aut^0(W,F_S)\cong\Aut^0(V)\subset 
B(\GL_3(\CC))/\z(\GL_3(\CC))$ of rank 1, see Case \xref{thm:toric-case-iv} of
Proposition~\xref{thm:toric-case}.
By Lemma~\xref{lem:ineq-J}, in our case $\dim\Aut^0(W,F_S)\ge 4$.
Hence $\dim\Ru(\Aut^0(W,F_S))\ge 3$. The image of $\Ru(\Aut^0(W,F_S))$ in 
$\Aut(\Upsilon)=\PGL_2(\CC)$ is contained in the unipotent radical 
$\Ru(B(\PGL_2(\CC))\cong\Ga$ of a Borel subgroup.
The exact sequence~\eqref{eq:2nd-sequence-1} reads:
\begin{equation}
\label{eq:2nd-sequence-2} 
1\longrightarrow\Ru\rtimes\z(\Levi)\longrightarrow\Aut(W)\stackrel{\varrho}{
\longrightarrow}\Aut(\Upsilon)=\PGL_2(\CC)\longrightarrow 1.
\end{equation}
It follows that the kernel of $\varrho|_{\Ru(\Aut^0(W,F_S))}$ has dimension $\ge 
2$ and is contained in $\Ru$.
Therefore, $\dim(\Ru\cap\Aut^0(W,F_S))\ge 2$. The latter contradicts 
Corollary~\xref{cor:unipotent-actions}.

Thus, $(\Upsilon,J_S)$ is a pair of 
type~\xref{rem:touching-conics}\xref{lem:touching-conics-Gm}, 
$\varrho(\Aut^0(W,F_S))=\Aut^0(\Upsilon,J_S)\cong\Gm$, and 
$\ker(\varrho|_{\Aut^0(W,F_S)})\cong\Ga$, see 
Corollary~\xref{cor:unipotent-actions}.
Finally, one has $\Aut^0(V)\cong\Aut^0(W,F_S)\cong\Ga\times\Gm$.
\end{proof}

\begin{scorollary}
\label{cor:stable-cones}
Case~\xref{thm:toric-case-iv} of Proposition~\xref{thm:toric-case} does not 
occur.
\end{scorollary}

\begin{proof}
Assume to the contrary that $\Sigma_{\s}(V)$ is a union 
$\mathcal{F}_1\cup\mathcal{F}_2$ of the second kind.
By Lemma~\xref{lem:F1-plus-F2-elimination} we have $\Aut^0(V)\cong\Ga\times\Gm$. 
This implies that
any cubic cone $S$ in $V$ is $\Aut^0(V)$-invariant. Indeed,
notice that $\Aut(V,S)=\Aut(V,v(S))$ is the stabilizer in $\Aut(V)$ of the 
vertex $v(S)$ of $S$, see Proposition~\xref{lem:a-d}\xref{lem:a-d(d)}. If $S$ 
were not $\Aut^0(V)$-invariant, then the $\Aut^0(V)$-orbit of $v(S)$ would be a 
curve in $V$, see Lemma~\xref{lem:cubic-cones}. Hence the stabilizer of $S$ in 
$\Aut^0(V)$ would have codimension 1, so $\dim\Aut^0(V,S)=1$. The latter 
contradicts~\eqref{eq:stabiliser}. Thus, the cubic cones in $V$ are 
$\Aut^0(V)$-invariant.

It follows that the $\Aut^0(V)$-action on $\Sigma_{\s}(V)$ preserves any ruling 
$f_{i,t}$ of $\mathcal{F}_i\to\PP^1$, $i=1,2$, see 
Proposition~\xref{claim:ruling-in-scroll}\xref{claim:ruling-in-scroll-a}.
If the factor $\Gm$ of $\Aut^0(V)=\Ga\times\Gm$ acts non-trivially on a general 
ruling $f_{i,t}$ of $\mathcal{F}_i$, then its two fixed points must be fixed by 
the $\Ga$-subgroup. Anyway, at least one of the factors $\Ga$ and $\Gm$ of 
$\Aut^0(V)\cong\Ga\times\Gm$ acts trivially on $\mathcal{F}_i$. The latter 
contradicts Lemma~\xref{lem:group-embedding}.
\end{proof}

We can remove now an extra assumption in 
Lemma~\xref{lem:j=upsilon}\xref{lem:j=upsilon-iii}.

\begin{scorollary}
\label{cor:GL2-stable-cone}
Let $S$ be an $\Aut^0(V)$-invariant cubic cone in $V$, and $(W,F_S)$ be the pair 
linked to $(V,S)$. Let $J_S=F_S\cap\Xi$. Then
$\Aut^0(V)\cong\GL_2(\CC)$ if and only if $J_S=\Upsilon$.
\end{scorollary}

\begin{proof}
The ``only if'' part follows from Lemma~\xref{lem:j=upsilon}.
The ``if'' part follows from Proposition~\xref{thm:toric-case} and 
Corollary~\xref{cor:stable-cones}.
\end{proof}

In the next lemma we examine the cubic cones in $V$ which are not 
$\Aut^0(V)$-invariant.

\begin{lemma}
\label{lem:exceptional-cases-a}
Let $S\subset V$ be a cubic cone, and let $(W,F_S)$ be the pair linked to 
$(V,S)$. Let $J_S=F_S\cap\Xi$.
Assume that $S$ is
not $\Aut^0(V)$-invariant.
Then $\Aut^0(V)\cong\GL_2(\CC)$ and $(\Upsilon,J_S)$ is a pair of 
type~\xref{rem:touching-conics}\xref{lem:touching-conics-Gm}, see 
Figure~\xref{fig-touching-conics}. Furthermore,
\begin{equation}
\label{eq:dim-3}
\Aut^0(V,S)\cong\Aut^0(W,F_S)\cong\Ga\rtimes (\Gm)^2
\cong (\Ga\times\Gm)\rtimes\Gm
,
\end{equation}
where $\Ga=\Ru\cap\Aut^0(W,F_S)$,
a non-abelian subgroup $\Ga\rtimes\Gm\subset\Aut^0(V,S)$
acts on $F_S$ preserving the rulings,
and the $\Gm$-subgroup centralizing the $\Ga$-subgroup acts effectively on 
$J_S$.
\end{lemma}

\begin{proof}
By Proposition~\xref{thm:toric-case} and Corollary~\xref{cor:stable-cones} one 
has $\Sigma_{\s}(V)=\mathcal{F}_1\cup\mathcal{F}_2$, where 
$(\mathcal{F}_1,\mathcal{F}_2)$ is a pair of the first kind, and 
$\Aut^0(V)\cong\GL_2(\CC)$.
By Lemma~\xref{lem:fixed-pt-unique} there are exactly two $\Aut^0(V)$-invariant 
cubic cones $S_1,S_2$ in $V$; these correspond to the exceptional sections of 
$\mathcal{F}_1$ and $\mathcal{F}_2$.
Under the $\Aut^0(V)$-action, the cone $S\neq S_1, S_2$
varies in a one-parameter family of cubic cones, which correspond to the rulings 
of a scroll $\mathcal{F}_i$, $i\in\{1,2\}$.
It follows that $\dim\Aut^0(V,S)=\dim\Aut^0(V)-1=3$. By 
Corollary~\xref{cor:GL2-stable-cone}, $J_S\neq\Upsilon$. Since 
$\Aut^0(V)/\Aut^0(V,S)\cong\PP^1$, then $\Aut^0(V,S)=:B$ is a Borel subgroup of 
$\Aut^0(V)\cong\GL_2(\CC)$. Hence~\eqref{eq:dim-3} holds, and 
$\Aut^0(W,F_S)\cong B$ contains a singular torus of $\Aut(W)$.

The maximal torus of $\Aut^0(W,F_S)$ acts on $J_S$ non-identically with exactly 
two fixed points. It follows by 
Lemma~\xref{lem:aut-pencils}\xref{lem:aut-pencils-a} that $J_S\neq\Upsilon$ is 
of type~\xref{rem:touching-conics}\xref{lem:touching-conics-Gm}, and
the $\Ga$-subgroup of $\Aut^0(W,F_S)$ acts identically on $J_S$. Hence it 
preserves each ruling of $F_S$.
In the notation of Corollary~\xref{cor:aut-center} one has 
$\Ru\cap\Aut^0(W,F_S)\cong\Ga$ and
$G\cong (\Gm)^2$. Now the remaining assertions are immediate.
\end{proof}

\begin{lemma}
\label{lem:exceptional-cases-c}
Let $S\subset V$ be a cubic cone, let $(W,F_S)$ be the pair linked to $(V,S)$, 
and let $J_S=F_S\cap\Xi$.
Assume that $(\Upsilon,J_S)$ is of 
type~\xref{rem:touching-conics}\xref{lem:touching-conics-Ga} and $F_S$ is 
$Z$-invariant, where $Z\subset\Aut(W)$ is a singular torus. Then
the cone $S$ is $\Aut^0(V)$-invariant, $\Ru\cap\Aut^0(W,F_S)=\{1\}$, and 
$\Aut^0(V)\cong\Ga\times\Gm$.
\end{lemma}

\begin{proof}
Due to 
Corollary~\xref{cor:aut-center}\xref{cor:aut-center-a}-\xref{cor:aut-center-b} 
and its proof one has
\begin{equation}
\label{eq:aut-dim-3}
\Aut^0(V,S)\cong\Aut^0(W,F_S)\cong (\Ru\cap\Aut^0(W,F_S))\rtimes (\Ga\times\Gm),
\end{equation}
where the $\Gm$-subgroup $Z$ preserves the rulings of $F_S$ and the 
$\Ga$-subgroup centralized by the $\Gm$-subgroup acts effectively on $J_S$. In 
particular, $\rk\Aut^0(V,S)=1$.

If the cone $S$ is $\Aut^0(V)$-invariant, then $\rk\Aut^0(V)=1$, and so, 
$\Aut^0(V)\cong\Ga\times\Gm$ by Proposition~\xref{thm:toric-case}. Hence 
$\Ru\cap\Aut^0(W,F_S)=\{1\}$ due to~\eqref{eq:aut-dim-3}.

Otherwise, $\dim\Aut^0(V)=\dim\Aut^0(V,S)+1\ge 3$. Hence by 
Proposition~\xref{thm:toric-case} $\Aut^0(V)=\GL_2(\CC)$. So, 
Lemma~\xref{lem:exceptional-cases-a} applies to the pair $(V,S)$. According to 
this lemma, $\rk\Aut^0(V,S)=2$ contrary to~\eqref{eq:aut-dim-3}.
\end{proof}

\begin{mdefinition}
\label{ss:special-FM-4folds}
We use the notation $V^{\s}_{18}$ and $V^{\aaa}_{18}$ introduced 
in~\xref{nota:aut-orbits}. Let as before $\mathfrak{g}_2=\Lie(\G)$. Recall that 
$V^{\s}_{18}=V^{g_{\s}}$ ($V^{\aaa}_{18}=V^{g_{\aaa}}$, respectively) for a 
singular semisimple element $g_{\s}\in\mathfrak{g}_2$ (a regular non-semisimple 
element $g_{\aaa}\in\mathfrak{g}_2$, respectively). Due to 
Proposition~\xref{prop:orbits}\xref{prop:orbits-b}, such nonzero elements 
$g_{\s}$ and $g_{\aaa}$ do exist, and their images $[g_{\s}]$ and $[g_{\aaa}]$ 
in $\PP(\mathfrak{g}_2)$ form two distinct orbits of the induced 
$\Ad(\G)$-action on $\PP(\mathfrak{g}_2)$. Therefore, the Fano-Mukai fourfolds 
$V^{\s}_{18}$ and $V^{\aaa}_{18}$ of genus 10 do exist and are unique up to 
isomorphism. By Corollary~\xref{cor:aut-orbits} one has 
$\Aut^0(V^{\s}_{18})\cong\GL_2(\CC)$ and 
$\Aut^0(V^{\aaa}_{18})\supset\Ga\times\Gm$. Moreover, the following hold.
\end{mdefinition}

\begin{slemma}
\label{cor:case-ii}
Assume that $V\cong V^{\aaa}_{18}$. Then $\Aut^0(V)=\Ga\times\Gm$, and any cubic 
cone $S\subset V$ is $\Aut^0(V)$-invariant. Let $(W,F_S)$ be the pair linked to 
$(V,S)$. Then, for a suitable choice of $S$, $J_S=F_S\cap\Xi$ is of 
type~\xref{rem:touching-conics}\xref{lem:touching-conics-Ga}, and $F_S$ is 
invariant under a singular torus $\z(\Levi)\subset\Aut(W)$.
\end{slemma}

\begin{proof}
Fix a smooth conic $J'\subset\Xi$ of 
type~\xref{rem:touching-conics}\xref{lem:touching-conics-Ga} and a Levi subgroup 
$\Levi$ of $\Aut(W)$. By 
Proposition~\xref{prop:center-inv-scroll}\xref{prop:center-inv-scroll-a} there 
exists a $\z(\Levi)$-invariant quintic scroll $F'\subset R$ such that 
$J'=F'\cap\Xi$. Let $(V',S')$ be the pair linked to $(W,F')$. Then by 
Lemma~\xref{lem:exceptional-cases-c}
the cone $S'$ is $\Aut^0(V')$-invariant, and $\Aut^0(V')\cong\Ga\times\Gm$.

By Theorem~\xref{thm:Mukai}, $V'=V^g_{18}$ for some $g\in\mathfrak{g}_2$. By 
Corollary~\xref{cor:aut-orbits}, $g$ is regular non-semisimple. Hence $V'\cong 
V^{\aaa}_{18}\cong V$, and so, $\Aut(V)\cong\Aut(V')\cong\Ga\times\Gm$. We 
identify $V$ and $V'$ via this isomorphism.

By Proposition~\xref{thm:toric-case},
$\Sigma_{\s}(V)$ is irreducible and contains exactly $4$ lines. Therefore, $V$ 
contains exactly $4$ cubic cones, which are all $\Aut^0(V)$-invariant. By 
construction, the cone $S\subset V$ which corresponds to $S'\subset V'$ 
satisfies the conditions of the lemma.
\end{proof}

\begin{mdefinition}\emph{Proof of Theorem~\xref{cor:final}.}
\label{proof:thm-cor:final}
Corollary~\xref{cor:GL2-stable-cone} yields the equivalence
\begin{equation*}
J_S=\Upsilon\Longleftrightarrow\Aut^0(V)\cong\GL_2(\CC).
\end{equation*}
If $V\cong V^{\s}_{18}$, then $\Aut^0(V)\cong\GL_2(\CC)$ by Corollary 
\xref{cor:aut-orbits}\xref{cor:aut-orbits-i}. Conversely, suppose that 
$\Aut^0(V)\cong\GL_2(\CC)$. By Theorem~\xref{thm:Mukai}, $V=V^g_{18}$ for some 
$g\in\mathfrak{g}_2$, where $g$ is singular semisimple, see Proposition 
\xref{prop:regular-centralizers}\xref{prop:regular-centralizers-a}--\xref{prop:regular-centralizers-c}. 
Hence $V\cong V^{\s}_{18}$, see Proposition 
\xref{prop:orbits}. This yields the equivalence
\begin{equation}
\label{eq:GL2-as}
\Aut^0(V)\cong\GL_2(\CC)\quad\Longleftrightarrow\quad V\cong V^{\s}_{18}.
\end{equation}
Combining with Proposition~\xref{thm:toric-case} this ends the proof 
of~\xref{cor:final}\xref{cor:final-i}.

Suppose further that $\Aut^0(V)\cong\Ga\times\Gm$, and let $V=V^g_{18}$ for 
some $g\in\mathfrak{g}_2$. By Proposition~\xref{prop:regular-centralizers} and 
Corollary~\xref{cor:aut-orbits}, $g$ is regular non-semisimple, that is, 
$V\cong 
V^{\aaa}_{18}$. By Lemma~\xref{cor:case-ii} this gives the equivalence
\begin{equation}
\label{eq:Ga-Gm}\Aut^0(V)\cong\Ga\times\Gm\quad\Longleftrightarrow\quad V\cong 
V^{\aaa}_{18}.
\end{equation}
From~\eqref{eq:GL2-as} and~\eqref{eq:Ga-Gm} we deduce
\begin{equation*}
\Aut^0(V)\cong(\Gm)^2\quad\Longleftrightarrow\quad V\not\cong 
V^{\s}_{18},\,V^{\aaa}_{18}.
\end{equation*}

If $\Aut^0(V)\cong(\Gm)^2$, then
the pair $(\Upsilon, J_S)$ is of 
type~\xref{rem:touching-conics}\xref{lem:touching-conics-Gm} for any cubic cone 
$S\subset V$. Thus, if for some $\Aut^0(V)$-invariant cubic cone $S\subset V$ 
the corresponding pair $(\Upsilon, J_S)$ is of 
type~\xref{rem:touching-conics}\xref{lem:touching-conics-Ga}, then 
$\rk\Aut^0(V)=1$ and $\Aut^0(V)\cong\Ga\times\Gm$. By Lemma~\xref{cor:case-ii} 
we have the converse implication, and so, the equivalence
\begin{equation*}
\Aut^0(V)\cong\Ga\times\Gm\quad\Longleftrightarrow\quad (\Upsilon, 
J_S)\,\,\,\mbox{is of 
type~\xref{rem:touching-conics}\xref{lem:touching-conics-Ga}}\,\,\,\mbox{for a 
cubic cone}\,\,\, S\subset V .
\end{equation*}
Together with Proposition~\xref{thm:toric-case} and 
Corollary~\xref{cor:stable-cones} this 
proves~\xref{cor:final}\xref{cor:final-ii}.

From~\xref{cor:final}\xref{cor:final-i}--\xref{cor:final}\xref{cor:final-ii} we 
deduce the equivalence
\begin{equation*}
\Aut^0(V)\cong(\Gm)^2\quad\Longleftrightarrow\quad (\Upsilon,J_S)\,\,\,\mbox{is 
of type~\xref{rem:touching-conics}\xref{lem:touching-conics-Gm} for any cubic 
cone}\,\,\, S\subset V.
\end{equation*}
Due to Proposition~\xref{thm:toric-case} and Corollary~\xref{cor:stable-cones} 
this completes the proof of~\xref{cor:final}\xref{cor:final-iii}.
\qed
\end{mdefinition}

\begin{sremark}
\label{rem-deform-space}
By virtue of 
Proposition~\xref{prop:center-inv-scroll}\xref{prop:center-inv-scroll-b}, for a 
conic $J\subset\Xi$ touching $\Upsilon$ with even multiplicities and a Levi 
subgroup $\Levi$ of $\Aut(W)$ there is a unique $\z(\Levi)$-invariant scroll 
$F\subset R$ with $J=F\cap\Xi$ if $J=\Upsilon$, and exactly two distinct such 
scrolls otherwise. Since the singular tori in $\Aut(W)$ are conjugated under 
the 
$\Ru$-action, the quintic scrolls $F\subset R$ with given $J=F\cap R$ are 
equivalent under the $\Ru$-action on $W$ up to passing to the ``conjugate'' 
scroll in the case $J\neq\Upsilon$. The variety $\mathcal{F}_J$ of such rational 
quintic scrolls is isomorphic to $\CC^4$ if $J=\Upsilon$. For $J$ of 
type~\xref{rem:touching-conics}\xref{lem:touching-conics-Ga}, $\mathcal{F}_J$ 
consists of two disjoint components isomorphic to $\CC^4$. For $J$ of 
type~\xref{rem:touching-conics}\xref{lem:touching-conics-Gm}, $\mathcal{F}_J$ 
has two disjoint components isomorphic to $\CC^4$, and then eventually also 
some 
number of lower-dimensional components. Anyway, the automorphism $\tilde\kappa$ 
as in Proposition~\xref{prop:center-inv-scroll}~\xref{prop:center-inv-scroll-d} 
interchanges these two $\CC^4$-components. Cf.\ also 
Lemma~\xref{lem:ineq-J}\xref{lem:ineq-J-b}.
\end{sremark}

\section{Proofs of the main theorems and beyond}
\label{sec:proofs}
In this section we prove Theorems~\xref{thm:main} --~\xref{thm:main-aut-n} 
from the Introduction. Besides, Remark~\xref{cor:moduli} and 
Theorem~\xref{thm:main-aut} complement the main results.

\begin{mdefinition}\emph{Proof of Theorem~\xref{thm:main}.} By 
Proposition~\xref{lem:fixed-pt}, $V$ contains two distinct $\Aut^0(V)$-invariant 
cubic cones
$S_i\in\SSS_i(V)$, $i=1,2$.
To the pair $(V,S_i)$ there corresponds an $\Aut^0(V)$-equivariant Sarkisov link
\eqref{diagram-2}. Now
Theorem~\xref{thm:main} follows immediately from Corollary
\xref{rem:interrompu}.
\qed
\end{mdefinition}

\begin{mdefinition}
{\emph{Proof of Theorem~\xref{thm:G/P}}.}
By Proposition~\xref{prop:orbits} and Theorem~\xref{cor:final}, for any 
$g\in\mathfrak{g}_2$ such that $[g]\notin D_{\ol}$ one has 
$\Aut^0(V^g)=\Stab_{\G}(g)^0$. Now the result follows.
\qed
\end{mdefinition}

From the proof we deduce the following corollary (cf.~\xref{sit:citations} 
below).

\begin{scorollary}
\label{cor:induced-action}
For each $g\in\mathfrak{g}_2$ such that $[g]\notin D_{\ol}$ the group 
$\Aut^0(V^g)$ is the identity component of the stabilizer of $V^g$ in 
$\Aut^0(\Omega)\cong\G$.
\end{scorollary}

\begin{mdefinition}
{\emph{Proof of Theorem~\xref{thm:main-aut-n}}.}
The existence and the uniqueness in 
Theorem~\xref{thm:main-aut-n}\xref{thm:main-aut-n-GL2} 
and~\xref{thm:main-aut-n-Ga} follow from~\xref{ss:special-FM-4folds} by virtue 
of Theorem~\xref{cor:final}. Since $\Aut^0(V_{18}^{\s})\cong\GL_2(\CC)$ comes 
from a Levi subgroup $\Levi\subset\Aut(W)$, and $\Levi$ acts on $W$ with a 
principal open orbit, see Proposition \ref{proposition-GL2-action}(c), then the 
latter holds as well for the induced $\GL_2(\CC)$-action on $V_{18}^{\s}$. 
As for the description of the fixed points in~\xref{thm:main-aut-n-GL2} see Theorem \ref{thm:main-aut}\xref{prop:vertices-a} below. 
Inclusions~\eqref{eq:aut-Gm-2} follow from Proposition~\xref{thm:toric-case} and 
Theorem~\xref{cor:final}. Let us show~\eqref{eq:aut-GL2} 
and~\eqref{eq:aut-Ga-Gm}.

Suppose that $\Aut^0(V)\cong\GL_2(\CC)$, that is, $V\cong V_{18}^{\s}$. We claim 
that there is an exact sequence
\begin{equation}
\label{eq:decomposition}
1\longrightarrow\Aut^0(V)\longrightarrow\Aut(V)\longrightarrow\ZZ/2\ZZ\longrightarrow 0.
\end{equation}
Indeed, let $(S_1,S_2)$
be the unique pair of $\Aut^0(V)$-invariant cubic cones, see 
Lemma~\xref{lem:fixed-pt-unique}.
By Corollary~\xref{cor:unique-GL2} there exists an isomorphism $\tau_V: 
(V,S_1)\stackrel{\cong}{\longrightarrow}(V,S_2)$.
Due to the uniqueness of the pair $(S_1,S_2)$ one has $\tau_V(S_2)=S_1$. 
Therefore, $\tau_V^2\in\Aut(V,S_1)$. Since the group 
$\Aut(V,S_1)\cong\GL_2(\CC)$ is connected, one obtains $\tau_V^2\in\Aut^0(V)$. 
Moreover, $\Aut^0(V)$ and $\tau_V$ generate
$\Aut(V)$. This proves our claim.

Thus, $\Aut(V)/\Aut^0(V)\cong\ZZ/2\ZZ$. It follows that the embedding 
$\Aut(V)\hookrightarrow\GL_2(\CC)\rtimes (\ZZ/2\ZZ)$ in Case \xref{thm:toric-case-i} of 
Proposition~\xref{thm:toric-case} is an isomorphism. 
The 
last assertion of Theorem~\xref{thm:main-aut-n}\xref{thm:main-aut-n-GL2} 
is straightforward.

Similarly, we claim that~\eqref{eq:decomposition} still holds in the case where $V\cong 
V^{\aaa}_{18}$, and so,~\eqref{eq:aut-Ga-Gm} follows by Case \xref{thm:toric-case-ii} of
Proposition~\xref{thm:toric-case}.

Indeed, among the four cubic cones contained in $V$, there is a unique pair of 
disjoint cones. Namely, these are the cubic cones $S_1,S_2$ which do not contain 
the ruling corresponding to the singular point of $\Sigma_{\s}(V)$, see 
diagram~\eqref{eq:A1}.

Consider the pair $(W,F_2)$ linked to $(V,S_2)$. The image of $S_1$ in $W$ is a 
cubic cone $S_{1,W}$, which meets $F_2$ along a twisted cubic section $\Psi$. 
The construction being $\Aut^0(V)$-equivariant, the image of 
$\Aut^0(V)\cong\Ga\times\Gm$ in $\Aut(W)$ is contained in the stabilizer of the 
vertex $v(S_{1,W})\in W\setminus R$.

The latter stabilizer is a Levi subgroup, say, $\Levi_2\subset\Aut(W)$, see 
Proposition~\xref{proposition-GL2-action}\xref{proposition-GL2-action-cc}. The 
$\Ga$-subgroup of $\Aut^0(V)$ cannot preserve the rulings of $S_{1,W}$ and 
$F_2$, respectively. Indeed, otherwise on such a ruling $l$ it would have two 
fixed points $l\cap\Psi$ and $v(S_{1,W})$ ($l\cap\Psi$ and $l\cap\Xi$, 
respectively). In particular, $\Aut^0(V)$ would act identically on $S_1$, which 
is impossible. It follows that
the $\Gm$-subgroup of $\Aut^0(V)$ corresponds to the singular torus 
$\z(\Levi_2)$ fixing $\Psi$ and $J_2:=F_2\cap\Xi$ pointwise and preserving the 
rulings of both $S_{1,W}$ and $F_2$. By contrast, the $\Ga$-subgroup of 
$\Aut^0(W,F_2)\cong\Aut^0(V)$ acts effectively on $J_2$. Hence $(\Upsilon,J_2)$ 
is a pair of type~\xref{rem:touching-conics}\xref{lem:touching-conics-Ga}.

Symmetrically, for the pair $(W,F_1)$ linked to $(V,S_1)$ the $\Gm$-subgroup of
$\Aut^0(W,F_1)\cong\Ga\times\Gm$ acts trivially on $J_1=F_1\cap\Xi$, while the 
$\Ga$-subgroup acts effectively. So, $(\Upsilon,J_1)$ is as well a pair of 
type~\xref{rem:touching-conics}\xref{lem:touching-conics-Ga}.

We claim that the pairs $(W,F_1)$ and $(W,F_2)$ are isomorphic. Indeed, the 
corresponding Levi subgroups $\Levi_1$ and $\Levi_2$ are conjugate in $\Aut(W)$. 
Hence, up to an automorphism of $W$, one may suppose that 
$\Levi_1=\Levi_2=:\Levi$. By 
Lemma~\xref{lem:aut-pencils}\xref{lem:aut-pencils-b} the pairs $(\Upsilon, J_1)$ 
and $(\Upsilon, J_2)$ of 
type~\xref{rem:touching-conics}\xref{lem:touching-conics-Ga} are isomorphic 
under the $\Aut(\Xi,\Upsilon)$-action. This isomorphism can be realized by an 
element of $\Levi$. Hence one may suppose also that $J_1=J_2=:J$. By virtue of 
Proposition~\xref{prop:center-inv-scroll}\xref{prop:center-inv-scroll-b} 
and~\xref{prop:center-inv-scroll-d}, under these assumptions there is an 
isomorphism of pairs $(W,F_1)\cong (W,F_2)$, as claimed.

The latter isomorphism induces an isomorphism of linked pairs $\tau_V: 
(V,S_1)\stackrel{\cong}{\longrightarrow} (V,S_2)$. Since $\{S_1,S_2\}$ is the 
only pair of disjoint cubic cones in $V$, one has $\tau_V^2\in\Aut(V,S_i)$, 
$i=1,2$. Since $\Aut^0(V)\subset\Aut(V,S_i)$ and $[\Aut(V):\Aut^0(V)]\le 2$, see 
Proposition~\xref{thm:toric-case}.\xref{thm:toric-case-ii}, we 
deduce~\eqref{eq:decomposition}, as desired.
\qed
\end{mdefinition}

\begin{sremark}
\label{rem:two-linked-actions}
Let again $V\cong V^{\aaa}_{18}$, and let $\{S_1',S_2'\}$ be the pair of cubic 
cones in $V$ with a common ruling, which corresponds to the unique singular 
point of $\Sigma_{\s}(V)$, cf.\ diagram~\eqref{eq:A1}. For the corresponding 
linked pairs $(W,F_j')$, $j=1,2$, one has 
$\Aut^0(W,F_j')\cong\Aut^0(V,S_j')\cong\Ga\times\Gm$. The argument in the proof 
above shows that the $\Gm$-subgroup of $\Aut^0(W,F_j')$ acts nontrivially on 
$J'_j=F_j'\cap\Xi$, while the $\Ga$-subgroup acts identically on $J'_j$. It 
follows that $\Ru\cap\Aut^0(W,F_j')\cong\Ga$, and $(\Upsilon, J_j')$ is a pair 
of type~\xref{rem:touching-conics}\xref{lem:touching-conics-Gm} for $j=1,2$.
\end{sremark}

\begin{remark}
\label{cor:moduli}
It is a folklore that
the moduli space $\mathcal{M}_{18}$ of the Fano-Mukai fourfolds of genus $10$ is 
one-dimensional.
Indeed, this can be seen as follows.

Identify $\G$ with its image in 
$\Aut(\PP(\mathfrak{g}_2^\vee))\cong\PGL_{14}(\CC)$ under the dual of the 
adjoint representation. The open set 
$\mathcal{U}=\PP(\mathfrak{g}_2^\vee)\setminus (D_{\ol}\cup D_{\s})$ is 
$\G$-invariant. Each point $[g]\in\mathcal{U}$ corresponds to the hyperplane 
section $V^g=\Omega\cap g^\bot$; the latter is a Fano-Mukai fourfolds of genus 
$10$ with $\Aut^0(V^g)\cong (\Gm)^2$, see Theorem 
\xref{thm:main-aut-n}\xref{thm:main-aut-n-Gm}. Then $\mathcal{M}_{18}$ is 
dominated by the one-dimensional quotient $\mathcal{U}/\G$, see 
Proposition~\xref{prop:orbits}. The fiber of $\mathcal{U}\to\mathcal{U}/\G$ 
through a point $[g]\in\mathcal{U}$ is isomorphic to $\G/\Stab_{\G}(V^g)$. By 
Corollary~\xref{cor:induced-action} one has
\begin{equation*}
\Aut^0(V^g)=\Stab_{\G}(V^g)^0\subset\Stab_{\G}(V^g)\subset\Aut(V^g).
\end{equation*}
The fiber of the morphism $\mathcal{U}/\G\to\mathcal{M}_{18}$ though the image 
of $[g]$ in $\mathcal{U}/\G$ is isomorphic to $\Aut(V^g)/\Stab_{\G}(V^g)$. The 
latter is a cyclic group whose order is a factor of $6$, 
see~\eqref{eq:aut-Gm-2}. Hence $\mathcal{U}/\G\to\mathcal{M}_{18}$ is a finite 
morphism, and so, $\dim\mathcal{M}_{18}=1$.

Notice, however, that the moduli space $\mathcal{M}_{18}$ does not exist as a 
separated scheme. Indeed, the points
$[V^{\s}_{18}]$ and $[V^{\aaa}_{18}]$ do not admit disjoint neighborhoods in 
$\mathcal{M}_{18}$, as follows from Proposition~\xref{prop:orbits}.
\end{remark}

In addition to Theorem~\xref{thm:main-aut-n}\xref{thm:main-aut-n-GL2} we have 
the following results. Item \xref{prop:vertices-e} will be used in the next section.

\begin{theorem}
\label{thm:main-aut}
Let $V=V_{18}^{\s}$ be the Fano-Mukai fourfold 
as in Theorem~\xref{thm:main-aut-n}\xref{thm:main-aut-n-GL2} and let 
$S_1,\, S_2 \subset V$ be the unique $\Aut^0(V)$-invariant cubic cones \textup(see 
Lemma~\xref{lem:fixed-pt-unique}\textup). 
Let $A_i=A_{S_i}$,\, $i\in\{1,2\}$ be the unique $\Aut^0(V)$-invariant
hyperplane section with $\Sing(A_i)=S_i$. 
Then the following hold.
\begin{enumerate}
\renewcommand\labelenumi{\rm (\alph{enumi})}
\renewcommand\theenumi{\rm (\alph{enumi})}
\item
\label{prop:vertices-b}
For $i\neq j$, $A_j$ cuts $S_i$ along a rational twisted cubic curve $\Gamma_j\subset S_i$. The curves $\Gamma_1$ and $\Gamma_2$ are disjoint, the union $\Gamma_1\cup\Gamma_2$ is an $\Aut(V)$-orbit contained in $A_1\cap 
A_2$ and pointwise fixed under the $\z(\Aut^0(V))$-action.

\item
\label{prop:vertices-c}
For $j=1,2$ consider the family $(S_{j,t})_{t\in\PP^1}$ of cubic cones on $V$
such that $\Lambda(S_{j,t})$ is a ruling of $\mathcal{F}_j\to\PP^1$
\textup(see Corollary~\xref{cor:finiteness}\xref{cor:finiteness-b}\textup).
Then each point of $\Gamma_j$ is the vertex $v_{j,t}=v(S_{j,t})$,
and for $j\neq i$ one has $A_i=\bigcup_{t\in\PP^1} S_{j,t}$. 

\item
\label{prop:vertices-f} The twisted cubics
$\Gamma_1$ and $\Gamma_2$ are sections of a rational normal scroll $D$ of degree $6$ which is the image of $\PP^1\times \PP^1$ embedded to $\PP^7\subset \PP^{12}$ by the 
linear system of bidegree $(1,3)$. One has $(A_1\cap A_2)_{\rm red}=D$.
\item
\label{prop:vertices-a} 
The vertices $v(S_i)$, $i=1,2$, are the unique
$\GL_2(\CC)$-fixed points in $V$. Furthermore, $D\setminus (\Gamma_1\cup\Gamma_2)$ is an orbit of 
$\GL_2(\CC)$.
\item

\label{prop:vertices-g}
$A_1$ and $A_2$ are two components of the branching divisor $\mathcal{B}$ of the 
morphism $s:\LLL(V)\to V$ in~\eqref{equation-universal-family-V}.

\item
\label{prop:vertices-e}
The Fano fourfold $V$ is covered by the affine charts 
\begin{equation}\label{eq:affine-charts} U_i=V\setminus A_i\cong\CC^4\quad\mbox{and}\quad U_{i,t}=V\setminus A_{i,t}\cong\CC^4,\quad t\in\PP^1,\,\,\, i=1,2\,,\end{equation} 
where $A_{i,t}=A_{S_{i,t}}$ is the unique
hyperplane section of $V$ with $\Sing(A_{i,t})=S_{i,t}$.
\end{enumerate}
\end{theorem}

\begin{proof} 
\xref{prop:vertices-b} 
Since $S_1\cap S_2=\emptyset$ one has $v(S_j)\notin A_i$ for $j\neq i$. This 
yields the first assertion. The cones $S_1$ and $S_2$ being disjoint (see Lemma 
\ref{lem:fixed-pt-unique}) also the curves $\Gamma_1$ and $\Gamma_2$ are. 
Since $\Gamma_i\subset S_j\subset A_j$ one has $\Gamma_i\subset A_i\cap A_j$, $i=1,2$.

The 
factor $\ZZ/2\ZZ$ of $\Aut(V)$ in 
Theorem~\xref{thm:main-aut-n}~\xref{thm:main-aut-n-GL2} is generated by an 
involution $\tau\in\Aut(V)\setminus\Aut^0(V)$ interchanging $S_1$ and $S_2$. 
Then $\tau$ switches also $\Gamma_1$ and $\Gamma_2$. Since both $A_i$ and $S_j$ are $\Aut^0(V)$-invariant then $\Gamma_i=A_i\cap S_j$ is. 
The curve 
$\Gamma_i$ is an orbit of $\Aut^0(V)$. It is pointwise fixed under the $\z(\Aut^0(V))$-action, and $\Gamma_1\cup\Gamma_2$ is an orbit of $\Aut(V)$. 

\xref{prop:vertices-c} Since $S_{j,t}$ and $S_i$ contain a common ruling one has $v(S_{j,t})\in 
S_i$, and so, $S_{j,t}\subset A_i$ (recall that $A_i$ is the union of lines in $V$ 
meeting $S_i$, see Lemma \ref{lemma--SS}). The union 
$\bigcup_{t\in\PP^1}S_{j,t}\subset A_i$ is a closed subvariety of dimension 3. 
Since $A_i$ is irreducible of dimension $3$ one has 
$A_i=\bigcup_{t\in\PP^1} S_{j,t}$. The set of vertices 
$\{v(S_{j,t})\}_{t\in\PP^1}\subset S_i$ is $\Aut^0(V)$-invariant, hence,
$\{v(S_{j,t})\}$ is a closed one-dimensional $\Aut^0(V)$-orbit in $S_i$. The center $\z(\Aut^0(V))$ acts nontrivially on any ruling $l$ of $S_i$ with just two fixed points, $v(S_i)$ and $\Gamma_i\cap l$. Therefore, $\Gamma_j=S_i\cap A_j$ is the only one-dimensional $\Aut^0(V)$-orbit in $S_i$. 
It follows that $\{v(S_{j,t})\}=\Gamma_j$. 

\xref{prop:vertices-f} 
The conic 
$C=\mathcal{F}_1\cap\mathcal{F}_2\subset\Sigma_{\s}(V)$ parameterizes a family 
of lines $(l_t)$ in $V$, where $l_t\subset S_{1,t}\cap S_{2,t}$. Since 
$v(S_{j,t})\in\Gamma_j\cap l_t$, $j=1,2$, the line $l_t$ meets both $S_1$ and 
$S_2$, hence is contained in $A_1\cap A_2$. Thus, $D=\bigcup_{t\in C} l_t\subset 
A_1\cap A_2$ is a rational normal scroll. 

According to Corollary \ref{cor:finiteness}(b) the exceptional section $\Lambda(S_2)$ of $\mathcal{F}_1$ and a ruling $\Lambda(S_{2, t})$ of $\mathcal{F}_2$ project to two distinct points of $\SSS_2(V)\cong\PP^2$. Since $S_2$ and $S_{2,t}$ are members of the same family $\SSS_2(V)$ one has $S_2\cdot S_{2,t}=1$, see Proposition \ref{lem:cohomology}. Since they have no common ruling, they meet transversally in one point, say, $P_{2,t}$, which is smooth on both $S_2$ and $S_{2, t}$, see Corollary \ref{cor:intersection-of-cones}. Hence there is a unique ruling $l_{2,t}\ni P_{2,t}$ of $S_{2, t}$ which meets $S_2$. By symmetry, there is a unique ruling $l_{1,t}$ of $S_{1, t}$ which meets $S_1$. Notice that the unique common ruling $l_t$ of $S_{1, t}$ and $S_{2, t}$ meets both $S_1$ and 
$S_2$, hence $l_{1,t}=l_{2,t}=l_t\subset D\subset A_1\cap A_2$, see the proof of \xref{prop:vertices-c}. 

The hyperplane section $A_1$ passes through the vertex $v_{1,t}\in\Gamma_1$ of $S_{1,t}\subset A_2$ cutting $S_{1,t}$ along a union of at most 3 rulings. Since $v_{1,t}\in\l_t\subset \Gamma_1=A_1\cap S_2$ each of these rulings meets $S_{2}$. It follows by the preceding that the ruling of $S_{1,t}$ which meets $S_{2}$ is unique and coincides with $l_t$. We conclude that 
$l_t$ is a triple intersection of $S_{1,t}\subset A_2$ and $A_1$. Therefore,
$A_1\cap A_2=3D$, and so, $(A_1\cap A_2)_{\red}=D$, as stated. 
We leave to the 
reader to check the remaining statements of \xref{prop:vertices-f}.

\xref{prop:vertices-a} 
Since $S_1,\, S_2$ are $\Aut^0(V)$-invariant 
their vertices $v(S_1)$ and $v(S_2)$ are two distinct fixed points of $\Aut^0(V)$. 
The hyperplane sections $A_1$ and $A_2$ are 
$\Aut^0(V)$-invariant. For $i=1,2$ the center $\z(\Aut^0(V))$ acts on $V\setminus A_i\cong\mathbb{C}^4$ via homotheties with a unique fixed point $v(S_i)$. According to \xref{prop:vertices-b} and \xref{prop:vertices-c}, $\z(\Aut^0(V))$ fixes also the vertex $v_{j,t}$ of each cone $S_{j,t}$ leaving the cone invariant.

We claim that $\z(\Aut^0(V))$ acts nontrivially on each ruling $l_t$ of $D$. 
Indeed, suppose to the contrary that $\z(\Aut^0(V))$ acts identically on $l_t$. 
The latter is true for any $t\in\PP^1$ due to the rigidity of the reductive 
group actions. Since $l_t$ is a ruling of the cone $S_{1,t}$, then also 
$\z(\Aut^0(V))$ acts identically on any ruling of $S_{1,t}$. Hence it acts 
identically on the cone $S_{1,t}$ for any $t\in\PP^1$. By (b) one has $A_2 = 
\bigcup_{t\in\PP^1} S_{1,t}$. It follows that $\z(\Aut^0(V))$ acts identically 
on $S_2\subset A_2$. However, $\z(\Aut^0(V))$ acts via homotheties in 
$V\setminus A_1\cong \CC^4$ with the unique fixed point $v_2=v(S_2)$. This gives 
a contradiction.

Hence each ruling $l_t$ of $D$ is $\z(\Aut^0(V))$-invariant and contains just 
two $\z(\Aut^0(V))$-fixed points $l_t\cap (\Gamma_1\cup\Gamma_2)$. Now the 
absence of fixed points of $\Aut^0(V)\cong \GL_2(\CC)$ both in $D$ and in 
$V\setminus (D \cup \{v(S_1),\, v(S_2)\})$ follows. The second assertion is now 
straightforward. 

\xref{prop:vertices-g} is also straightforward by virtue of~\xref{prop:vertices-c} 
and Lemma~\xref{lem:cubic-cones}.

\xref{prop:vertices-e} Notice first of all that any affine chart in 
\eqref{eq:affine-charts} 
is isomorphic indeed to $\CC^4$ by Corollary \xref{rem:interrompu}. Suppose to 
the contrary that there is a point $P\in V$ which is not 
covered by any of the affine charts in \eqref{eq:affine-charts}. Thus, all the 
hyperplane sections $A_i$, $A_{i,t}$ pass through $P$. In particular, $P\in 
D=(A_1\cap A_2)_{\rm red}$, see \xref{prop:vertices-f}, and so, $$P\in 
B:=D\cap\bigcap_{i=1,2, \,t\in\PP^1} A_{i,t}\,.$$
Let us show that $B$ should contain a point of $\Gamma_1\cup\Gamma_2$. Indeed, 
$D\setminus (\Gamma_1\cup\Gamma_2)$ is an orbit of $\Aut^0(V)$, see the proof of 
\xref{prop:vertices-a}. If this orbit contains a point $P\in B$ then $B$, being 
$\Aut^0(V)$-invariant, contains the whole orbit $D\setminus 
(\Gamma_1\cup\Gamma_2)$. Thus, also $B\supset D$ since $B$ is closed. 

Therefore, one may assume that, say, $P\in\Gamma_1\subset B$. 
By \xref{prop:vertices-c} the curve $\Gamma_1$ is filled in by the vertices of 
the cones $S_{1,t}$. We claim that for any $t_1\neq t_2$ the cones 
$S_{1,t_1}\in\SSS_1(V)$ and $S_{2,t_2}\in\SSS_2(V)$ are disjoint. Indeed, by 
Proposition \ref{lem:cohomology} one has $S_{1,t_1}\cdot S_{2,t_2}=0$. Since 
these cones have no common ruling, our claim follows by Corollary 
\ref{cor:intersection-of-cones}.

Let us show that the vertex $v_{1,t_1}$ of $S_{1,t_1}$ cannot lie on $A_{2,t_2}$, 
hence also in $B$, which gives a desired contradiction. Indeed, otherwise there 
is a line through $v_{1,t_1}$ meeting $S_{2,t_2}$. The only lines through 
$v_{1,t_1}$ are the rulings of $S_{1,t_1}$. However, as we have seen, the latter 
rulings are dijoint with $S_{2,t_2}$. 
\end{proof}

\section{Flexibility of affine cones over Fano-Mukai fourfolds $V_{18}$}
\label{sec:aut-aff-cones}

\begin{mdefinition}
\label{def:aut-aff-cones}
An affine variety $X$ of dimension at least $2$ is called \emph{flexible} (in 
the sense of (\cite{AFKKZ}) if the subgroup 
$\operatorname{SAut}(X)\subset\Aut(X)$ generated by all the unipotent algebraic 
subgroups of $\Aut(X)$ acts transitively on the smooth locus 
$X_{\operatorname{reg}}$. We say that $X$ is \emph{flexible in codimension one} 
if $\operatorname{SAut}(X)$ admits an open orbit $O_X$ whose complement has 
codimension at least $2$ in $X$. In the latter case $\operatorname{SAut}(X)$ 
acts $m$-transitively on $O_X$ for any natural $m$ (\cite[Theorem~2.2]{AFKKZ}).

Actions of unipotent groups on
affine cones over Fano varieties was recently a subject of intensive studies.
Several families of smooth Fano fourfolds were examined from this viewpoint;
see, e.g.,~\cite{Prokhorov-Zaidenberg-2015},~\cite{Prokhorov-Zaidenberg-4-Fano}
and the references therein. Using a criterion in
\cite{Kishimoto-Prokhorov-Zaidenberg-criterion}, one can deduce from
\cite{Prokhorov-Zaidenberg-2015} and Theorem~\xref{thm:main} such a corollary.
\end{mdefinition}

\begin{scorollary}
Let $V=V_{18}\subset\PP^{12}$ be a Fano-Mukai fourfold of genus $10$, and let
$A$ be a hyperplane section of $V$ containing either a smooth cubic scroll, or a
cubic cone. Then $V\setminus A$ contains a principal open cylinder $U\cong
Z\times\CC$, where $\dim Z=3$. Consequently, any affine cone $X=\Cone_H(V)$,
where $H$ is an ample polarization of $V$, admits an effective
$\mathbb{G}_{\aaa}$-action.
\end{scorollary}

Flexibility of affine cones over Fano varieties was studied recently 
in~\cite{Michalek-Perepechko-Suss-2016}. Theorem~1.4 and Remark 1.2 in [\emph{ibid}] 
lead to the following criterion.

\begin{proposition}
\label{lem:flexibility-of-cones}
Let $V$ be a projective variety with a very ample polarization $H$. Suppose that 
$V$ contains a Zariski open subset $U$ whose complement $V\setminus U$ has 
codimension at least $2$ in $V$ and such that $U=\bigcup_\alpha U_\alpha$ is 
covered by a collection $(U_\alpha)$ of smooth, flexible principal open subsets 
$U_\alpha=V\setminus H_\alpha$ where $H_\alpha\in |H|$. Then the corresponding 
affine cone $X=\Cone _H(V)$ is flexible in codimension one. More precisely, the 
pullback of $U$ in $X$ is contained in the open orbit of 
$\operatorname{SAut}(X)$. In particular, if $U=V$ then $X$ is flexible. 
\end{proposition}

Applying this criterion in the setting of Theorem~\xref{thm:main} one obtains the 
following result.

\begin{theorem}
\label{thm:flexible-cones}
The affine cone over any polarized Fano-Mukai fourfold $V=V_{18}$ of genus $10$ 
is flexible in codimension one. For $V=V^{\mathrm s}_{18}$ this cone is 
flexible. 
\end{theorem}

\begin{proof}
Let $A_i$ be the hyperplane sections of $V$ with $\Sing(A_i)=S_i$, 
$i=1,2$, where $S_1, S_2$ are two distinct cubic cones in $V$, see 
Theorem~\xref{thm:main}. Then the principal open subsets $U_i:=V\setminus 
A_i\cong\CC^4$, $i=1,2$, are flexible, and $\codim_V (V\setminus (U_1\cup 
U_2))=\codim_V (A_1\cap A_2)=2$. Since $\Pic(V)\cong\ZZ$ the criterion of 
Proposition~\xref{lem:flexibility-of-cones} applies and gives the result.

If $V=V^{\mathrm s}_{18}$ then by Theorem \ref{thm:main-aut}\xref{prop:vertices-e}, $V$ is 
covered by the affine charts isomorphic to $\CC^4$. Once again, the flexibility of the cone 
$\Cone _H(V)$ follows from Proposition~\xref{lem:flexibility-of-cones}.
\end{proof}

\section{Final remarks and open questions}
\label{sec:open-problems}

\begin{mdefinition}
\label{sit:citations}
Given a simple affine algebraic group $G$ over $\CC$ of adjoint type and a 
parabolic subgroup $P\subset G$, consider the flag variety $G/P$. It is known 
(see~\cite{Demazure1977},~\cite[Theorem~2 in \S~3.3 and the subsequent 
remark]{Akhiezer1995}) that $\Aut^0(G/P)\cong G$ except in certain three cases. 
In the exceptional cases $G':=\Aut^0(G/P)$ is again a simple affine algebraic group of adjoint type, 
and $G/P=G'/P'$ for a parabolic subgroup $P'\subset G'$.
In particular, for $G=\G$ and $G/P=\Omega$ as in Section \ref{appendix} one has $\Aut^0(G/P)\cong\G$. A 
similar phenomenon might occur as well for smooth hyperplane sections of adjoint 
varieties.
\end{mdefinition}

\begin{problem}
\label{prob:stabilizers} 
Consider a flag variety $G/P$, where $G$ is a semisimple affine algebraic 
group with trivial center and $P\subset G$ is a parabolic subgroup. Choose $G$ 
and $P$ suitable so that $\Aut^0(G/P)\cong G$. Let $\iota\colon 
G/P\hookrightarrow\PP^n$ be a $G$-equivariant embedding with linearly 
nondegenerate image. Since $\Pic(G/P)$ is discrete, one may identify 
$G\cong\Aut^0(G/P)$ with $\Aut^0(\PP^n,\iota(G/P))$.
We wonder as to when for any smooth hyperplane section $H$ of $\iota(G/P)$ one 
has $\Aut^0(H)=\Stab^0_G(H)$.
\end{problem}

Theorem~\xref{thm:G/P} says that this is indeed the case for $G=\G$ and for the 
adjoint orbit $\Omega=\G/P$. By \cite[Lem.\ 8.2(i)]{Fu-Hwang2016} this is 
the case for any irreducible hermitian symmetric space $G/P$ of compact type, 
provided the embedding $\iota\colon G/P\hookrightarrow\PP^n$ is given by the 
ample generator of the Picard group $\Pic(G/P)\cong\ZZ$. See \cite[Prop.\ 8.4 
and Thm.\ 8.5]{Fu-Hwang2016} for concrete examples. 

One can ask the same question more generally for the smooth linear sections $L$ 
of $\iota(G/P)$ provided the Picard group of $L$ is isomorphic to $\ZZ$. The 
answer is known to be affirmative for \emph{general} linear sections of codimension $l \le N - 2$ 
and $l=3$ in $N=4$ of the 
Grassmannians of lines in $\PP^N$, $N\ge 4$, see \cite[Thm.\ 
1.2, Cor.\ 1.3 and its proof]{Piontkowski-Van-de-Ven-1999}. 

\begin{question}
It is known that for any compactification $(V,A)$ of $\CC^3$ with $\bb_2(V)=1$ 
the middle Betti number $\bb_3(V)$ vanishes. In all known examples in dimension 
$4$, that is, for $\PP^4$, $Q^4\subset\PP^5$, the examples in 
\cite{Prokhorov1994}, and the ones in Theorem~\xref{thm:main}, the middle Betti number satisfies the inequality 
$\bb_4(V)\le 2$. \emph{We wonder whether this 
inequality still holds for any compactification $(V,A)$ of $\CC^4$ with 
$\bb_2(V)=1$.}
\end{question}

The next problem arises naturally regarding Theorem~\xref{thm:main-aut-n}, 
cf.\ Remark~\xref{lem:Aut-sing-DP}.

\begin{problem}
Describe explicitly the involution acting on $V=V^{\s}_{18}$ 
\textup{(}$V=V^{\aaa}_{18}$, respectively\textup{)} and interchanging the 
$\Aut^0(V)$-invariant cubic cones $S_1,\, S_2\subset V$. Determine the discrete 
groups $\Aut(V)/\Aut^0(V)$ for the Fano-Mukai fourfolds $V\not\cong 
V^{\s}_{18}, 
V^{\aaa}_{18}$ of genus $10$. 
\end{problem}

\begin{remark}
The group $\Aut(V^{\aaa}_{18})\cong (\Ga\times\Gm)\rtimes(\ZZ/2\ZZ)$ being 
non-abelian, the involution of $V^{\aaa}_{18}$ interchanging the 
$\Aut^0(V)$-invariant cubic cones $S_1,\, S_2$ does not admit an extension to an 
element of $\G$ acting on $\Omega$, see Propositions 
\ref{prop:regular-centralizers}(b) and \ref{prop:orbits}(b). However, we ignore 
if it can be extended to an automorphism of $\Omega$.
\end{remark}

\begin{problem}
Which Fano-Mukai fourfolds $V_{18}$ of genus $10$ admit a
K\"ahler-Einstein metric \textup{(\cite{Tian-2015})}? 
\emph{Notice that 
the group $\Aut^0 V^{\aaa}_{18}\cong\Ga\times\Gm$ is not reductive, 
see Theorem~\xref{thm:main-aut-n}\xref{thm:main-aut-n-Ga}. Hence, 
according to Matsushima's theorem \textup{(\cite{Matsushima-1957})},
the variety $V^{\aaa}_{18}$ does not admit such a metric.}
\end{problem}

\begin{problem}
Study \emph{singular} Fano-Mukai fourfolds of type $V_{18}$ 
\textup(cf.~\cite{Prokhorov-planes},~\cite{Prokhorov-v22}\textup).
\end{problem}

\par\medskip\noindent
\textbf{Acknowledgments.}
The paper started during the first author's stay at the Institute Fourier, 
Grenoble,
in June of 2016. He thanks the institute for its hospitality.
The authors are grateful to Alexander Kuznetsov for useful
discussions, to Michel Brion, Jun-Muk Hwang, and Laurent Manivel for important remarks 
around the material of Sect.~\xref{appendix}, 
and to Ivan Arzhantsev and Alexander Perepechko for a pertinent remark concerning the material of Section \ref{sec:aut-aff-cones}. 
Our thanks are due also to a referee for his remarks improving the style of the paper.


\newcommand{\etalchar}[1]{$^{#1}$}
\def\cprime{$'$}

\end{document}